\documentclass[11pt]{article}
\usepackage[utf8]{inputenc}
\usepackage{graphicx}
\usepackage{datetime}
\usepackage{tikz-cd}
\usetikzlibrary{calc, decorations.markings, decorations.pathreplacing}
\usetikzlibrary{arrows.meta} 
\usepackage{forest}
\usepackage{float}

\usepackage{amssymb}
\usepackage{mathtools}

\usepackage{capt-of}
\usepackage{wrapfig}
\usepackage{printlen}

\usepackage{enumitem}

\usepackage[only,llbracket,rrbracket]{stmaryrd}

\usepackage{amsmath,amssymb,amsfonts,mathrsfs,amsthm}
\usepackage{geometry}
\usepackage{booktabs}

\usepackage[table]{xcolor}

\usepackage[ruled,vlined]{algorithm2e}

\usetikzlibrary{positioning}
 
\usepackage{hyperref}
\hypersetup{
   colorlinks,breaklinks,
            urlcolor=[RGB]{13, 152, 186},
            linkcolor=[RGB]{13, 152, 186},
            citecolor=[RGB]{13, 152, 186}}

\usepackage{subfiles}
\theoremstyle{plain}
\renewcommand{\itshape}{\slshape}

\newtheorem{theorem}{\sc Theorem}[section]
\newtheorem{corollary}[theorem]{\sc Corollary}
\newtheorem{lemma}[theorem]{\sc Lemma}
\newtheorem{proposition}[theorem]{\sc Proposition}
\newtheorem{claim}{\sc Claim}

\theoremstyle{definition}
\newtheorem{definition}[theorem]{\sc Definition}
\newtheorem{example}[theorem]{\sc Example}
\newtheorem{remark}[theorem]{\sc Remark}

\DeclareMathOperator{\Hom}{Hom}

\DeclareMathOperator{\im}{im}

\let\phi\varphi
\let\epsilon\varepsilon

\definecolor{teal}{RGB}{0, 200, 200} 
\definecolor{lightblue}{RGB}{0, 100, 100}
\definecolor{amber}{RGB}{230, 127, 3}
\definecolor{myblue}{HTML}{fb5607}
\definecolor{red}{RGB}{200,0,0}

\newcommand{\field}{k}

\usepackage{titling} 

\newcommand{\ldeg}{\deg^l}
\newcommand{\rdeg}{\deg^r}

\newcommand{\lvl}{\mathrm{level}}

\newcommand{\V}{\mathbb{V}}

\newcommand{\T}{\mathcal{T}}

\newcommand{\M}{\mathcal{M}}
\newcommand{\Ss}{\mathcal{S}}

\newcommand{\Sb}{\mathbb{S}}
\newcommand{\mylim}{\varprojlim}
\newcommand{\mycolim}{\varinjlim}

\DeclareMathOperator{\merge}{merge}

\usepackage{tcolorbox}
\tikzset{
     base_style/.style={
        thick,
        shorten >=-0.1pt,
        shorten <=-0.1pt,
      },
  f_arrow/.style={
    base_style,
    postaction={
      decorate,
      decoration={
        markings,
        mark=at position 0.6 with {\arrow{stealth}}
      }
    }
  },
  b_arrow/.style={
    base_style,
    postaction={
      decorate,
      decoration={
        markings,
        mark=at position 0.4 with {\arrowreversed{stealth}}
      }
    }
  },
   my_arrow/.style n args={2}{
        base_style,
        postaction={
            decorate,
            decoration={
                markings,
                mark=at position #1 with {\arrow[scale=#2]{stealth}}
            }
        }
    }
}

\forestset{
  splittin_up/.style={
    for tree={
      calign=first,
      s sep=\fpeval{0.5 * \ssepvalue}
    }
  }
}

\pgfkeys{
 /mynode/.is family, /mynode,
 pfill/.initial=none,      
 isep/.initial=2pt,        
 name/.initial=,           
 fill/.initial=none,       
 x/.initial=0,             
 y/.initial=0,             
 xs/.initial=1,            
 ys/.initial=1,            
 lpos/.initial=above,      
 ltxt/.initial=,           
}

\NewDocumentCommand{\mynode}{ o m }{%
  \pgfkeys{/mynode, #2}
  \node[
    circle, draw,
    inner sep=\pgfkeysvalueof{/mynode/isep},
    name=\pgfkeysvalueof{/mynode/name},
    fill=\pgfkeysvalueof{/mynode/fill}, 
    label={\pgfkeysvalueof{/mynode/lpos}:\pgfkeysvalueof{/mynode/ltxt}},
    path picture={
      \IfValueT{#1}{
        \fill[#1] 
          (path picture bounding box.south) 
          rectangle 
          (path picture bounding box.north east);
      }
    }
  ] at (\fpeval{\pgfkeysvalueof{/mynode/x} * \pgfkeysvalueof{/mynode/xs}}, 
        \fpeval{\pgfkeysvalueof{/mynode/y} * \pgfkeysvalueof{/mynode/ys}}) {};
}

\tikzset{
     base_style/.style={
        thick,
        shorten >=-0.1pt,
        shorten <=-0.1pt,
      },
  f_arrow/.style={
    base_style,
    postaction={
      decorate,
      decoration={
      markings,
      mark=at position 0.7 with {\arrow[scale=1.3]{stealth}}
    }
    }
  },
  b_arrow/.style={
    base_style,
    postaction={
      decorate,
      decoration={
        markings,
        mark=at position 0.4 with {\arrowreversed[scale=1.3]{stealth}}
      }
    }
  },
   short_style/.style={
    thick,
    shorten >=75pt,
    shorten <=75pt,
      postaction={
            decorate,
            decoration={
                markings,
                mark=at position 0.63 with {\arrow[scale=1.6]{stealth}}
            }
        }
  },
   my_arrow/.style n args={2}{
        base_style,
        postaction={
            decorate,
            decoration={
                markings,
                mark=at position #1 with {\arrow[scale=#2]{stealth}}
            }
        }
    }
}

\definecolor{beige}{RGB}{255, 240, 210}
\definecolor{my_red}{RGB}{104, 36, 109}
\definecolor{my_teal}{RGB}{127, 185, 183}
\definecolor{my_grey}{RGB}{115, 110, 104}

\definecolor{my_blue}{RGB}{100,143,255}
\definecolor{my_pink}{RGB}{220,38,127}
\definecolor{my_gold}{RGB}{255,176,0}
\definecolor{my_violet}{RGB}{150, 120, 200}

\NewDocumentCommand{\myynode}{mmmmmmmmO{none}}{\node[ circle, draw, inner sep=3, name=#2, 
fill=#3 ] at (\fpeval{#5 * #7},\fpeval{#6 * #8}) {}; }

\definecolor{my_actual_red}{RGB}{240, 117, 105}

\definecolor{my_mustard}{RGB}{248, 217, 70}

\author{\Large David Lanners\thanks{\texttt{david.lanners@durham.ac.uk}} \\ \large Durham University}

\newcommand{\htest}[1]{\hyperref[#1]{\ref*{#1}}}

\usepackage[
backend=biber,
style=numeric,
maxbibnames=99,
sorting=none
]{biblatex}
\makeatletter
\def\moverlay{\mathpalette\mov@rlay}
\def\mov@rlay#1#2{\leavevmode\vtop{%
   \baselineskip\z@skip \lineskiplimit-\maxdimen
   \ialign{\hfil$\m@th#1##$\hfil\cr#2\crcr}}}
\newcommand{\charfusion}[3][\mathord]{
    #1{\ifx#1\mathop\vphantom{#2}\fi
        \mathpalette\mov@rlay{#2\cr#3}
      }
    \ifx#1\mathop\expandafter\displaylimits\fi}
\makeatother

\newcommand{\cat}[1]{\mathbf{#1}}

\newcommand{\expspacing}{\mkern 2mu}

\usepackage{musicography}
\usepackage{xpatch}

\makeatletter
\renewcommand{\itshape}{\slshape}
\makeatother
\makeatletter
\renewcommand\th@plain{\slshape}
\xpatchcmd{\proof}{\itshape}{\slshape}{}{}
\makeatother

\usepackage{graphicx}
\usepackage{subcaption}
\usepackage[justification=raggedright,singlelinecheck=false]{caption}

\usepackage{cancel}

 \date{}

\usepackage{titlesec}
\titlelabel{\thetitle.\quad}

\title{\vspace{-3.0cm} \huge From Frames to Features: Scalable Zigzag\\ Persistence for Binary Video}

\usepackage{adjustbox}

\usepackage{parskip}
\makeatletter
\def\thm@space@setup{%
  \thm@preskip=\parskip \thm@postskip=0pt
}
\makeatother

\newlength{\myDisplaySkip}
\expandafter\def\expandafter\normalsize\expandafter{%
    \normalsize%
    \setlength{\abovedisplayskip}{\myDisplaySkip}%
    \setlength{\belowdisplayskip}{\myDisplaySkip}%
    \setlength{\abovedisplayshortskip}{0pt}
    \setlength{\belowdisplayshortskip}{\myDisplaySkip}%
}

\setlist[enumerate]{
  itemsep=2pt,
  parsep=2pt,
  topsep=0pt
}

\setlist[itemize]{
  itemsep=2pt,
  parsep=2pt,
  topsep=0pt
}

\begin{document}
\maketitle











\begin{abstract}
Zigzag persistence tracks topological features in spatio-temporal data through combinatorial invariants called barcodes. For binary videos, existing methods are bottlenecked by the construction of prohibitively large cubical complexes and performing Gaussian elimination on large boundary matrices, rendering high-resolution videos out of reach. We show that the $H_0$ and $H_1$ barcodes can be extracted directly from connected-component dynamics. By encoding these dynamics in a graph, we bypass cubical complexes entirely and are able to leverage the near-linear time barcode decomposition algorithm by Dey and Hou~\cite{deyComputingZigzagPersistence2021}, leading to significant speedups. The total runtime of our pipeline is dominated by the construction of the underlying graph structures, which scales linearly with pixel count and is embarrassingly parallel across frames, ensuring excellent scalability. We demonstrate how this approach enables zigzag persistence on 4k video at real-time rates on consumer hardware.
\end{abstract}

\section{Introduction}

Binary videos (sequences of black-and-white frames) arise naturally in many settings where data can be segmented into foreground and background regions. They broadly fall into two categories: temporal evolution, such as cell tracking~\cite{odaPersistentHomologicalCell2023, garcia-redondoFastTopologicalSignal2024}, neural activity imaging~\cite{billingsSimplicialTopologicalDescriptions2021}, or dynamic pattern formation (e.g., reaction-diffusion systems); and analysis of volumetric data via sequential 2D slicing (e.g., sequential slices of porous materials~\cite{songGeneralizedMorseTheory2025} or CT and electron microscopy data~\cite{francoisTrainFreeSegmentationMRI2025, tanabeHomologicalApproachMathematical2021, wangSemiautomaticMethodExtracting2024}). In both settings, capturing the evolution of topological features presents a core challenge: quantifying how connected components and loops appear, split, merge, and disappear across the frames constitutes a hard problem.

Topological Data Analysis (TDA) provides many tools to quantify changing topology~\cite{tanweerTopologicalFrameworkIdentifying2024, munchApplicationsPersistentHomology2013a, oesterlingComputingVisualizingTimeVarying2017a, pereaSlidingWindowsPersistence2015, kimSpatiotemporalPersistentHomology2021, edelsbrunnerTimevaryingReebGraphs2008, cohen-steinerVinesVineyardsUpdating2006, hajijVisualDetectionStructural2018}; however, a naïve application of its most common variant, persistent homology~\cite{edelsbrunnerTopologicalPersistenceSimplification2002, otterRoadmapComputationPersistent2017, zomorodianComputingPersistentHomology2005}, is not well-suited to this setting. It is often constructed from a filtration (a nested sequence of spaces), whereas the sequence of frames of a video does not generally represent a filtered space.

Zigzag persistence~\cite{carlssonZigzagPersistence2010} adapts persistent homology to sequences of spaces with both forward and backward inclusions, and thus provides a natural setting for studying topological changes in non-monotone sequences. For binary videos, such a sequence can be constructed by inserting suitable intermediate comparison frames~\cite{mcdonaldZigzagPersistenceCoral2023}, consisting of either the union or the intersection of consecutive frames. In recent years, zigzag-based methods have been applied to a wide range of time-varying systems, including the analysis of dynamic and temporal graphs~\cite{kimAnalysisDynamicGraphs2020, myersTemporalNetworkAnalysis2023, myersTopologicalAnalysisTemporal2023}, the study of combinatorial dynamical systems~\cite{deyTrackingDynamicalFeatures2022}, and the evolution of biological and ecological patterns~\cite{glennTrackingTemporalEvolution2025, mcdonaldZigzagPersistenceCoral2023, yangTopologicalClassificationTumourimmune2025}.

Despite its conceptual appeal for spatiotemporal problems, zigzag persistence poses significant computational challenges. Existing approaches to binary videos typically rely on cubical complexes~\cite{mcdonaldZigzagPersistenceCoral2023, divasonZigzagPersistenceImage2024}, whose sizes are of order comparable to the number of pixels, leading to prohibitive memory requirements. Even though there has been much progress in speeding up zigzag computations~\cite{milosavljevicZigzagPersistentHomology2011, mariaDiscreteMorseTheory2024, deyFastComputationZigzag2022}, existing methods for binary video present cubic time complexity in the number of cells, rendering such methods impractical and effectively restricting the applicability to very low-resolution data and videos comprising only a small number of frames. As a result, these restrictions have hindered the broader adoption of zigzag persistence into the toolkit of TDA methods.

This work removes the computational bottleneck and makes zigzag persistence a practical tool for binary video data. Our approach departs from cubical complexes and instead reduces the issue to a graph problem. In a first step, we focus on the computation of the $H_0$-zigzag persistence, which tracks the connectivity of the data by monitoring how components appear, merge, and split across frames. By clustering connected components within each frame (which is an embarrassingly parallel task, scaling linearly in the number of pixels), we reduce the problem to a graph setting, where maps between frames are simple: tracking a single representative pixel per cluster suffices. This data can be compactly represented in graphs, called \textit{formigrams}~\cite{MemoliExtractingPersistentClustersinDynamicData}.

\begin{figure}[H]
\centering
\begin{tikzpicture}[xscale=0.5]
    \def\scale{1}
    \def\myheight{4}
    \def\lowerdiagramy{-6} 

    \node[font=\Large, anchor=center, align=center] at (-15, \myheight) {Intermediate \\union frames};
    \node[font=\Large, anchor=center, align=center] at (-15, 0) {Original \\ video};

    \node at (-9,0) {\includegraphics[scale=\scale]{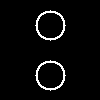}};
    \node at (-6,\myheight) {\includegraphics[scale=\scale]{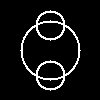}};
    \node at (-3,0) {\includegraphics[scale=\scale]{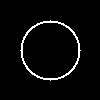}};
    \node at (0,\myheight) {\includegraphics[scale=\scale]{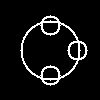}};
    \node at (3,0) {\includegraphics[scale=\scale]{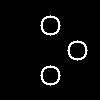}};
    \node at (6,\myheight) {\includegraphics[scale=\scale]{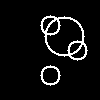}};
    \node at (9,0) {\includegraphics[scale=\scale]{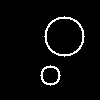}};

    \draw[short_style] (-9,0) to (-6, \myheight);
    \draw[short_style] (-3,0) to (-6, \myheight);
    \draw[short_style] (-3,0) to (0, \myheight);
    \draw[short_style] (3,0) to (0, \myheight);
    \draw[short_style] (3,0) to (6, \myheight);
    \draw[short_style] (9,0) to (6, \myheight);

    \begin{scope}[shift={(-2,\lowerdiagramy)}, xscale=1.3, yscale=0.9] 
        
        \def\xscale{0.9pt}
        \def\yscale{1pt}
        \def\nodesize{1pt}
        \def\shifttwo{-1.5}

        \node[anchor=center, font=\LARGE, align=center] at (-5+\shifttwo, 3) {Formigram};
        
        \myynode{\nodesize}{a_1_1}{white}{}{\fpeval{-9+\shifttwo}}{-1.5}{\xscale}{\yscale}
        \myynode{\nodesize}{a_1_2}{my_teal}{}{\fpeval{-9+\shifttwo}}{1.5}{\xscale}{\yscale}
        \myynode{\nodesize}{b_1_1}{white}{}{\fpeval{-7.5+\shifttwo}}{0}{\xscale}{\yscale}
        \myynode{\nodesize}{c_1_1}{white}{}{\fpeval{-6+\shifttwo}}{0}{\xscale}{\yscale}
        \myynode{\nodesize}{d_1_1}{white}{}{\fpeval{-4.5+\shifttwo}}{0}{\xscale}{\yscale}
        \myynode{\nodesize}{e_1_1}{my_actual_red}{}{\fpeval{-3+\shifttwo}}{1.5}{\xscale}{\yscale}
        \myynode{\nodesize}{e_1_2}{white}{}{\fpeval{-3+\shifttwo}}{0}{\xscale}{\yscale}
        \myynode{\nodesize}{e_1_3}{my_mustard}{}{\fpeval{-3+\shifttwo}}{-1.5}{\xscale}{\yscale}
        \myynode{\nodesize}{f_1_1}{white}{}{\fpeval{-1.5+\shifttwo}}{1.5}{\xscale}{\yscale}
        \myynode{\nodesize}{f_1_2}{my_mustard}{}{\fpeval{-1.5+\shifttwo}}{-1.5}{\xscale}{\yscale}
        \myynode{\nodesize}{g_1_1}{white}{}{\fpeval{0+\shifttwo}}{1.5}{\xscale}{\yscale}
        \myynode{\nodesize}{g_1_2}{my_mustard}{}{\fpeval{0+\shifttwo}}{-1.5}{\xscale}{\yscale}

        \draw[my_grey, dashed] (0,3) -- (0,-3);
        \draw[my_grey] (-12,4) -- (-0.5,4) -- (0,3.7) -- (0.5,4) -- (10.5,4);

        \draw[f_arrow] (a_1_1) to (b_1_1);
        \draw[f_arrow] (a_1_2) to (b_1_1);
        \draw[f_arrow] (c_1_1) to (b_1_1);
        \draw[f_arrow] (c_1_1) to (d_1_1);
        \draw[f_arrow] (e_1_1) to (d_1_1);
        \draw[f_arrow] (e_1_2) to (d_1_1);
        \draw[f_arrow] (e_1_3) to (d_1_1);
        \draw[f_arrow] (e_1_1) to (f_1_1);
        \draw[f_arrow] (e_1_2) to (f_1_1);
        \draw[f_arrow] (e_1_3) to (f_1_2);
        \draw[f_arrow] (g_1_1) to (f_1_1);
        \draw[f_arrow] (g_1_2) to (f_1_2);

        \def\myshift{2}
        \node[anchor=center,align=center, font=\LARGE] at (\fpeval{3+\myshift}, 3) {Barcode};

        \myynode{\nodesize}{ba_1_1}{white}{}{\fpeval{0+\myshift}}{0}{\xscale}{\yscale}
        \myynode{\nodesize}{ba_1_2}{my_teal}{}{\fpeval{0+\myshift}}{0.75}{\xscale}{\yscale}
        \myynode{\nodesize}{bb_1_1}{white}{}{\fpeval{1.5+\myshift}}{0}{\xscale}{\yscale}
        \myynode{\nodesize}{bc_1_1}{white}{}{\fpeval{3+\myshift}}{0}{\xscale}{\yscale}
        \myynode{\nodesize}{bd_1_1}{white}{}{\fpeval{4.5+\myshift}}{0}{\xscale}{\yscale}
        \myynode{\nodesize}{be_1_1}{white}{}{\fpeval{6+\myshift}}{0}{\xscale}{\yscale}
        \myynode{\nodesize}{be_1_2}{my_actual_red}{}{\fpeval{6+\myshift}}{1.5}{\xscale}{\yscale}
        \myynode{\nodesize}{be_1_3}{my_mustard}{}{\fpeval{6+\myshift}}{-1.5}{\xscale}{\yscale}
        \myynode{\nodesize}{bf_1_1}{white}{}{\fpeval{7.5+\myshift}}{0}{\xscale}{\yscale}
        \myynode{\nodesize}{bf_1_2}{my_mustard}{}{\fpeval{7.5+\myshift}}{-1.5}{\xscale}{\yscale}
        \myynode{\nodesize}{bg_1_1}{white}{}{\fpeval{9+\myshift}}{0}{\xscale}{\yscale}
        \myynode{\nodesize}{bg_1_2}{my_mustard}{}{\fpeval{9+\myshift}}{-1.5}{\xscale}{\yscale}

        \draw[base_style] (ba_1_1) to (bb_1_1);
        \draw[base_style] (bb_1_1) to (bc_1_1);
        \draw[base_style] (bc_1_1) to (bd_1_1);
        \draw[base_style] (bd_1_1) to (be_1_1);
        \draw[base_style] (be_1_1) to (bf_1_1);
        \draw[base_style] (be_1_3) to (bf_1_2);
        \draw[base_style] (bf_1_1) to (bg_1_1);
        \draw[base_style] (bf_1_2) to (bg_1_2);
    \end{scope}
\end{tikzpicture}
\captionof{figure}{\textbf{Pipeline overview}. Intermediate frames are generated by taking the union of white regions between original video frames. Connected components in each frame form the vertices of a formigram, while directed edges represent their evolution through intermediate frames. Finally, we compute the corresponding zigzag barcode in near-linear time.}\label{fig:formi-first}
\end{figure}
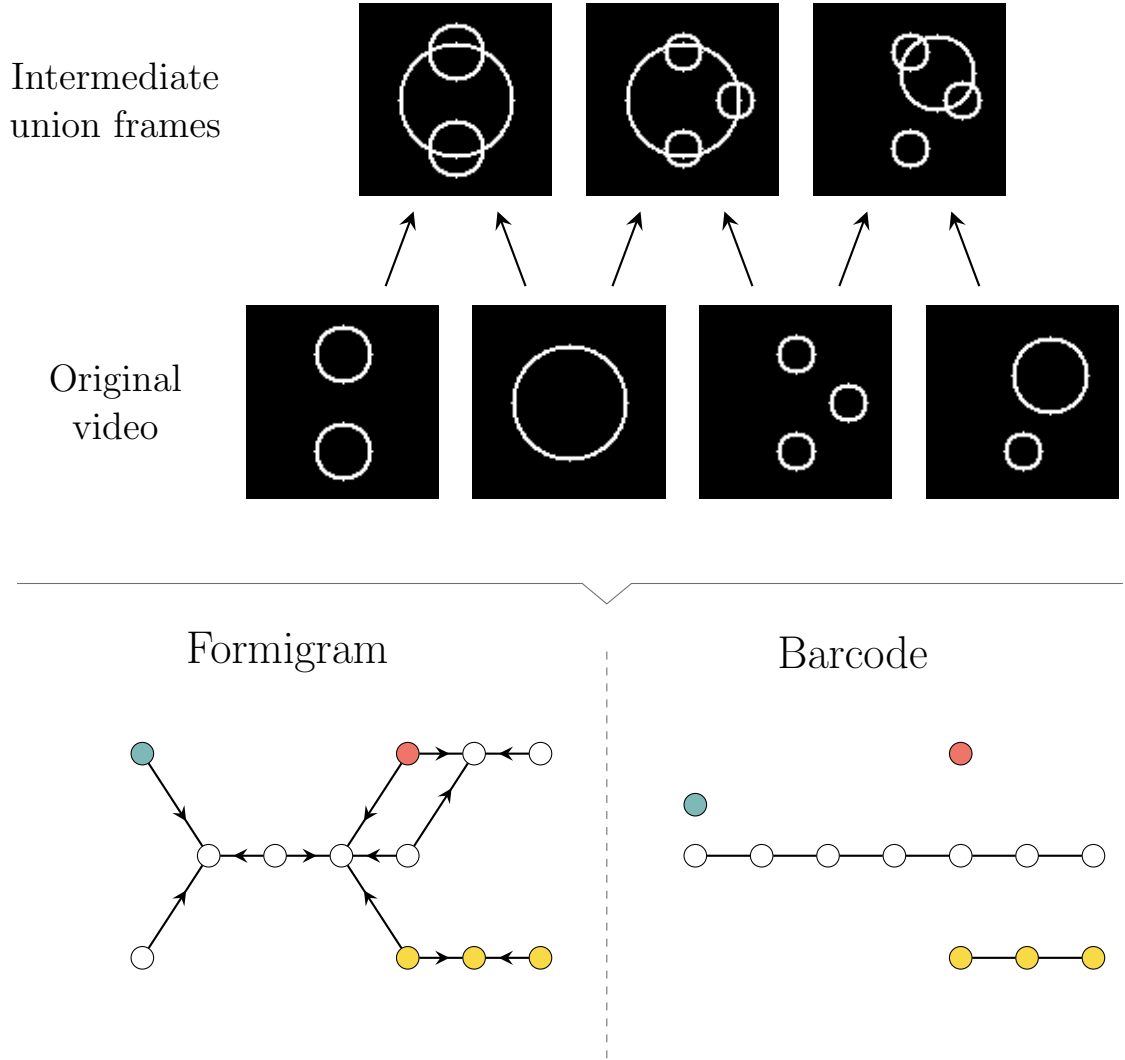

From these graphs, we extract the barcode in near-linear time by building on the work of Dey and Hou~\cite{deyComputingZigzagPersistence2021}. To compute the $H_1$ zigzag persistence, which tracks the topology of cycles, we use Alexander duality~\cite{MR4454782}, allowing the $H_1$-zigzag computation to be reduced to an $H_0$ zigzag computation on the complementary (opposite-colour) complex. In this way, we can compute both $H_0$ and $H_1$ zigzag persistence of a binary video in near-linear time with respect to the number of features. For general discussions on representations of the $H_0$-functor, especially in the context of multiparameter persistence, we refer to~\cite{BAUER2020107171, bauerAdditiveImage0th2025, binduaDecomposingZerodimensionalPersistent2024, brodzkiComplexityZerodimensionalMultiparameter2020}.

\paragraph{Implementation.}
We provide an open-source implementation of the algorithms presented in this paper in the software package \texttt{Zigvid}~\cite{ZigVidSoftwareRef}. The core algorithms are implemented in Rust to achieve high performance, and optional Python bindings are provided to facilitate integration with existing scientific Python workflows.

\paragraph{Outline.}
Section~\ref{sec:foundations} introduces the topology associated with binary videos and recalls the basics of zigzag persistence. In Section~\ref{sec:formigrams-chapter}, we introduce the graph structures underlying $H_0$-zigzag modules and show in Section~\ref{sec:field-indep} that the resulting $H_0$-persistence is independent of the choice of coefficient field. Section~\ref{sec:frames-to-formi} presents a linear-time algorithm, in the number of pixels, for constructing these graph structures from a binary video. In Section~\ref{sec:int-decomp}, we recall the algorithm of Dey and Hou for computing the interval decomposition of $H_0$-zigzag modules. Section~\ref{sec:dualities} introduces two duality principles: a Mayer-Vietoris duality of~\cite{carlssonZigzagPersistence2010} relating the union and intersection constructions, and Alexander duality relating $H_0$- and $H_1$-zigzag persistence. Finally, our findings in Section~\ref{sec:benchmarking} show that by shifting the perspective from cubical complexes to sparse graphs, zigzag persistence becomes a viable tool for real-time and high-resolution video analysis.

\section{Foundations of Persistence for Binary Video}\label{sec:foundations}

 This section establishes the necessary foundations. We begin with the pixel-level topology used to represent foreground and background regions, then introduce the algebraic framework of zigzag modules used to track features across non-monotonic sequences of spaces. Finally, we discuss the computational challenges that standard methods face when confronted with high-resolution video data. 

\subsection{The Topology of Binary Images}\label{sec:top-bin-imag}
Let $\Omega_{w,h}=\{1,\dots, w\}\times\{1,\dots,h\}\subset\mathbb{Z}^2$ be a finite grid. Consider a map
\[
I : \Omega_{w,h} \to \{ 0, 1\},
\]
\noindent
\adjustbox{valign=t}{%
\begin{minipage}{0.68\textwidth}
where each point of $\Omega_{w,h}$ corresponds to a \emph{pixel}. We interpret the set $I^{-1}(1)$ as \emph{foreground pixels} and $I^{-1}(0)$ as \emph{background pixels}. Because the grid $\Omega_{w,h}$ is discrete, modelling the topology of these pixel sets requires an explicit choice of pixel adjacency. Two conventions are common:
\emph{4-connectivity}, where only vertical and horizontal neighbours are adjacent, and
\emph{8-connectivity}, where diagonal neighbours are also included. The
difference is illustrated to the right: solid edges represent 4-connectivity, while the dashed edges indicate the additional diagonal links present in 8-connectivity.
\end{minipage}}%
\hfill
\adjustbox{valign=t}{%
\begin{minipage}{0.28\textwidth}
\centering
\vspace{1.1em} 
\begin{tikzpicture}[scale=1.7]
    \coordinate (Center) at (0,0);

    \foreach \x in {-1,0,1} {
        \foreach \y in {-1,0,1} {
            \ifnum\x=0 \ifnum\y=0 \else
                \coordinate (P\x\y) at (\x,\y);
            \fi \else
                \coordinate (P\x\y) at (\x,\y);
            \fi
        }
    }

    \draw[thick] (Center) -- (0,1);
    \draw[thick] (Center) -- (0,-1);
    \draw[thick] (Center) -- (1,0);
    \draw[thick] (Center) -- (-1,0);

    \draw[thick,dashed] (Center) -- (1,1);
    \draw[thick,dashed] (Center) -- (-1,1);
    \draw[thick,dashed] (Center) -- (1,-1);
    \draw[thick,dashed] (Center) -- (-1,-1);

    \foreach \x in {-1,0,1}{
        \foreach \y in {-1,0,1}{
            \ifnum\x=0 \ifnum\y=0 \else
                \fill (P\x\y) circle (2.5pt);
            \fi \else
                \fill (P\x\y) circle (2.5pt);
            \fi
        }
    }

    \fill[red!20] (Center) circle (4.5pt);
    \fill[red] (Center) circle (2.5pt);
    \draw[red, thick] (Center) circle (4.5pt);
\end{tikzpicture}
\end{minipage}}

These conventions lead to different notions of connectedness: two diagonally positioned foreground pixels form a single component under 8-connectivity, but remain disconnected under 4-connectivity. It is a classical result~\cite{KOVALEVSKY, KongDigitalTopology} that in order to avoid connectivity paradoxes, such as the failure of the discrete analogue of the Jordan Curve Theorem, one needs to assign complementary adjacencies, i.e.\ (8,4) or (4,8), to the foreground and background. This ensures that the boundary of one colour partitions the other into distinct interior and exterior regions. We incorporate this need for complementary connectivities into our definition of binary images.

\begin{definition}
A \emph{binary image} is a map $I:\Omega_{w,h}\to \{0,1\}$ together with a choice of \emph{foreground connectivity} $\kappa_I \in \{4,8\}$. The complementary \emph{background connectivity} is defined by $\bar{\kappa}_I = 12 - \kappa_I$.
\end{definition}

In the realm of TDA~\cite{RobinsMorseFromGray, MR4454782, garinDualityPersistentHomology2020}, $4$-connectivity and $8$-connectivity are realised via two cubical complex constructions that can be associated with a binary image $I:\Omega_{w,h}\to\{ 0,1\}$. We recall these constructions briefly. Let $K\subset\mathbb{R}^2$ be the cubical complex generated by the 2-cells
\begin{equation} \label{eq:primal-cubes}
Q_{i,j}\coloneqq\left[i-\frac{1}{2}, i+\frac{1}{2}\right]\times\left[j-\frac{1}{2},j+\frac{1}{2}\right], \ \ \  \ (i,j)\in \Omega_{w,h}.\end{equation}
This definition centres each 2-cell of $K$ at the integer pixel coordinate $(i,j)\in \Omega_{w,h}$, naturally assigning the pixel value $I(i,j)$ to each top-dimensional cell $Q_{i,j}$. Alternatively, one can consider the dual complex $K^*$, whose vertices coincide with $\Omega_{w,h}$, and assign the pixel value $I(i,j)$ to each vertex $(i,j)$ of $K^*$. These dual interpretations induce two distinct subcomplexes of $K$ and $K^*$, respectively, modelling the two pixel adjacencies for a fixed colour $c\in\{0,1\}$:
\begin{itemize}[leftmargin=*]
\item \textbf{T-construction (8-connectivity)}: Defines the subcomplex
\[
T_I(c)\coloneqq \bigcup_{(i,j)\in I^{-1}(c)} \mathrm{cl}(Q_{i,j}) \subseteq K,
\]
assigning a 2-cell and its closure to each pixel of colour $c$. Diagonally adjacent pixels produce squares intersecting in one vertex, inducing 8-connectivity. 

\item \textbf{V-construction (4-connectivity)}: Defines the subcomplex
\[
V_I(c) \coloneqq \{ \sigma\in K^* \ | \ \text{$\sigma^{(0)}\subseteq I^{-1}(c)$}  \} \subseteq K^*,
\]
where $\sigma^{(0)}$ denotes the set of vertices of a cell $\sigma\in K^*$. The $V$-construction assigns a vertex to each $c$-coloured pixel, together with all edges and 2-cells, whose vertices fully lie in $I^{-1}(c)$. Since diagonal pixels do not share an edge in $K^*$, the V-construction induces 4-connectivity.
\end{itemize}

\begin{center}
\begin{tikzpicture}[scale=1.2]

    \begin{scope}[xshift=-4cm, yshift=1.5cm]
    \node at (0,1.5) {\Large $I$ };
        \node at (0,0) {\scalebox{1.5}{$\begin{bmatrix} 0 & 1 & 0 \\ 1 & 0 & 1 \\ 0 & 1 & 0 \end{bmatrix}$}};
    \end{scope}

    \begin{scope}[xshift=-1.5cm, ]
        \node at (1.5, 3.6) {\Large $K$};
         \node at (1.5, -0.5) {\color{red}{\Large $T_I(1)$}};
        
        \fill[red!20] (1,2) rectangle (2,3); 
        \fill[red!20] (0,1) rectangle (1,2); 
        \fill[red!20] (2,1) rectangle (3,2); 
        \fill[red!20] (1,0) rectangle (2,1); 
        

        \draw[thick, red] (1,0)--(1,3);
        \draw[thick, red] (2,0)--(2,3);
        \draw[thick, red] (0,1)--(3,1);
        \draw[thick, red] (0,2)--(3,2);

        \draw[thick, red] (1,0)--(2,0);
        \draw[thick, red] (0,1)--(0,2);
        \draw[thick, red] (1,3)--(2,3);
        \draw[thick, red] (3,1)--(3,2);

        \draw[thick, black] (0,0)--(0,1);
        \draw[thick, black] (0,0)--(1,0);
        \draw[thick, black] (0,3)--(0,2);
        \draw[thick, black] (0,3)--(1,3);
        \draw[thick, black] (3,3)--(3,2);
        \draw[thick, black] (3,3)--(2,3);
        \draw[thick, black] (3,0)--(3,1);
        \draw[thick, black] (3,0)--(2,0);

        \fill[red] (1,0) circle (2pt);
        \fill[red] (1,1) circle (2pt);
        \fill[red] (2,0) circle (2pt);
        \fill[red] (2,1) circle (2pt);

        \fill[red] (1,2) circle (2pt);
        \fill[red] (2,2) circle (2pt);
        \fill[red] (1,3) circle (2pt);
        \fill[red] (2,3) circle (2pt);

        \fill[red] (0,1) circle (2pt);
        \fill[red] (0,2) circle (2pt);
        \fill[red] (3,1) circle (2pt);
        \fill[red] (3,2) circle (2pt);
        
        \node[black] at (0.5, 2.5) {0}; \node[red] at (1.5, 2.5) {1}; \node[black] at (2.5, 2.5) {0};
        \node[red] at (0.5, 1.5) {1}; \node[black] at (1.5, 1.5) {0}; \node[red] at (2.5, 1.5) {1};
        \node[black] at (0.5, 0.5) {0}; \node[red] at (1.5, 0.5) {1}; \node[black] at (2.5, 0.5) {0};
    \end{scope}

    \begin{scope}[xshift=3cm, yshift=0.5cm]
        \node at (1, 2.75) {\Large $K^*$};
        \node at (1, -0.75) {\color{red}\Large $V_I(1)$};
        \draw[thick] (0,0) grid (2,2);
        \foreach \x in {0,...,2} \foreach \y in {0,...,2} \fill (\x,\y) circle (2pt);

        \fill[red] (1,2) circle (2.1pt); 
        \fill[red] (0,1) circle (2.1pt); 
        \fill[red] (2,1) circle (2.1pt); 
        \fill[red] (1,0) circle (2.1pt); 

        \node[above left] at (0,2) {0}; \node[above left, red] at (1,2) {1}; \node[above left] at (2,2) {0};
        \node[above left,red]  at (0,1) {1}; \node[above left] at (1,1) {0}; \node[above left, red] at (2,1) {1};
        \node[above left] at (0,0) {0}; \node[above left,red] at (1,0) {1}; \node[above left] at (2,0) {0};
    \end{scope}
\end{tikzpicture}
\captionof{figure}{\textbf{Cubical complex constructions induced by binary image $I$}: The $T$-construction on the foreground $T_I(1)\subseteq K$ realises 8-connectivity, while the $V$-construction $V_I(1)\subseteq K^*$ realises 4-connectivity. Note that $T_I(1)$ is homotopy equivalent to a circle, whereas $V_I(1)$ consists of four discrete points.}
\label{fig:2-complexes}
\end{center}

Both the $T$-construction and the $V$-construction will serve as models for the topology of the foreground and background. For a cubical complex $X\subseteq K$ or $X\subseteq K^*$, its geometric realisation is given by
\[
|X| \coloneq \bigcup_{\sigma\in X} \sigma \subseteq \mathbb{R}^2.
\]
To treat the topological spaces arising from both connectivity conventions $(4,8)$ and $(8,4)$ in a notationally uniform way, we introduce the following definition.

\begin{definition} \label{def:top-real}
Let $I$ be a binary image. We define the \textit{foreground} and \textit{background regions} of $I$ to be the following geometric realisations of subcomplexes of $K$ and $K^*$:
\[
(F_I,B_I)\coloneq
\begin{cases}
\big(\left|V_I(1)\right|,\;\left|T_I(0)\right|\big) & \text{if } \kappa_I=4,\\
\big(\left|T_I(1)\right|,\;\left|V_I(0)\right|\big) & \text{if } \kappa_I=8.
\end{cases}
\]
\end{definition}

\begin{example}

Figure~\ref{fig:conn} illustrates the effect of the complementary connectivity pairings $(8,4)$ and $(4,8)$ on the resulting foreground and background topology, as seen through their geometric realisations.
\end{example}

\begin{definition}
A \emph{binary video of resolution $(w,h)$} is a finite sequence of binary images 
\[
(I_i)_{i=1}^n, \qquad I_i : \Omega_{w,h} \to \{0,1\},
\]
sharing a common foreground connectivity $\kappa \in \{4,8\}$, i.e.\ $\kappa_{I_i}=\kappa$ for all $i$.
\end{definition}

As will become clear in Section~\ref{sec:curse-of-res}, realising an entire binary video as cubical complexes becomes infeasible in both memory and computation as the resolution of the video increases. We present a workaround in Section~\ref{sec:formigrams-chapter} by focusing on the evolution of connected components, providing us with a simple combinatorial representation of the topology of a binary video. Before proceeding, we introduce the algebraic framework of zigzag persistence to relate the topology of consecutive frames in a video.

\begin{figure}[htbp]
\captionsetup[subfigure]{justification=centering}
\centering
\begin{subfigure}[c]{0.3\textwidth}
\centering
\[
\resizebox{\textwidth}{!}{$
\begin{bmatrix}
0 & 1 & 1 & 0 & 0 & 1 & 1 \\
0 & 0 & 0 & 1 & 1 & 1 & 1 \\
1 & 0 & 1 & 1 & 0 & 0 & 0 \\
1 & 0 & 0 & 0 & 0 & 1 & 1 \\
0 & 1 & 1 & 0 & 1 & 1 & 0 \\
0 & 1 & 1 & 1 & 1 & 0 & 1 \\
0 & 0 & 0 & 1 & 0 & 1 & 1
\end{bmatrix}
$}
\]
\caption{Binary image $I$}
\end{subfigure}%
\hfill
\begin{subfigure}[c]{0.3\textwidth}
\centering
\includegraphics[width=\textwidth]{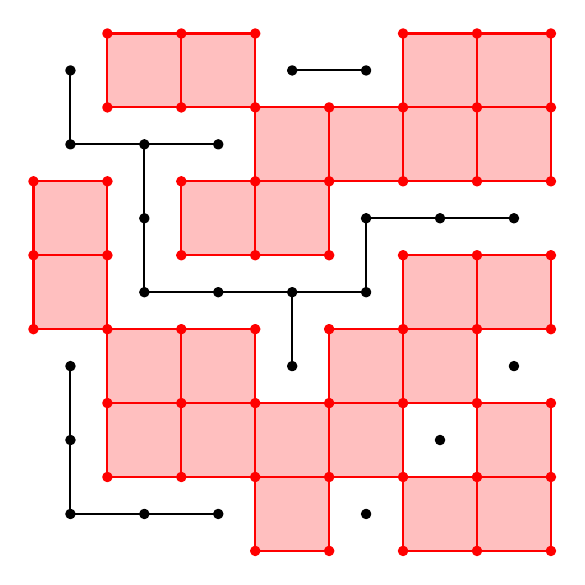}
\caption[V-construction on 1s, T-construction on 0s]{%
    \begin{tabular}{@{}r@{}}
        $F_I=\left|T_I(1)\right|$ \\ 
        $B_I=\left|V_I(0)\right|$ \\
    \end{tabular}}
\end{subfigure}%
\hfill
\begin{subfigure}[c]{0.3\textwidth}
\centering
\includegraphics[width=\textwidth]{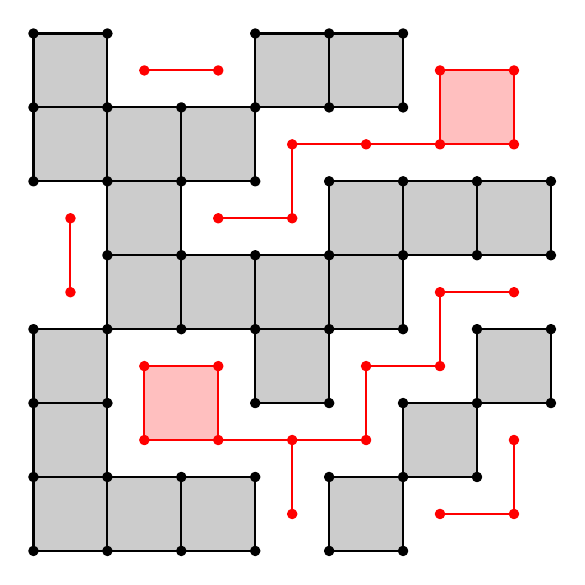}
\caption[V-construction on 1s, T-construction on 0s]{%
    \begin{tabular}{@{}r@{}}
        $F_I=\left|V_I(1)\right|$\\ 
        $B_I=\left|T_I(0)\right|$
    \end{tabular}}
\end{subfigure}

\caption{\textbf{Dual topology pairings} induced by a binary image $I$. 
Subfigures (b) and (c) correspond to the complementary connectivity choices $(\kappa_I,\bar{\kappa}_I)=(8,4)$ and $(4,8)$, respectively. In each case, the foreground and background are shown via the geometric realisations of the corresponding subcomplexes of $K$ and $K^*$, yielding distinct but consistent topologies.}
\label{fig:conn}
\end{figure}

\subsection{Zigzag Persistence}

Having established topological models for the individual frames, we now lay the foundations needed to tackle the challenge of relating these representations across a temporal sequence. In our setting of binary videos, coloured regions can appear, split, merge, or disappear entirely. In general, this precludes the use of standard persistent homology, which requires nested sequences of spaces. The theory of \emph{zigzag persistence}, first introduced and developed in~\cite{carlssonZigzagPersistence2010}, is uniquely tailored to such a scenario. In this section, we recall the necessary definitions.

Zigzag modules generalise standard persistence modules by allowing for arrows in both directions. To make this precise, we first introduce the underlying structure that allows us to encode directionality.



\begin{definition}
A \emph{zigzag type} $\tau$ is a directed graph on vertices
$\{1,\dots,|\tau|\}$ with exactly one edge between each consecutive
pair $(i,i+1)$, oriented according to a sequence
\[
\tau=(\tau_1,\dots,\tau_{|\tau|-1}),
\qquad
\tau_i\in\{\to,\leftarrow\}.
\]

\noindent The integer $|\tau|$ is called the \emph{length} of $\tau$. The associated \emph{path category} $\mathcal P_\tau$ has objects
$\{1,\dots,|\tau|\}$ and a unique morphism $i\to j$ if and only if
there exists a directed path from $i$ to $j$ in $\tau$, with
composition induced by path concatenation.
\end{definition}

\begin{example}\hspace{0.5cm}$\tau$:
\begin{tikzcd}[row sep=0.9cm]
 \overset{1}{\bullet}  \arrow[r] & \overset{2}{\bullet}   & \arrow[l] \overset{3}{\bullet}  \arrow[r] & \overset{4}{\bullet}  \arrow[r] & \overset{5}{\bullet}  &  \arrow[l]\overset{6}{\bullet} \arrow[r] & \overset{7}{\bullet} 
\end{tikzcd}
\end{example}

\begin{definition}
Let $\tau$ be a zigzag type and $\mathcal{C}$ be a category. A \emph{$\mathcal{C}$-valued diagram over $\tau$} is a functor 
\[
\mathbb{D} : \mathcal{P}_\tau \to \mathcal{C}.
\]
A morphism $\mathbb{D}_1\to\mathbb{D}_2$ between $\mathcal{C}$-valued diagrams is a natural transformation $\psi: \mathbb{D}_1 \Rightarrow \mathbb{D}_2$.
\end{definition}

\begin{remark}
Let $\mathbb{D}:\mathcal{P}_\tau\to\mathcal{C}$ be a $\mathcal{C}$-valued diagram over $\tau$. For each edge $i\to j$ in $\tau$, i.e.\ with $j=i\pm 1$, we refer to the morphism
\[
\mathbb{D}[i\to j]:D_i\to D_j
\]
as a \emph{structure map} of $\mathbb{D}$.
\end{remark}

\begin{remark}\label{rmk:comp}
We denote diagrams by blackboard bold symbols such as $\mathbb{D}$, and write $D_i\coloneq \mathbb{D}(i)$ for the point-wise object at vertex $i$ of $\tau$. Moreover, given a functor $\psi : \mathcal{C} \to \mathcal{D}$ and a $\mathcal{C}$-valued diagram $\mathbb{X} : \mathcal{P}_\tau \to \mathcal{C}$, we obtain a $\mathcal{D}$-valued diagram by composition 
\[
\psi(\mathbb{X}) := \psi \circ \mathbb{X}.
\]
\end{remark}
Throughout this paper, we fix a field $k$. We will be particularly interested in the case where $\mathcal{C}=\cat{vect_k}$ (the category of finite-dimensional vector spaces over $k$).

\begin{definition}
Let $\tau$ be a zigzag type. A \emph{zigzag module} $\mathbb{V}$ of type $\tau$ is a $\cat{vect_k}$-valued diagram over $\tau$
\[ \mathbb{V}: \mathcal{P}_\tau \to \cat{vect_k},
\]
and we write 
\[ \cat{Zig}_\tau := [\mathcal{P}_\tau, \cat{vect_k}]
\]
for the category of zigzag modules of type $\tau$.
\end{definition}

\begin{example} \label{ex:zigmod}
A simple example of a zigzag module is given by
\begin{center}
\begin{tikzcd}[row sep=0.9cm]
\textbf{Zigzag type $\tau$}:\arrow[d, "\mathbb{V}", ] & \overset{1}{\bullet} \arrow[d, swap] \arrow[r] & \overset{2}{\bullet} \arrow[d , swap]  & \arrow[l] \overset{3}{\bullet} \arrow[d , swap] \arrow[r] & \overset{4}{\bullet} \arrow[d , swap] \arrow[r] & \overset{5}{\bullet}  \arrow[d , swap]\\
\textbf{Zigzag module:} & \field \arrow[r, "\mathrm{id}", swap] & \field  & \arrow[l, "(1 \ 0 )"] \field^2 \arrow[r, "(0 \ 1)", swap] & \field \arrow[r, "0", swap] & 0
\end{tikzcd}
\end{center}
\end{example}

\noindent Unpacking the definitions, a morphism between two zigzag modules $\V$ and $\mathbb{W}$ of the same type is a collection of maps $\psi_i : V_i \to W_i$ such that the following diagrams commute
\[
\begin{tikzcd}
V_i \arrow[leftrightarrow]{r} \arrow[d, "\psi_i", swap] & V_{i+1} \arrow[d, "\psi_{i+1}"] \\
W_i \arrow[leftrightarrow]{r} & W_{i+1}.
\end{tikzcd}
\]

\noindent We denote forward structure maps by $f_i:V_i\to V_{i+1}$ and backward structure maps by $g_i: V_i \leftarrow V_{i+1}$. For any subinterval $[i,j]\subseteq [1, n ]$, we define the \emph{truncated module} $\V[i,j]$ to be the module $\V$ restricted to those indices.

\begin{remark}
The category $\cat{Zig}_\tau$ is an \emph{abelian category}~\cite{gabrielUnzerlegbareDarstellungen1972}. In particular, standard constructions such as $\mathrm{Hom}$-spaces, kernels, cokernels, and direct sums are well-defined and defined pointwise. For instance, given two zigzag modules $\mathbb{U}$ and $\mathbb{W}$, their direct sum $\V = \mathbb{U} \oplus \mathbb{W}$ is defined by
\[
V_i = U_i \oplus W_i
\]
for each index $i$, with structure maps given by the direct sums of the corresponding structure maps of $\mathbb{U}$ and $\mathbb{W}$. 
\end{remark}

Analogously to persistence modules~\cite{edelsbrunnerTopologicalPersistenceSimplification2002}, zigzag modules possess a tractable combinatorial description relying on the fact that they can be broken down into simple, easily understood, atomic components. 

\begin{definition}
Let $\V$ be a zigzag module of type $\tau$. A \emph{sub-zigzag module} $\mathbb{U}\subseteq \V$ is a subfunctor of $\V$. Concretely, this is a collection of vector subspaces $U_i\subseteq V_i$ such that every structure map of $\V$ restricts to a linear map between these subspaces.
\end{definition}

With this notion of a submodule, we can discuss decomposing a zigzag module into simpler components.

\begin{definition}
A submodule $\mathbb{U}$ is called a \emph{summand} of $\V$ if there exists a complementary submodule $\mathbb{W}$ such that $\V=\mathbb{U}\oplus\mathbb{W}$. A non-zero zigzag module $\V$ is \emph{indecomposable} if it cannot be written as a direct sum of two non-zero submodules.
\end{definition}

\begin{remark}
In contrast to ordinary persistence, not every submodule is a summand: that is, a submodule need not admit a complementary submodule; see~\cite[Example 2.10.]{carlssonZigzagPersistence2010}. This is an important distinction one has to consider while breaking down a zigzag module into indecomposable summands.
\end{remark}

In the context of zigzag modules, the indecomposables take on a particularly pleasant form.

\begin{definition}
Let $\tau$ be a zigzag type, and let $1\leq b \leq d \leq |\tau|$ be two integers. The \emph{interval module of type $\tau$} with start time $b$ and end time $d$, written $\mathbb{I}_{\tau}\langle b,d\rangle$, is defined pointwise as 
$$ I_\tau\langle b,d\rangle_i = \begin{cases}
    \field & \text{ if $b\leq i \leq d $},\\
    0 & \text{otherwise}.
\end{cases}$$
We insert identity maps between adjacent copies of $\field$ and zero maps otherwise.

\vspace{0.3cm}
\centering
\begin{tikzcd}[
  execute at end picture={
\draw[thick, decoration={brace, mirror, raise=4pt}, decorate] 
(A.south west) -- (B.south east) 
node[pos=0.5, below=10pt] {$[1, b-1]$ };
\draw[thick, decoration={brace, mirror, raise=4pt}, decorate] 
(C.south west) -- (D.south east) 
node[pos=0.5, below=10pt] {$[b, d]$ };
\draw[thick, decoration={brace, mirror, raise=4pt}, decorate] 
(E.south west) -- (F.south east) 
node[pos=0.5, below=10pt] {$[d+1, |\tau|]$ };
}
]
|[alias=A]| 0
  \arrow[r, "0"] 
& \cdots
  \arrow[l]
  \arrow[r, "0"] 
& |[alias=B]| 0
  \arrow[l]
  \arrow[r, "0"] 
& |[alias=C]|\field
  \arrow[l]
  \arrow[r, "\mathrm{id}"] 
& \cdots
  \arrow[l]
  \arrow[r, "\mathrm{id}"] 
& |[alias=D]|\field
  \arrow[l]
  \arrow[r, "0"] 
& |[alias=E]|0
  \arrow[l]
  \arrow[r, "0"] 
& \cdots
  \arrow[l]
  \arrow[r, "0"] 
& |[alias=F]|0. \arrow[l]
\end{tikzcd}
\end{definition}

It is straightforward to verify that interval modules are indecomposable. Moreover, they exhaust all indecomposable zigzag modules. Indeed, viewed through the lens of quiver theory, zigzag modules (which include persistence modules by considering the zigzag type with all structure maps pointing forward) are incarnations of representations of $A_n$ graphs~\cite{oudotPersistenceTheoryQuiver2015}. Their decomposition theorem is the simplest case of Gabriel's Theorem \cite{gabrielUnzerlegbareDarstellungen1972}.

\begin{theorem}\label{thm:gabriel}
    Every zigzag module $\V$ of type $\tau$ decomposes into a direct sum of interval modules
    \[ 
    \V \cong \bigoplus_{j=1}^N  \mathbb{I}_\tau\langle b_j,d_j\rangle.
    \]
\end{theorem}

A constructive proof of this theorem, together with a decomposition algorithm, is presented in \cite{carlssonZigzagPersistence2010}. 

\begin{example}
The zigzag module
\begin{center}
\begin{tikzcd}
\V:\ \  \field \arrow[r, "\mathrm{id}"] & \field  & \arrow[l, "(1 \ 0 )", swap] \field^2 \arrow[r, "(0 \ 1)"] & \field \arrow[r, "0"] & 0
\end{tikzcd}
\end{center}
of Example~\htest{ex:zigmod} decomposes into $\V \cong \mathbb{I}\langle 1,3\rangle \oplus\mathbb{I}\langle 3,4\rangle$.
\end{example}

The following theorem shows that decompositions of zigzag modules into indecomposable summands are unique up to isomorphism. This is a special case of the Krull–Remak–Schmidt theorem. 

\begin{theorem}[{\cite[Prop. 2.2.]{carlssonZigzagPersistence2010}}]\label{thm:unique-decomp}
Suppose a zigzag module $\V$ of type $\tau$ admits two direct sum decompositions into submodules:
$$\V = \mathbb{U}_1 \oplus\dots\oplus \mathbb{U}_M \text{\ \ and \ \ } \V=\mathbb{W}_1 \oplus \dots \oplus \mathbb{W}_N.$$ Then $M=N$ and there is some permutation $\sigma$ of $\{ 1, \dots, N \}$ such that $\mathbb{U}_j \cong \mathbb{W}_{\sigma(j)}$ for all $j$.
\end{theorem}

Theorem~\htest{thm:gabriel} together with Theorem~\htest{thm:unique-decomp} ensures that the decomposition of a zigzag module into interval summands is well-defined up to isomorphism. This allows us to distil the algebraic structure of a zigzag module into a simple combinatorial invariant.

\begin{definition}
Given a zigzag module $\mathbb{V}$ of type $\tau$ with interval decomposition $\mathbb{V} \cong \bigoplus_{j=1}^N \mathbb{I}_\tau\langle b_j, d_j \rangle$, the \emph{barcode} of $\mathbb{V}$ is the multiset\footnote{A multiset is a generalisation of a set where elements are allowed to appear more than once; the number of times an element appears is its multiplicity.} of pairs of start and endpoints, called \emph{bars},
\[
\mathrm{Bar}(\mathbb{V}) := \{ \langle b_j, d_j \rangle \mid j=1, \dots, N \}.
\]
Two barcodes are said to be equal if they coincide as multisets, i.e.\ if each pair $\langle b,d\rangle$ appears with the same multiplicity in both.
\end{definition}

\begin{remark}
Each bar $\langle b, d \rangle$ in the barcode represents the \emph{lifespan} of a \emph{feature} that persists from index $b$ to $d$ through the zigzag module.
\end{remark}

\subsection{Decorating the Endpoints}\label{sec:dec-barcode}

The barcode provides a complete invariant for a zigzag module $\V$ of type $\tau$. This means that two zigzag modules of type $\tau$ are isomorphic if and only if they have the same barcode. However, the algebraic mechanisms by which bars originate and terminate provide critical context for interpreting topological features, making them worth recording.

It is standard~\cite{MR3924175, Chazal2012TheSA, deyFastComputationZigzag2022} to decorate the endpoints of each bar $\langle b , d \rangle$ according to the directions of the structure maps at its boundary.

\begin{definition}
The left endpoint $b$ of a bar $\langle b , d \rangle\in\mathrm{Bar}(\V)$ is \emph{closed} if $V_{b-1}\to V_b$ points forward, and \emph{open} if it points backward $V_{b-1}\leftarrow V_b$. Dually, the right endpoint $d$ is \emph{closed} if $V_{d}\leftarrow V_{d-1}$ and \emph{open} if $V_{d}\to V_{d+1}$.
\end{definition}

\noindent Equivalently, closed ends arise as cokernel classes and open ends as kernel classes, yielding four interval types:

\vspace{0.5em}
\begin{center}
\begin{tabular}{c@{\hspace{0.2em}}c@{\hspace{0.2em}}c @{\hspace{0.75cm}} c@{\hspace{0.2em}}c@{\hspace{0.2em}}c}

$\mathrm{coker}$ &
\begin{tikzcd}[column sep=0.5cm, ampersand replacement=\&] 
0 \arrow[r, "f_{b-1}"] \& \field \arrow[r, leftrightarrow] \& \dots \arrow[r, leftrightarrow] \& \field \& 0 \arrow[l, "g_d"'] 
\end{tikzcd}
&
$\mathrm{coker}$,

&
$\mathrm{coker}$ &
\begin{tikzcd}[column sep=0.5cm, ampersand replacement=\&] 
0 \arrow[r, "f_{b-1}"] \& \field \arrow[r, leftrightarrow] \& \dots \arrow[r, leftrightarrow] \& \field \arrow[r, "f_d"] \& 0
\end{tikzcd}
&
$\mathrm{ker}$,
\\[0.3em]
& $[b,d]$ & & & $[b,d)$

\\[1em]

$\mathrm{ker}$ &
\begin{tikzcd}[column sep=0.5cm, ampersand replacement=\&] 
0 \& \field \arrow[l, "g_{b-1}"'] \arrow[r, leftrightarrow] \& \dots \arrow[r, leftrightarrow] \& \field \& 0 \arrow[l, "g_d"'] 
\end{tikzcd}
&
$\mathrm{coker}$,

&
$\mathrm{ker}$ &
\begin{tikzcd}[column sep=0.5cm, ampersand replacement=\&] 
0 \& \field \arrow[l, "g_{b-1}"'] \arrow[r, leftrightarrow] \& \dots \arrow[r, leftrightarrow] \& \field \arrow[r, "f_d"] \& 0 
\end{tikzcd}
&
$\mathrm{ker}$. 
\\[0.3em]
& $(b,d]$ & & & $(b,d)$

\end{tabular}
\end{center}

\vspace{0.5em}

\noindent In this way, the barcode $\mathrm{Bar}(\V)$ splits into the four types: \textit{closed-closed}, \textit{closed-open}, \textit{open-closed}, and \textit{open-open}.

When the decoration is not important, or when a bar has an endpoint at the initial or terminal index of the zigzag module, we instead use the undecorated notation $\langle b$, and $d \rangle$. It is worth reiterating that these decorations are merely a bookkeeping tool. Since the zigzag type $\tau$ is fixed for a given module, the undecorated barcode remains a complete invariant; the positions of the bars relative to the orientations in $\tau$ entirely determine the decorations.

\subsection{From Frames to Zigzag Modules}

Consider a binary video of resolution $(w,h)$, i.e.\ a finite sequence of binary images $(I_i)_{i=1}^n$, equipped with a common foreground connectivity $\kappa\in\{4,8\}$. Pixel changes between consecutive images $I_i$ and $I_{i+1}$, for $i=1,\dots,n-1$, induce additions and deletions of cells in the associated cubical complexes underlying the coloured regions $F_{I_i}$ and $B_{I_i}$.

The flexibility of zigzag persistence provides a natural framework to model the evolving topology under these non-monotonic changes. We first detail the standard constructions used to relate subspaces of a common topological space. 

The category of topological spaces is denoted by $\cat{Top}$. As all spaces we consider arise from finite binary images, their homology groups are finitely generated. Accordingly, we restrict our attention to the full subcategory $\cat{top}$ of topological spaces whose homology groups are finitely generated in every degree.

\begin{definition}
Let $\tau$ be a zigzag type. A \emph{topological diagram over $\tau$} is a $\cat{top}$-valued diagram
\[
\mathbb{X} : \mathcal{P}_\tau \to \cat{top}.
\]
\end{definition}

\begin{remark}
Concretely, this consists of a sequence of topological spaces with continuous maps $$X_1 \leftrightarrow X_2\leftrightarrow\dots \leftrightarrow X_n,$$ where the directions of the arrows are prescribed by $\tau$. Such diagrams generalise the notion of filtrations from regular persistence to zigzag persistence.
\end{remark}

The following constructions form the basic
building blocks for most common constructions in zigzag persistence.

\begin{definition}\label{def:constr}

Let $(X_1, \dots, X_n)$ be a sequence of subspaces of a common ambient topological space $X$. We define the associated \emph{union construction} $\mathbb{X}^{\expspacing \cup}$ and \emph{intersection construction} $\mathbb{X}^{\expspacing \cap}$ to be the topological diagrams:
\begin{center}
\begin{tabular}{c l}
$\mathbb{X}^{\expspacing \cup}$: &
\begin{tikzcd}[column sep=11]
X_1 \arrow[r, hook] & X_1 \cup X_2 & X_2 \arrow[l, hook']\arrow[r, hook] & X_2 \cup X_3 & X_3 \arrow[l, hook'] \arrow[r, hook] & \dots
\end{tikzcd}\\[0.5em]

$\mathbb{X}^{\expspacing \cap}$: &
\begin{tikzcd}[column sep=11]
X_1 & X_1 \cap X_2 \arrow[l, hook'] \arrow[r, hook] & X_2 & X_2 \cap X_3 \arrow[l, hook'] \arrow[r, hook] & X_3 & \dots. \arrow[l, hook']
\end{tikzcd}
\end{tabular}
\end{center}
\noindent where the maps between consecutive spaces are the natural inclusion maps.
\end{definition}

At the level of binary videos, these constructions are realised by inserting suitable interpolation images~\cite{mcdonaldZigzagPersistenceCoral2023}, defined via pointwise logic operations.

\begin{definition}
Let $I_1,I_2:\Omega_{w,h}\to \{0,1\}$ be binary images of the same resolution, equipped with a common foreground connectivity $\kappa \in \{4,8\}$. We define the pointwise operations
\[
(I_1 \vee I_2)(x)\coloneq \max\{I_1(x),I_2(x)\}, \quad \text{and} \quad (I_1 \wedge I_2)(x)\coloneq \min\{I_1(x),I_2(x)\}.
\]
The resulting binary images are equipped with the same foreground connectivity $\kappa$ as $I_1$ and $I_2$.
\end{definition}

\begin{remark}
A binary video is equipped with a fixed foreground connectivity $\kappa$, and complementary background connectivity $\bar{\kappa}$ for all its images. Accordingly, all foreground regions $F_{I_i}$ (respectively background regions $B_{I_i}$) arise as subcomplexes of the same cubical complex, either $K$ or $K^*$, depending on $\kappa$. Consequently, unions and intersections of consecutive coloured regions are well-defined.
\end{remark}

\begin{lemma}\label{lem:inter-uni}
Let $I_1, I_2:\Omega_{w,h}\to\{0,1\}$ be binary images equipped with a common foreground connectivity $\kappa$. The operations $\vee$ and $\wedge$ correspond to unions and intersections of foreground regions and background regions in the following sense:
\begin{alignat*}{2}
F_{I_1}\cup F_{I_2} \;&=\; F_{I_1\vee I_2}, \qquad & F_{I_1}\cap F_{I_2} \;&=\; F_{I_1\wedge I_2},\\
B_{I_1}\cup B_{I_2} \;&=\; B_{I_1\wedge I_2}, \qquad & B_{I_1}\cap B_{I_2} \;&=\; B_{I_1\vee I_2}.
\end{alignat*}
\end{lemma}

\begin{proof}
We prove $F_{I_1}\cup F_{I_2} = F_{I_1\vee I_2}$ in the case $\kappa=8$, where $F_I = |T_I(1)|$. Then
\[
F_{I_1}\cup F_{I_2} = |T_{I_1}(1)| \cup |T_{I_2}(1)|
= |T_{I_1}(1)\cup T_{I_2}(1)|,
\]
since both $T_{I_1}(1)$ and $T_{I_2}(1)$ are subcomplexes of the same cubical complex $K$. Substituting in the definitions and combining the unions, we obtain
\[
T_{I_1}(1)\cup T_{I_2}(1)
= \bigcup_{(i,j)\in I_1^{-1}(1)\cup I_2^{-1}(1)} \mathrm{cl}(Q_{i,j}).
\]
Using $I_1^{-1}(1)\cup I_2^{-1}(1) = (I_1\vee I_2)^{-1}(1)$, the claim follows. The remaining identities are analogous; the background cases additionally use the complementarity of the pixel sets: $I^{-1}(0)=\Omega_{w,h}\setminus I^{-1}(1)$.
\end{proof}


\begin{example}
    
Let $I_1,I_2:\Omega_{w,h}\to \{0,1\}$ be binary images equipped with a common foreground connectivity. The interpolation images $I_1 \vee I_2$ and $I_1 \wedge I_2$ induce the following diamond of inclusions at the level of foreground regions, as illustrated in Figure~\ref{fig:foreground-incl}.

\begin{figure}[htbp]
\centering
\begin{tikzpicture}[every node/.style={inner sep=0pt, anchor=center}]
    \coordinate (top)    at (0,3);
    \coordinate (left)   at (-3,0);
    \coordinate (right)  at (3,0);
    \coordinate (bottom) at (0,-3);

    \node (A) at (top)    {\includegraphics[width=2cm]{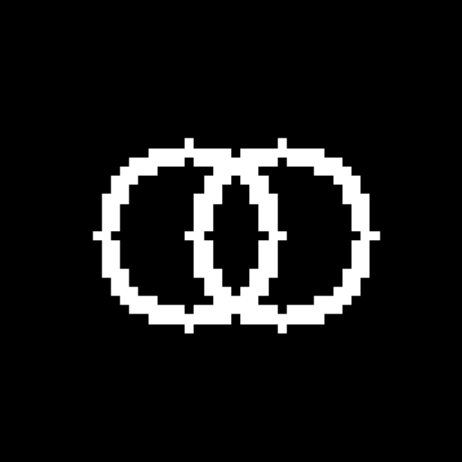}};
    \node (B) at (left)   {\includegraphics[width=2cm]{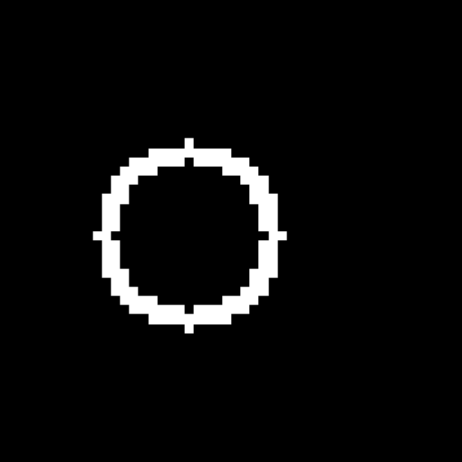}};
    \node (C) at (right)  {\includegraphics[width=2cm]{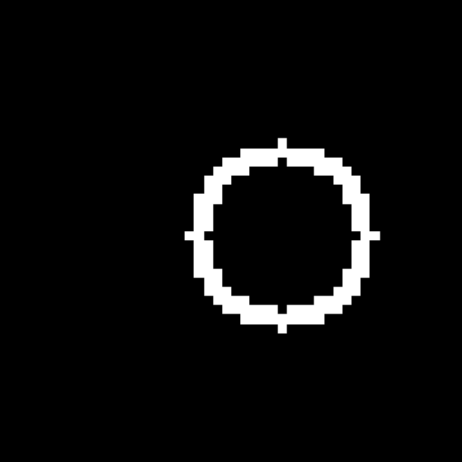}};
    \node (D) at (bottom) {\includegraphics[width=2cm]{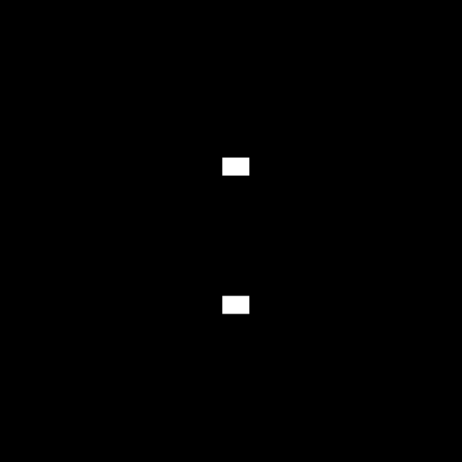}};

    \node at (top) [yshift=-1.4cm] {$I_1\vee I_2$};
    \node at (left) [yshift=-1.4cm] {$I_1$};
    \node at (right) [yshift=-1.4cm] {$I_2$};
    \node at (bottom) [yshift=-1.4cm] {$I_1\wedge I_2$};

    \draw[-Stealth, shorten >=1.5mm, shorten <=1.5mm] (C) -- (A);
    \draw[-Stealth, shorten >=1.5mm, shorten <=1.5mm] (B) -- (A);
    \draw[-Stealth, shorten >=1.5mm, shorten <=1.5mm] (D) -- (B);
    \draw[-Stealth, shorten >=1.5mm, shorten <=1.5mm] (D) -- (C);
\end{tikzpicture}
\caption{\textbf{Pointwise logic operations for binary images}. The arrows represent inclusions of the foregrounds.}
\label{fig:diamond_of_inclusions}\label{fig:foreground-incl}
\end{figure}
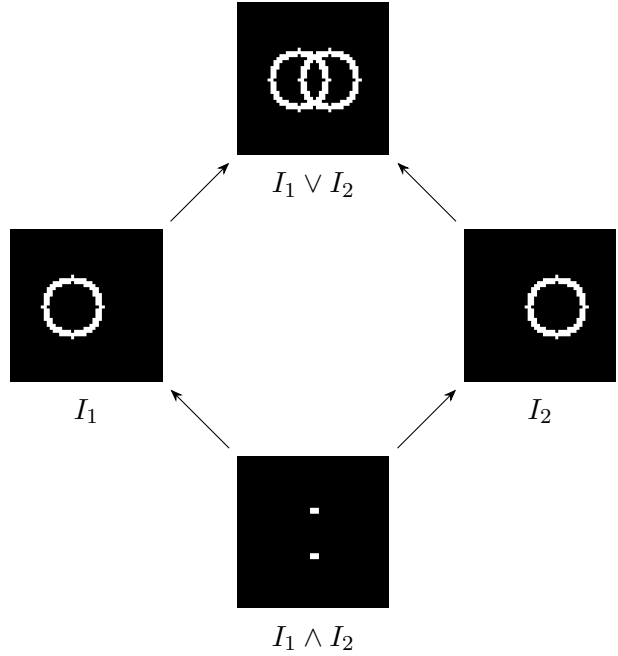
\end{example}



Lemma~\htest{lem:inter-uni} shows that unions and intersections of consecutive coloured regions are realised via interpolation frames. This leads us to the following definition.

\begin{definition}\label{def:top-constr}
Let $\mathcal{V}=(I_i)_{i=1}^n$ be a binary video. The insertion of interpolation frames of the form $I_i\vee I_{i+1}$ defines the topological diagrams
\begin{center}
\begin{tikzcd}[row sep=2pt]
\mathbb{F}^\cup_\mathcal{V}: & F_{I_1} \arrow[r, hook] & F_{I_1\vee I_2} & \arrow[l, hook'] F_{I_2} \arrow[r,hook] & \dots \arrow[r,hook]&  F_{I_{n-1}\vee I_n}  &
\arrow[l, hook'] F_{I_n}\\
\mathbb{B}^\cap_\mathcal{V}: & B_{I_1}  & \arrow[l,hook'] B_{I_1\vee I_2} \arrow[r, hook] &  B_{I_2}  & \arrow[l, hook'] \dots & \arrow[l,hook']  B_{I_{n-1}\vee I_n} \arrow[r, hook] &
B_{I_n}
\end{tikzcd}
\end{center}
 while the insertion of interpolation frames $I_i\wedge I_{i+1}$  defines
\begin{center}
\begin{tikzcd}[row sep=2pt]
\mathbb{F}^\cap_\mathcal{V}: & F_{I_1}  & \arrow[l,hook'] F_{I_1\wedge I_2} \arrow[r, hook] &  F_{I_2}  & \arrow[l, hook'] \dots & \arrow[l,hook']  F_{I_{n-1}\wedge I_n} \arrow[r, hook] &
F_{I_n}\\
\mathbb{B}^\cup_\mathcal{V}: & B_{I_1} \arrow[r, hook] & B_{I_1\wedge I_2} & \arrow[l, hook'] B_{I_2} \arrow[r,hook] & \dots \arrow[r,hook]&  B_{I_{n-1}\wedge I_n}  &
\arrow[l, hook'] B_{I_n}.
\end{tikzcd}
\end{center}
\end{definition}

\begin{example}
An example of the union construction $\mathbb{F}^{\cup}_\mathcal{V}$ using 8-connectivity on the foreground can be seen in Figure~\htest{fig:formi-first}.
\end{example}

Through Remark~\htest{rmk:comp}, topological diagrams are our main source of zigzag modules.

\begin{definition}\label{def:zig-asso}
Let $\mathbb{X} : \mathcal{P}_\tau \to \cat{top}$ be a topological diagram and let $H_\bullet(-,k): \cat{top} \to \cat{vect_k}$ be the singular homology functor of an arbitrary degree. The \emph{$H_\bullet$-zigzag module} associated to $\mathbb{X}$ is the zigzag module
\[
H_\bullet(\mathbb{X};k) := H_\bullet(-;k) \circ \mathbb{X}:\mathcal P_\tau \to \cat{vect_k}.
\]
We write 
$\mathrm{Bar}(H_*(\mathbb X)) \coloneq \bigsqcup_{l\geq 0} \mathrm{Bar}(H_l(\mathbb X)), $ for the disjoint union of the barcodes over all homological degrees.
\end{definition}

\begin{remark}
In Section~\ref{sec:dualities}, we investigate relations between $\mathrm{Bar}(H_*(\mathbb{F}^\cup_\mathcal{V}))$ and $\mathrm{Bar}(H_*(\mathbb{B}^\cap_\mathcal{V}))$, as well as between $\mathrm{Bar}(H_*(\mathbb{X}^\cap))$ and $\mathrm{Bar}(H_*(\mathbb{X}^\cup))$ for a general topological diagram $\mathbb{X}$.

\end{remark}

The zigzag modules of Definition~\htest{def:zig-asso} track the persistence of features, such as connected components and non-contractible cycles, across evolving spaces. The next subsection addresses the cost of computing the interval decomposition of such zigzag modules using cubical complexes.

\subsection{The Curse of Resolution}\label{sec:curse-of-res}

Knowing how to relate consecutive frames via insertion of suitable intermediate frames, we now confront the fundamental obstacle that motivates the rest of this paper: computing the barcode using cubical complexes is computationally infeasible for videos of even moderate resolution. Standard algorithms for zigzag persistence~\cite{carlssonZigzagPersistence2010,Carlsson2009ZigzagPH} compute an interval decomposition by maintaining a reduced boundary matrix $D$ associated with a sequence of cell insertions and deletions. The computational cost is governed by two factors: the size of the underlying complex and the total number of updates required. 

The first difficulty is the size of the cubical complex itself. For a binary image of resolution $w\times h$, with a total of $p$ pixels, the complex $K$ underlying the $T$-construction has $(w+1)(h+1)$ vertices, $2wh + w + h$ edges, and $wh$ faces --- approximately $4wh$ cells in total. For a $1000 \times 1000$ image, this already amounts to around four million cells, resulting in a boundary matrix with $\mathcal{O}(p^2)$ entries. Moreover, a binary video with $n$ frames then gives rise to a zigzag sequence of $2n-1$ spaces, each realised as a subcomplex of $K$.

The second difficulty is the number of cell updates. Consider the scenario illustrated in Figure~\ref{fig:three-figures}: a circular foreground region of radius $r=200$ in a $1000\times1000$ video that moves slightly between two consecutive frames $F_1$ and $F_2$. Due to the high resolution, even a modest displacement of a few pixels shifts the boundary of the circle, so that approximately $2 \times 10^4$ pixels are gained going from $F_1$ to $F_1\cup F_2$ and $2 \times 10^4$ pixels are lost moving from $F_1\cup F_2$ to $F_2$. This amounts to a total of $|F_1 \triangle F_2|\approx 4\times10^4$ cell insertions and deletions that must be processed to pass from $F_1$ to $F_2$ through $F_1 \cup F_2$, each triggering a column reduction on the boundary matrix of dimension $\mathcal{O}(p)\times\mathcal{O}(p)$, and costing $\mathcal{O}(p^3)$ in the worst case. For $p\approx 10^6$ and $|F_1\Delta F_2| \approx 4 \cdot 10^4$, this is an astronomically high cost, rendering naive cubical zigzag persistence impractical for any realistic resolution.

\begin{figure}[htbp]
  \centering
  \begin{tikzpicture}[
    image/.style={},
    arrow/.style={->, >=stealth, thick},
    swatch/.style={draw=none, fill=#1, rectangle, minimum size=2.5mm, inner sep=0pt},
  ]
  \definecolor{localCyan}{RGB}{0, 200, 200}
\definecolor{mygray}{RGB}{128, 86, 119}

    \node[image] (img1) at (0,0)
      {\includegraphics[width=0.22\textwidth]{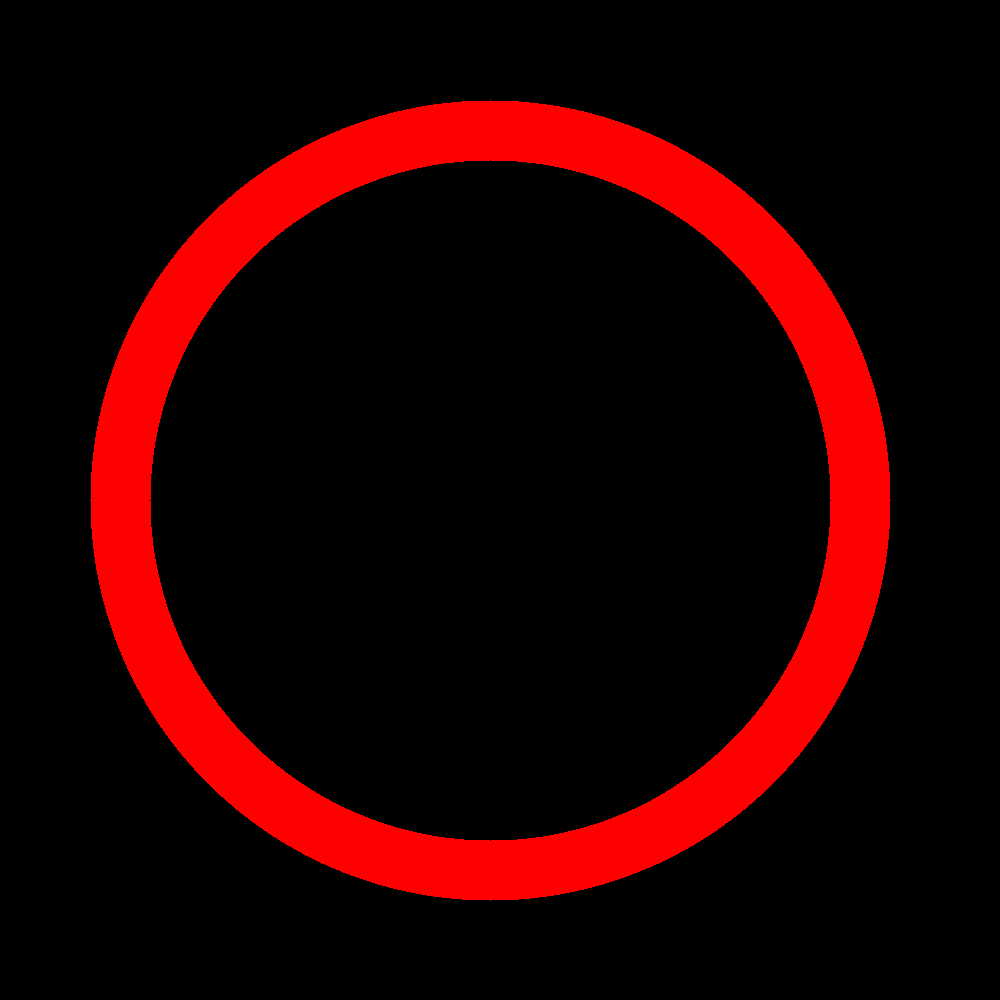}};

    \node[image] (img2) at (5.3,0)
      {\includegraphics[width=0.22\textwidth]{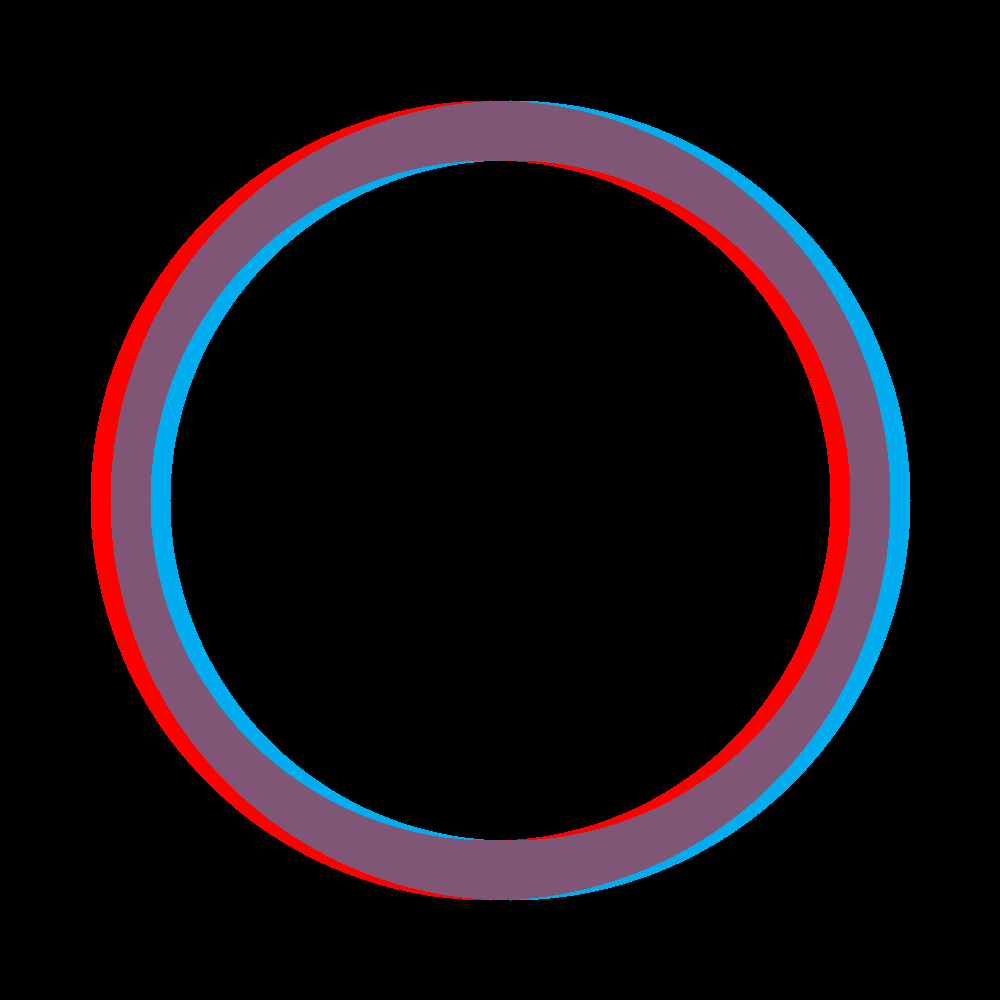}};

    \node[image] (img3) at (10.6,0)
      {\includegraphics[width=0.22\textwidth]{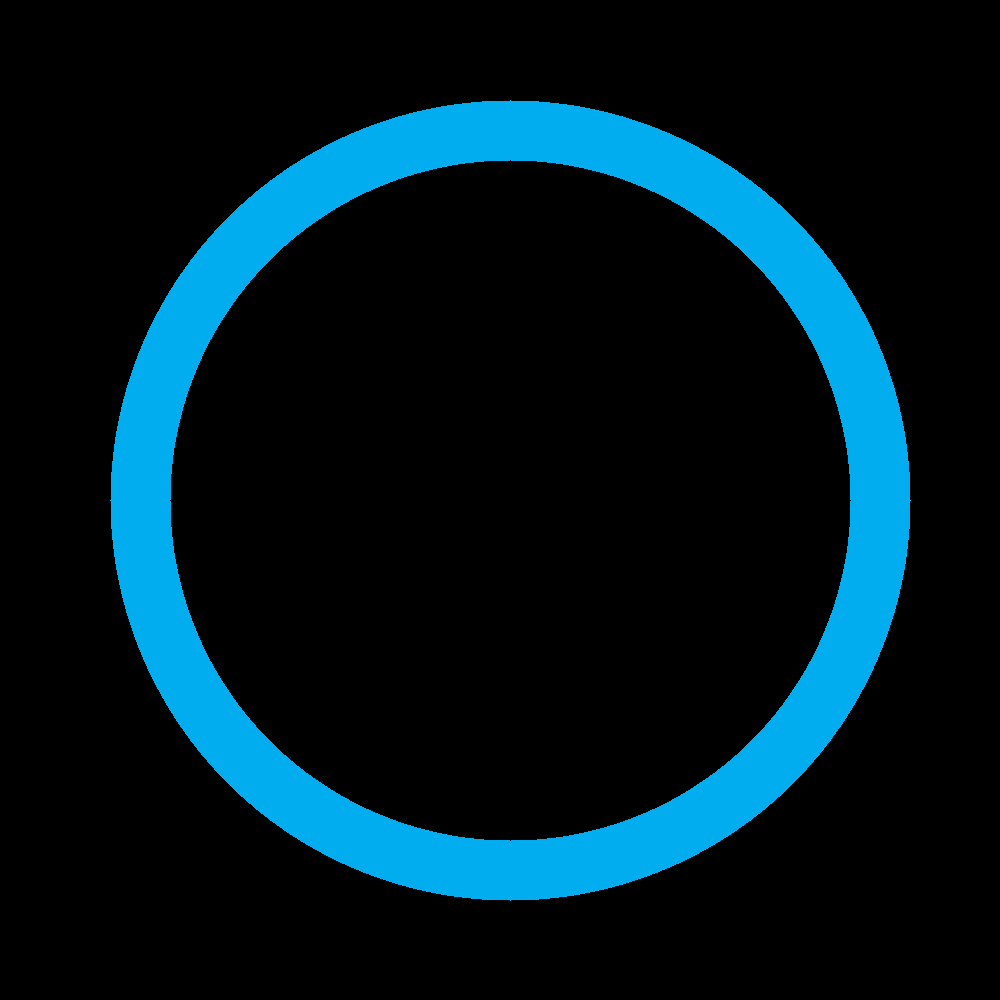}};


    \draw[arrow] (img1) -- (img2) node[midway, below] { $+$\textcolor{cyan}{$F_2 \setminus F_1$}};

    \draw[arrow] (img3) -- (img2) node[midway, below] { $-$\textcolor{red}{$F_1 \setminus F_2$}};

    \node at (img1.north) [above=0.05cm] {$F_1$};
    \node at (img2.north) [above=0.05cm] {$F_1\cup F_2$};
    \node at (img3.north) [above=0.05cm] {$F_2$};


    \node at (img2.south) [below=0.2cm, draw=gray!30, fill=gray!5, rounded corners=2pt, inner sep=1pt] {
    \setlength{\tabcolsep}{2pt}
      \begin{tabular}[colsep=10pt]{ll}
        \tikz\node[swatch=cyan]{};  & \footnotesize \textcolor{cyan}{$F_2 \setminus F_1$ (Added)} \\
        \tikz\node[swatch=mygray]{};  & \footnotesize \textcolor{mygray}{$F_1 \cap F_2$ (Static)} \\
        \tikz\node[swatch=red]{};   & \footnotesize \textcolor{red}{$F_1 \setminus F_2$ (Dropped)} 
      \end{tabular}
    };

     \node at (img1.south) [below=0.2cm, draw=gray!30, fill=gray!5, rounded corners=2pt, inner sep=1pt] {
    \setlength{\tabcolsep}{2pt}
      \begin{tabular}[colsep=10pt]{ll}
        \tikz\node[swatch=red]{};   & \footnotesize \textcolor{red}{$F_1$} 
      \end{tabular}
    };

   \node at (img3.south) [below=0.2cm, draw=gray!30, fill=gray!5, rounded corners=2pt, inner sep=1pt] {
    \setlength{\tabcolsep}{2pt}
      \begin{tabular}[colsep=10pt]{ll}
        \tikz\node[swatch=cyan]{};   & \footnotesize \textcolor{cyan}{$F_2$} 
      \end{tabular}
    };

  \end{tikzpicture}
  \caption{\textbf{The curse of resolution}: Transition between two frames $F_1$ and $F_2$ of resolution $1000\times 1000$ via their union. Pixels in $F_2\setminus F_1$ (cyan) are inserted going from $F_1$ to $F_1\cup F_2$, while pixels in $F_1\setminus F_2$ (red) are deleted going from $F_1\cup F_2$ to $F_2$. Due to the high resolution, even a small displacement leads to a large symmetric difference; in this example, a total of $|F_1\Delta F_2|\approx 4\cdot10^4$ pixels change, each inducing cell insertions and deletions in the cubical complex.}
  \label{fig:three-figures}
\end{figure}
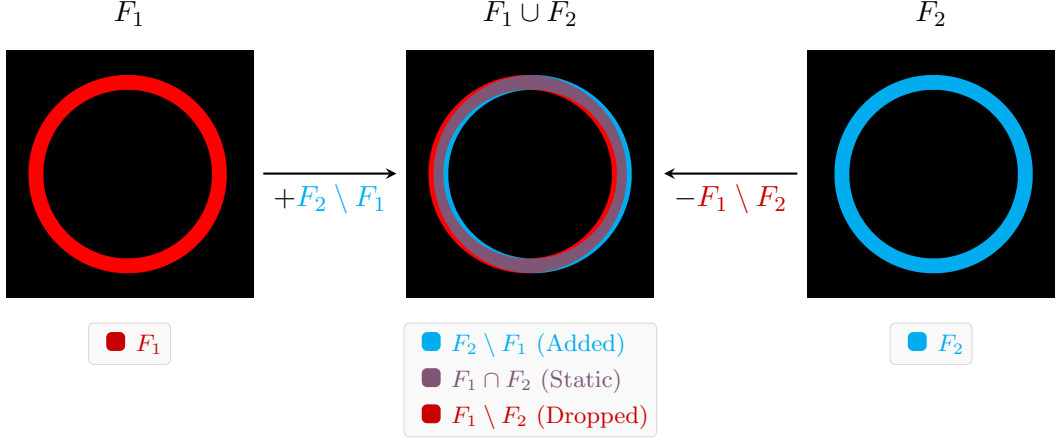

One might hope that the regular structure of cubical complexes mitigates this cost. Since each edge meets exactly 2 vertices and each face meets exactly 4 edges, the boundary matrix $D$ is highly sparse, and its columns can be computed on-the-fly by enumerating faces rather than being stored explicitly. For ordinary persistent homology, the regular structure of cubical complexes can be exploited very effectively by modern software such as Cubical Ripser~\cite{Kaji2020CubicalRS}, which inherits many of the ideas underlying Ripser~\cite{Bauer2021Ripser}.

However, these techniques rely on the fact that for standard persistence, cells are inserted once and never removed (monotonicity of the filtration). In the zigzag setting, this cannot be guaranteed, as cells are both inserted and deleted as frames evolve, and these monotonicity-based optimisations do not apply; it is unclear whether analogous cubical-specific optimisations could be recovered. While similar cubical-specific optimisations could in principle be incorporated into zigzag persistence algorithms, existing implementations for binary videos currently do not exploit them. As we demonstrate in Section~\ref{sec:benchmarking}, current implementations are consequently limited to resolutions on the order of 200×200 pixels.

The rest of this paper is dedicated to sidestepping this obstacle entirely by observing that, for binary images, the relevant topological information can be extracted from the 0-skeleton, without ever fully realising the cubical complexes.

\section{Formigrams: the 0-Skeleton of Zigzag Persistence}\label{sec:formigrams-chapter}

In this section, we restrict our attention to $H_0$-zigzag modules arising from topological diagrams and establish the combinatorial nature of the functor $H_0(-, k ):\cat{top}\to\cat{vect_k}$ by directly encoding the topological information contained in $H_0$-zigzag modules into graph structures that can be computed efficiently (c.f. Section~\ref{sec:frames-to-formi}), thereby avoiding the construction of cubical complexes altogether.

We adopt a perspective similar to that of~\cite{kimGeneralizedPersistenceDiagrams2021}. The functor $H_0$ factors through the category $\cat{set}$~\cite{bauerAdditiveImage0th2025} of finite sets:
\[ 
H_0(-; k )= k[-] \circ \pi_0,
\]
where $\pi_0:\cat{top}\to \cat{set}$ is the path-component functor, assigning to each topological space its set of path-connected components, and to each continuous map $f:X\to Y$ the induced map
\[
\pi_0(f):\pi_0(X)\to\pi_0(Y), \quad [x]\mapsto [f(x)].
\]
Moreover, $k[-]:\cat{set}\to \cat{vect_k}$ denotes the free vector space functor over a fixed field $\field$. Thus, for any topological diagram $\mathbb{X}$ over $\tau$, we have $H_0(\mathbb{X})= k[\pi_0(\mathbb{X})]$, i.e.\ $H_0(\mathbb{X})$ is the linearisation of $\pi_0(\mathbb{X})$. If $i\to j$ is an edge in $\tau$, let $p_i\coloneq \mathbb{X}[i\to j] : X_i \to X_j$ denote the continuous map between adjacent spaces in $\mathbb{X}$. Then
\[
\pi_0(p_i):\pi_0(X_i)\to \pi_0(X_{j})
\]
sends a connected component of $X_i$ to a single connected component of $X_j$, reflecting that connected components cannot split under continuous maps. 

In~\cite{deyComputingZigzagPersistence2021}, Dey and Hou exploit that the combinatorial nature of $\pi_0(\mathbb{X})$ carries over to its linearisation $H_0(\mathbb{X};k)$ to provide a graph-based algorithm that computes the interval decomposition of a certain generic\footnote{c.f. Def~\ref{def:elem-zigzag-mod}.} class of $H_0$-zigzag modules. In this way, costly Gaussian elimination can be avoided entirely. We recall their method in Section~\ref{sec:the-algo} and show that it extends to all $H_0$-zigzag modules in Section~\ref{sec:elem-refin}.

Moreover, in Section~\ref{sec:alex-duality}, we show how the $H_1$-zigzag persistence of one colour can be recovered from the $H_0$-zigzag persistence of the opposite colour using Alexander duality, thereby reducing the study of the evolution of both connected components and loops to $H_0$-zigzag persistence computations. 

Before introducing the class of graph structures underlying the interval decomposition algorithm, we highlight a key result.

\subsection{Field Independence of \texorpdfstring{$H_0$}{H₀}-Zigzag Persistence}
\label{sec:field-indep}

The fact that $H_0$-zigzag persistence of a topological diagram $\mathbb X$ is the linearisation of the $\cat{set}$-valued diagram $\pi_0(\mathbb X)$ has deep consequences for its interval decomposition. The following result is an immediate consequence of Proposition~4.9 and Theorem~3.14 in~\cite{kim_generalized_2021}.

\begin{theorem}\label{thm:coeff}
Let $\mathbb X: \mathcal P_\tau \to \cat{top}$ be a topological diagram, and let $k_1$ and $k_2$ be fields. Then 
$$\mathrm{Bar}(H_0(\mathbb X;k_1)) = \mathrm{Bar}(H_0(\mathbb X;k_2)) . $$
\end{theorem}

We provide a proof at the end of the section. In the special case of persistent homology, this was shown to hold in~\cite{obayashiFieldChoiceProblem2023}. We recall the necessary results and refer to Appendix~\ref{appendix:some-cat} for the required categorical notions. 

Let $\mathcal C$ be a category admitting all $\mathcal P_\tau$-shaped limits and colimits\footnote{C.f. Remark~\ref{rmk:lim-colim}.}. For a given diagram 
$$\mathbb D : \mathcal P_\tau \to \mathcal C, $$ 
denote by $(\varprojlim\mathbb D, (\pi_i)_{i\in\mathcal P_\tau})$ its limit and by $(\varinjlim\mathbb D, (\iota_i)_{i\in\mathcal P_\tau})$ its colimit. For a morphism $e:i \to j$ in $\mathcal P_\tau$, these satisfy the usual compatibility conditions
$$ \mathbb D[e] \circ \pi_i = \pi_j, \qquad\text{and} \qquad \iota_j\circ\mathbb D[e] =\iota_i   .$$
These morphisms can be combined into the following commutative diagram

\begin{center}
\begin{tikzcd}[column sep=small,]
&& \arrow[lld, "\pi_1",swap]  \varprojlim \mathbb D \arrow[d, "\pi_i",swap]   \arrow[rrd, "\pi_n"]   \\ 
D_1\arrow[drr, "\iota_1", swap] \arrow[r, <->] & \dots \arrow[r, <->] & \arrow[d,"\iota_i",swap]  D_i \arrow[r, <->] & \dots \arrow[r, <->] & D_n \arrow[lld,"\iota_n"]\\
&&      \varinjlim \mathbb D 
\end{tikzcd}
\end{center}
Concretely, this means that if we set $\psi_i=\iota_i\circ\pi_i$, we have $\psi_i=\psi_j$ for each $i,j\in\tau$. Therefore, we are able to define
$$\psi_{\mathbb D} : \varprojlim \mathbb D \to \varinjlim \mathbb D $$ to be the unique morphism such that
$$\psi_{\mathbb D} = \psi_i, \text{ for each $i\in\tau$.} $$

\begin{definition}[Def. 3.5 in~\cite{kim_generalized_2021}]
Let $\V:\mathcal P_\tau \to \cat{vect_k}$ be a zigzag module over $\tau$. For $i,j\in\tau$, the \textit{generalised rank invariant} is given by 
$$\mathrm{rk}_{i,j}(\V) \coloneq\mathrm{rank}\left( \psi_{\V[i,j]} : \varprojlim \V[i,j] \to \varinjlim  \V[i,j]\right). $$
\end{definition}

As shown in~\cite[Thm 3.14]{kim_generalized_2021}, the generalised rank invariant determines the interval decomposition fully. The following statement follows as a corollary.

\begin{lemma}\label{lem:rank-total}
Let $\V_1: \mathcal P_\tau \to \cat{Vect_{k_1}}$ and $\V_2: \mathcal P_\tau \to \cat{Vect_{k_2}}$ be zigzag modules of type $\tau$ defined over fields $k_1$ and $k_2$. Then
$$ \mathrm{rk}_{i,j}(\V_1)=\mathrm{rk}_{i,j}(\V_2) \quad \forall i\leq j \implies    \mathrm{Bar}(\V_1) = \mathrm{Bar}(\V_2).$$
\end{lemma}
\begin{proof}
Zigzag modules are interval decomposable. Hence, by \cite[Thm 3.14]{kim_generalized_2021}, the generalised rank invariant fully recovers the barcode via Möbius inversion. This statement is independent of the field used. Therefore, equality of rank invariants implies equality of barcodes.
\end{proof}

The following general result cements the fact that the purely combinatorial structure of formigrams fully determines $H_0$-zigzag modules.

\begin{proposition}[{\cite[Prop. 4.9]{kim_generalized_2021}}]\label{prop:fields-and-setmaps}
Let $\mathbb S : \mathcal P_\tau \to \cat{set}$ be a $\cat{set}$-valued diagram, and let $k$ be a field. Then 
$$ \mathrm{rank} \left(  \psi_{k[\mathbb{S}]} : \varprojlim k[\mathbb S] \to \varinjlim k[\mathbb{S}]     \right)
= 
\left|  \im(\psi_{\mathbb S} : \varprojlim \mathbb S \to \varinjlim \mathbb{S}  )      \right|,
$$
where $|\cdot|$ denote the set cardinality.
\end{proposition}
\begin{proof}
The statement is proved in \cite[Prop.~4.9]{kim_generalized_2021} using extended persistence. For completeness, we provide an alternative self-contained proof in Appendix~\ref{app:proof-of-prop}.
\end{proof}

\begin{lemma}\label{lem:field-indep-linearisations}
Let $\mathbb S : \mathcal P_\tau \to \cat{set}$ be a $\cat{set}$-valued diagram, and let $k_1$ and $k_2$ be fields. Then
\[
\mathrm{Bar}(k_1[\Sb])=\mathrm{Bar}(k_2[\Sb]).
\]
\end{lemma}

\begin{proof}
From Proposition~\ref{prop:fields-and-setmaps}, it immediately follows that
$$ \mathrm{rk}_{i,j}\big(k_1[\Sb]\big)=\mathrm{rk}_{i,j}\big(k_2[\Sb]\big)  ,\quad \text{ for all $i\leq j$.}$$
Lemma~\ref{lem:rank-total} then yields the claim.
\end{proof}

\begin{proof}[Proof of Theorem~\ref{thm:coeff}]
The proof follows by setting $\mathbb{S} = \pi_0(\mathbb X)$ in Lemma~\ref{lem:field-indep-linearisations}.
\end{proof}

\subsection{Merge Graphs}

To capture the class of graphs arising in the interval decomposition algorithm, we introduce a category that encodes a broader notion of morphisms between sets.

\begin{definition}
The category $\cat{set_{par}}$ is defined as follows:
\begin{itemize}
\item \textbf{Objects}: Finite sets.
\item \textbf{Morphisms}: A morphism $f\in\mathrm{Hom}_{\cat{set_{par}}}(X,Y)$ is a \textit{partial function}, i.e.\ a function 
$$f: \mathrm{Dom}(f) \to Y,$$
for some subset $\mathrm{Dom}(f) \subseteq X.$ We denote it by $f:X\rightharpoonup Y$.
\end{itemize}
Composition and identities are defined in the obvious way. 
\end{definition}

\begin{remark}
We say that a partial function $f:X\rightharpoonup Y$ extends to a \textit{total function} if $\mathrm{Dom}(f)=X$, in which case it defines a morphism in $\cat{set}$. Conversely, every function $g\in\Hom_{\cat{set}}(X,Y)$ defines a partial function with $\mathrm{Dom}(g)=X$.
\end{remark}

\begin{definition}
A \emph{partial formigram} is a $\cat{set_{par}}$-valued diagram
\[
\mathbb{S}:\mathcal{P}_\tau \to \cat{set_{par}}.
\]
If all structure maps of $\mathbb{S}$ extend to total functions, we call $\mathbb{S}$ a \emph{formigram}.
\end{definition}

\begin{example}
For any topological diagram $\mathbb{X}:\mathcal{P}_\tau\to\cat{top}$, the $\cat{set}$-valued diagram $\pi_0(\mathbb{X})$ defines a formigram, by viewing each structure map as a partial function with full domain.
\end{example}

\begin{remark}
In a setting, such as that of Definition~\htest{def:constr}, where the topological spaces $(X_i)_{i=1}^n$ of a topological diagram $\mathbb{X}$ are realised as subspaces of a common ambient space $X$, i.e.\  when $X_i\subseteq X$ for all $i$, the connected components $\pi_0(X_i)$, together with the complement $X\setminus X_i$, can be interpreted as a partition of $X$. From this perspective, formigrams encode a notion of evolving partitions, recovering the original definition introduced by Kim and Mémoli~\cite{MemoliExtractingPersistentClustersinDynamicData}.
\end{remark}

Partial formigrams admit a natural representation as directed graphs, which makes their structure explicit and suitable for algorithmic manipulation.

\begin{definition}\label{def:graph-formi}
A \emph{merge graph}\footnote{The resemblance of these graphs to ant tunnels served as inspiration for the name formigram in~\cite{MemoliExtractingPersistentClustersinDynamicData}, see Figure~\ref{fig:formi}. Merge graphs associated to formigrams are referred to as barcode graphs in~\cite{deyComputingZigzagPersistence2021}. } over a zigzag type $\tau$ is a directed graph $\mathcal M$ together with a morphism of directed graphs
\[
\lvl : \mathcal{M} \to \tau.
\]
For each vertex $i$ of $\tau$, let
\(
M_i := \lvl^{-1}(i)
\)
denote the fibre over $i$. For each directed edge $e:i\to j$ in $\tau$, the set of edges between $M_i$ and $M_j$ is given by the graph of a partial function $f_e: M_i \rightharpoonup M_j$:
$$\lvl^{-1} (e) = \{ (u,f_e(u) \mid u\in\mathrm{Dom}(f_e)  \}. $$
\end{definition}

\begin{remark}
Partial formigrams and merge graphs over a fixed zigzag type $\tau$ are in bijective correspondence, with structure maps between consecutive levels corresponding to edges given as graphs of partial functions. We will therefore freely identify the two. For a partial formigram $\mathbb{S}$, we write $\mathcal{G}(\mathbb{S})$ for the associated merge graph, and for a merge graph $\mathcal{M}$, we write $\mathcal{F}(\mathcal{M})$ for the corresponding partial formigram.
\end{remark}

\begin{remark}
The defining condition on the edge sets can be interpreted locally as follows. For each vertex $u\in G_i$ and each directed edge $e:i\to j$ in $\tau$, there is at most one edge in $\mathcal{G}$ with source $u$ mapping to $e$ under the level map. In the case of a formigram, this edge exists and is unique.
\end{remark}

\begin{center}
\vspace{0.1cm}
\begin{tikzpicture}[yscale=-1, scale=1.5, swatch/.style={draw=none, fill=#1, rectangle, minimum size=2.5mm, inner sep=0pt},
f_arrow/.style={
    base_style,
    postaction={
      decorate,
      decoration={
        markings,
        mark=at position 0.6 with {\arrow[scale=1.5]{stealth}}
      }
    }
  },]
\def\xscale{1pt}
\def\yscale{1pt}
\def\nodesize{1.5pt}
\def\specialsize{2pt}

\def\ytop{1}

\def\yarrow{0.54}
\def\ylinebot{0.7}
\def\ylinetop{2.3}
\def\ylabel{0.42}
\def\upshift{0.05}

\foreach \x in {0,...,4} {
    \draw[dashed] (\x,\ylinebot+\upshift) -- (\x,\ylinetop+\upshift);
    \node at (\x,\ylabel+\upshift) { $\overset{\fpeval{1+\x}}{\bullet}$};
}

\draw[->, >=stealth, shorten >=4pt, shorten <=4pt] (0,\yarrow) -- (1,\yarrow);
\draw[<-, >=stealth, shorten >=4pt, shorten <=4pt] (1,\yarrow) -- (2,\yarrow);
\draw[<-, >=stealth, shorten >=4pt, shorten <=4pt] (2,\yarrow) -- (3,\yarrow);
\draw[->, >=stealth, shorten >=4pt, shorten <=4pt] (3,\yarrow) -- (4,\yarrow);

\node (a) at (0,1) [circle, draw, fill=white, inner sep=1.5pt] {};
\node (a_2) at (0,2) [circle, draw, fill=white, inner sep=1.5pt] {};

\node (b) at (1,1.5) [circle, draw, fill=white, inner sep=1.5pt] {};

\node (d) at (2,1) [circle, draw, fill=white, inner sep=1.5pt] {};
\node (d_2) at (2,1.5) [circle, draw, fill=white, inner sep=1.5pt] {};
\node (d_3) at (2,2) [circle, draw, fill=white, inner sep=1.5pt] {};

\node (e) at (3,1) [circle, draw, fill=white, inner sep=1.5pt] {};
\node (e_2) at (3,1.5) [circle, draw, fill=white, inner sep=1.5pt] {};
\node (e_3) at (3,2) [circle, draw, fill=white, inner sep=1.5pt] {};

\node (f) at (4,1.5) [circle, draw, fill=white, inner sep=1.5pt] {};



\draw[f_arrow] (a) to (b);
\draw[f_arrow] (a_2) to (b);

\draw[f_arrow] (d) to (b);
\draw[f_arrow] (d_2) to (b);
\draw[f_arrow] (d_3) to (b);

\draw[f_arrow] (e) to (d);
\draw[f_arrow] (e_2) to (d_2);
\draw[f_arrow] (e_3) to (d_3);

\draw[f_arrow] (e) to (f);
\draw[f_arrow] (e_2) to (f);
\draw[f_arrow] (e_3) to (f);


\node at (2, \ylabel-0.4) {$\tau$};
\node at (2, \ylabel+2.3) {$\mathcal{G}(\mathbb{S}_1)$};

\begin{scope}[xshift=5cm]
\foreach \x in {0,...,4} {
    \draw[dashed] (\x,\ylinebot+\upshift) -- (\x,\ylinetop+\upshift);
    \node at (\x,\ylabel+\upshift) { $\overset{\fpeval{1+\x}}{\bullet}$};
}

\draw[->, >=stealth, shorten >=4pt, shorten <=4pt] (0,\yarrow) -- (1,\yarrow);
\draw[<-, >=stealth, shorten >=4pt, shorten <=4pt] (1,\yarrow) -- (2,\yarrow);
\draw[<-, >=stealth, shorten >=4pt, shorten <=4pt] (2,\yarrow) -- (3,\yarrow);
\draw[->, >=stealth, shorten >=4pt, shorten <=4pt] (3,\yarrow) -- (4,\yarrow);

\node (a) at (0,1) [circle, draw, fill=white, inner sep=1.5pt] {};
\node (a_2) at (0,2) [circle, draw, fill=white, inner sep=1.5pt] {};

\node (b) at (1,1.5) [circle, draw, fill=white, inner sep=1.5pt] {};

\node (d) at (2,1) [circle, draw, fill=white, inner sep=1.5pt] {};
\node (d_2) at (2,1.5) [circle, draw, fill=white, inner sep=1.5pt] {};
\node (d_3) at (2,2) [circle, draw, fill=white, inner sep=1.5pt] {};

\node (e) at (3,1) [circle, draw, fill=white, inner sep=1.5pt] {};
\node (e_2) at (3,1.5) [circle, draw, fill=white, inner sep=1.5pt] {};
\node (e_3) at (3,2) [circle, draw, fill=white, inner sep=1.5pt] {};

\node (f) at (4,1.5) [circle, draw, fill=white, inner sep=1.5pt] {};



\draw[f_arrow] (a) to (b);

\draw[f_arrow] (d_2) to (b);
\draw[f_arrow] (d_3) to (b);

\draw[f_arrow] (e) to (d);
\draw[f_arrow] (e_2) to (d_2);
\draw[f_arrow] (e_3) to (d_3);

\draw[f_arrow] (e_2) to (f);
\draw[f_arrow] (e_3) to (f);


\node at (2, \ylabel-0.4) {$\tau$};
\node at (2, \ylabel+2.3) {$\mathcal{G}(\mathbb{S}_2)$};
\end{scope}
\end{tikzpicture}
\vspace{0.1cm}

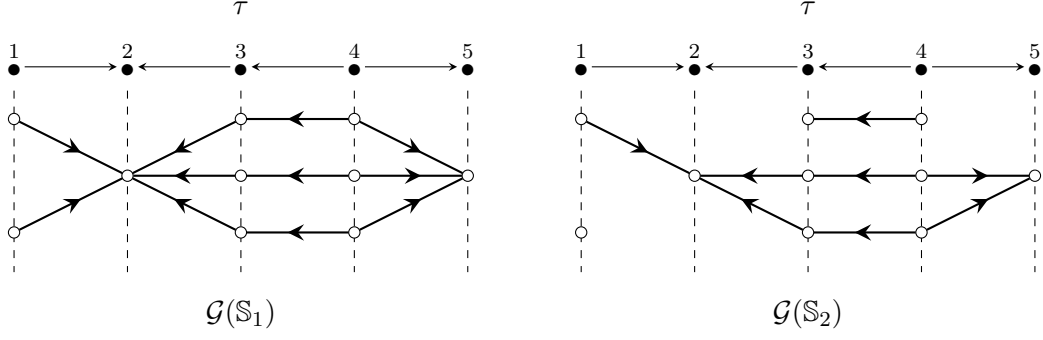
\captionof{figure}{\textbf{Merge graphs of partial formigrams $\mathbb{S}_1,\mathbb{S}_2:\mathcal{P}_\tau\to\cat{set_{par}}$.} Note that the edge sets of $\mathbb{S}_1$ arise as graphs of total functions, hence $\mathbb{S}_1$ is in fact a formigram, not just a partial formigram.} \label{fig:formi}
\end{center}

For a topological diagram $\mathbb{X}$, the merge graph $\mathcal{G}(\pi_0(\mathbb{X}))$ thus encodes the evolution of connected components across $\mathbb{X}$ in a graph. Before turning to their efficient construction in the setting of binary videos, we introduce the following taxonomy of vertices based on their connectivity to adjacent levels.

\begin{definition} Let $v$ be a vertex at level $i$ in a merge graph $\mathcal{M}$. We denote its \emph{left-degree} $\ldeg(v)$ and \emph{right-degree} $\rdeg(v)$ as the number of neighbours (in the underlying undirected graph) at levels $i-1$ and $i+1$, respectively. We categorise $v$ as a: 
\begin{itemize} 
\item \textbf{Birth} if $\ldeg(v)=0$, or a \textbf{merging} if $\ldeg(v)>1$.
\item \textbf{Death} if $\rdeg(v)=0$, or a \textbf{splitting} if $\rdeg(v)>1$ . 
\end{itemize} 
The \emph{merging multiplicity} and \emph{splitting multiplicity} of a vertex $v$ are defined by
$
\mathrm{merge}(v) \coloneqq \deg^l(v)-1,
$ and $
\mathrm{split}(v) \coloneqq \deg^r(v)-1.
$ Moreover, vertices at the first (resp. final) level are considered birth (resp. death) vertices. \end{definition}

\begin{remark} The labels ``birth” and ``merging” are mutually exclusive, as are ``death” and ``splitting”. However, these pairs are independent: a vertex may, for instance, be both a birth and a death, or both a merging and a splitting vertex. See Figure~\ref{fig:mon-graph} for an example.

\end{remark}

\begin{figure}[htbp]
\begin{center}
\hspace{-1.3cm}
\begin{tikzpicture}[yscale=-1, scale=1.5, swatch/.style={draw=none, fill=#1, rectangle, minimum size=2.5mm, inner sep=0pt},]
\def\xscale{1pt}
\def\yscale{1pt}
\def\nodesize{1.5pt}
\def\specialsize{2pt}

\def\ytop{1}

\def\yarrow{0.54}
\def\ylinebot{0.7}
\def\ylinetop{3.95}
\def\ylabel{0.42}

\foreach \x in {0,...,9} {
    \draw[dashed] (\x,\ylinebot) -- (\x,\ylinetop);
    \node at (\x,\ylabel) { $\overset{\fpeval{1+\x}}{\bullet}$};
}

\draw[<-, >=stealth, shorten >=4pt, shorten <=4pt] (0,\yarrow) -- (1,\yarrow);
\draw[->, >=stealth, shorten >=4pt, shorten <=4pt] (1,\yarrow) -- (2,\yarrow);
\draw[->, >=stealth, shorten >=4pt, shorten <=4pt] (2,\yarrow) -- (3,\yarrow);
\draw[->, >=stealth, shorten >=4pt, shorten <=4pt] (3,\yarrow) -- (4,\yarrow);
\draw[<-, >=stealth, shorten >=4pt, shorten <=4pt] (4,\yarrow) -- (5,\yarrow);
\draw[<-, >=stealth, shorten >=4pt, shorten <=4pt] (5,\yarrow) -- (6,\yarrow);
\draw[<-, >=stealth, shorten >=4pt, shorten <=4pt] (6,\yarrow) -- (7,\yarrow);
\draw[<-, >=stealth, shorten >=4pt, shorten <=4pt] (7,\yarrow) -- (8,\yarrow);
\draw[->, >=stealth, shorten >=4pt, shorten <=4pt] (8,\yarrow) -- (9,\yarrow);

\pgfkeys{/mynode,
    xs=\xscale, 
    ys=\yscale, 
    isep=2*\nodesize,
    fill=white
  }

\colorlet{birth_color}{my_gold}
\colorlet{split_color}{my_blue}

\colorlet{merge_color}{my_pink}
\colorlet{death_color}{black}

\mynode{name=a_1_1, fill=birth_color, x=0, y=2, xs=\xscale, ys=\yscale, isep=\specialsize}

\mynode[split_color]{name=a_1_2, fill=birth_color, x=0, y=3, xs=\xscale, ys=\yscale, isep=\specialsize}

\mynode{name=b_1_1, fill=white, x=1, y=2, xs=\xscale, ys=\yscale, isep=\nodesize}
\mynode{name=b_1_2, fill=white, x=1, y=2.5, xs=\xscale, ys=\yscale, isep=\nodesize}
\mynode{name=b_1_3, fill=white, x=1, y=3.5, xs=\xscale, ys=\yscale, isep=\nodesize}

\mynode[death_color]{name=b_1_4, fill=birth_color, x=4, y=2, xs=\xscale, ys=\yscale, isep=\specialsize}

\mynode{name=c_1_1, fill=white, x=2, y=2, xs=\xscale, ys=\yscale, isep=\nodesize}
\mynode{name=c_1_2, fill=merge_color, x=2, y=3, xs=\xscale, ys=\yscale, isep=\specialsize}

\mynode{name=d_1_1, fill=merge_color, x=3, y=2.5, xs=\xscale, ys=\yscale, isep=\specialsize}
\mynode{name=d_1_2, fill=birth_color, x=3, y=3., xs=\xscale, ys=\yscale, isep=\specialsize}

\mynode[split_color]{name=e_1_1, fill=merge_color, x=4, y=2.5, xs=\xscale, ys=\yscale, isep=\specialsize}

\mynode{name=f_1_2, fill=white, x=5, y=3, xs=\xscale, ys=\yscale, isep=\nodesize}
\mynode{name=f_1_1, fill=split_color, x=5, y=2, xs=\xscale, ys=\yscale, isep=\specialsize}

\mynode{name=g_1_1, fill=white, x=6, y=1.5, xs=\xscale, ys=\yscale, isep=\nodesize}
\mynode{name=g_1_2, fill=white, x=6, y=2, xs=\xscale, ys=\yscale, isep=\nodesize}
\mynode{name=g_1_3, fill=split_color, x=6, y=3, xs=\xscale, ys=\yscale, isep=\specialsize}

\mynode{name=h_1_1, fill=white, x=7, y=1.5, xs=\xscale, ys=\yscale, isep=\nodesize}
\mynode{name=h_1_2, fill=white, x=7, y=2, xs=\xscale, ys=\yscale, isep=\nodesize}
\mynode{name=h_1_3, fill=white, x=7, y=3, xs=\xscale, ys=\yscale, isep=\nodesize}
\mynode{name=h_1_4, fill=white, x=7, y=3.5, xs=\xscale, ys=\yscale, isep=\nodesize}

\mynode{name=i_1_1, fill=white, x=8, y=2, xs=\xscale, ys=\yscale, isep=\nodesize}
\mynode{name=i_1_2, fill=white, x=8, y=3, xs=\xscale, ys=\yscale, isep=\nodesize}
\mynode{name=i_1_3, fill=white, x=8, y=3.5, xs=\xscale, ys=\yscale, isep=\nodesize}
\mynode{name=i_1_4, fill=white, x=8, y=1.5, xs=\xscale, ys=\yscale, isep=\nodesize}

\mynode[death_color]{name=j_1_1, fill=merge_color, x=9, y=2.5, xs=\xscale, ys=\yscale, isep=\specialsize}
\mynode{name=j_1_2, fill=white, x=9, y=3.5, xs=\xscale, ys=\yscale, isep=\nodesize}

\mynode{name=k_1_1, fill=death_color, x=9, y=3.5, xs=\xscale, ys=\yscale, isep=\specialsize}

\draw[b_arrow] (a_1_1) to (b_1_1);
\draw[b_arrow] (a_1_2) to (b_1_2);
\draw[b_arrow] (a_1_2) to (b_1_3);

\draw[f_arrow] (b_1_1) to (c_1_1);
\draw[f_arrow] (b_1_2) to (c_1_2);
\draw[f_arrow] (b_1_3) to (c_1_2);

\draw[f_arrow] (c_1_1) to (d_1_1);
\draw[f_arrow] (c_1_2) to (d_1_1);

\draw[f_arrow] (d_1_1) to (e_1_1);
\draw[f_arrow] (d_1_2) to (e_1_1);

\draw[b_arrow] (e_1_1) to (f_1_1);
\draw[b_arrow] (e_1_1) to (f_1_2);

\draw[b_arrow] (f_1_1) to (g_1_1);
\draw[b_arrow] (f_1_1) to (g_1_2);
\draw[b_arrow] (f_1_2) to (g_1_3);

\draw[b_arrow] (g_1_1) to (h_1_1);
\draw[b_arrow] (g_1_2) to (h_1_2);
\draw[b_arrow] (g_1_3) to (h_1_3);
\draw[b_arrow] (g_1_3) to (h_1_4);

\draw[b_arrow] (h_1_2) to (i_1_1);
\draw[b_arrow] (h_1_3) to (i_1_2);
\draw[b_arrow] (h_1_4) to (i_1_3);

\draw[f_arrow] (i_1_1) to (j_1_1);
\draw[f_arrow] (i_1_2) to (j_1_1);
\draw[f_arrow] (i_1_3) to (j_1_2);

\draw[b_arrow] (k_1_1) to (j_1_2);

\draw[b_arrow] (h_1_1) to (i_1_4);
\draw[f_arrow] (i_1_4) to (j_1_1);

\node at (1.5,0.83) [below=0.0cm, draw=black!80, fill=white, rounded corners=1pt, inner sep=1pt] {
\setlength{\tabcolsep}{2pt}

\begin{tabular}[colsep=10pt]{ll ll}
\tikz\node[swatch=birth_color]{};  & \footnotesize Birth
& \tikz\node[swatch=merge_color]{};  & \footnotesize Merging \\

\tikz\node[swatch=split_color]{};  & \footnotesize Splitting
& \tikz\node[swatch=death_color]{}; & \footnotesize Death \\
\end{tabular}
};

\def\freemod{3.95+0.8}

\node at (-0.5, \ylabel) {$\tau$:};
\node at (-0.5, 2.5) {$\mathcal{G}$:};

\end{tikzpicture}
\captionsetup{hypcap=false}
\captionof{figure}{%
  \textbf{Merge Graph} $\mathcal{M}$: Nodes with two colours represent dual-type states. Arrows are consistently oriented with the zigzag type $\tau$, and each vertex is the tail of at most one edge in each direction. 
}\label{fig:mon-graph}
\end{center}
\end{figure}
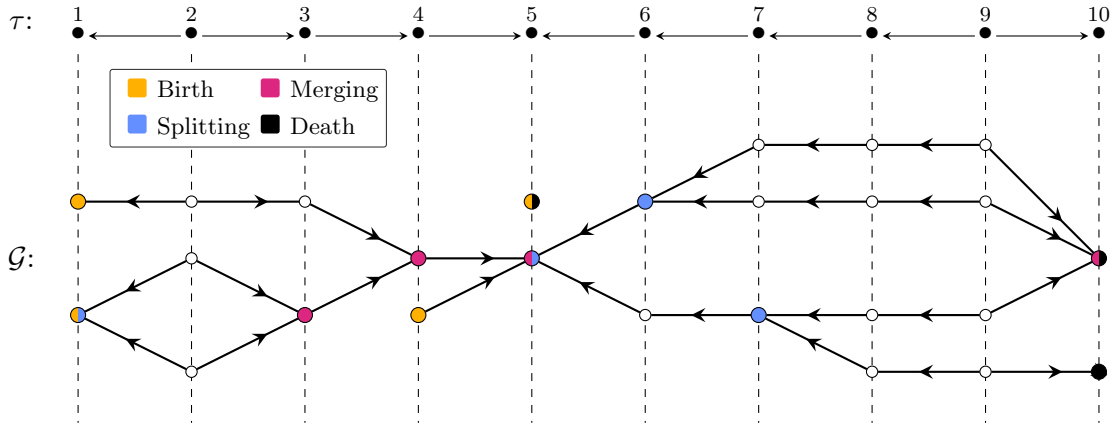

Next, we introduce some common constructions.

\begin{definition}
Let $\mathbb S_1, \mathbb S_2 : \mathcal P_\tau \to \cat{set_{par}}$ be partial formigrams. Their \emph{disjoint union} is the partial formigram
\[
\mathbb S_1 \sqcup \mathbb S_2 : \mathcal P_\tau \to \cat{set_{par}}
\]
defined pointwise by
\[
(\mathbb S_1 \sqcup \mathbb S_2)(i) \coloneq \mathbb S_1(i) \sqcup \mathbb S_2(i),
\]
and for each morphism $e:i\to j$ in $\mathcal P_\tau$ by
\[
(\mathbb S_1 \sqcup \mathbb S_2)(e) \coloneq \mathbb S_1(e) \sqcup \mathbb S_2(e),
\]
where the right-hand side denotes the disjoint union of partial maps.
\end{definition}

\begin{remark}
The disjoint union defined above inherits from the coproduct in $\cat{set_{par}}$, leading to the pointwise definitions. Moreover, this notion carries over to merge graphs. Concretely, for merge graphs $\mathcal M_1, \mathcal M_2$ over $\tau$, we set
\[
\mathcal M_1 \sqcup \mathcal M_2 \coloneq \mathcal G\bigl(\mathcal F(\mathcal M_1)\sqcup \mathcal F(\mathcal M_2)\bigr).
\]
\end{remark}

\begin{center}
\vspace{0.1cm}
\begin{tikzpicture}[yscale=-1, scale=1.5, swatch/.style={draw=none, fill=#1, rectangle, minimum size=2.5mm, inner sep=0pt},
f_arrow/.style={
    base_style,
    postaction={
      decorate,
      decoration={
        markings,
        mark=at position 0.6 with {\arrow[scale=1.5]{stealth}}
      }
    }
  },]
\def\xscale{1pt}
\def\yscale{1pt}
\def\nodesize{1.5pt}
\def\specialsize{2pt}

\def\ytop{1}

\def\yarrow{0.54}
\def\ylinebot{0.7}
\def\ylinetop{3.45}
\def\ylabel{0.42}
\def\upshift{0.05}

\foreach \x in {0,...,4} {
    \draw[dashed] (\x,\ylinebot+\upshift) -- (\x,\ylinetop+\upshift);
    \node at (\x,\ylabel+\upshift) { $\overset{\fpeval{1+\x}}{\bullet}$};
}

\draw[->, >=stealth, shorten >=4pt, shorten <=4pt] (0,\yarrow) -- (1,\yarrow);
\draw[<-, >=stealth, shorten >=4pt, shorten <=4pt] (1,\yarrow) -- (2,\yarrow);
\draw[<-, >=stealth, shorten >=4pt, shorten <=4pt] (2,\yarrow) -- (3,\yarrow);
\draw[<-, >=stealth, shorten >=4pt, shorten <=4pt] (3,\yarrow) -- (4,\yarrow);

\node (a) at (0,1) [circle, draw, fill=white, inner sep=1.5pt] {};
\node (a_2) at (0,2) [circle, draw, fill=white, inner sep=1.5pt] {};

\node (b) at (1,1.5) [circle, draw, fill=white, inner sep=1.5pt] {};

\node (d) at (2,1) [circle, draw, fill=white, inner sep=1.5pt] {};
\node (d_2) at (2,1.5) [circle, draw, fill=white, inner sep=1.5pt] {};
\node (d_3) at (2,2) [circle, draw, fill=white, inner sep=1.5pt] {};





\draw[f_arrow] (a) to (b);
\draw[f_arrow] (a_2) to (b);

\draw[f_arrow] (d) to (b);
\draw[f_arrow] (d_2) to (b);
\draw[f_arrow] (d_3) to (b);




\node (c2_1) at (1,2.75) [circle, draw, fill=white, inner sep=1.5pt] {};

\node (d2_1) at (2,2.75) [circle, draw, fill=white, inner sep=1.5pt] {};
\node (e2_1) at (3,2.25) [circle, draw, fill=white, inner sep=1.5pt] {};
\node (e2_2) at (3,3.25) [circle, draw, fill=white, inner sep=1.5pt] {};

\node (f2_1) at (4,2.25) [circle, draw, fill=white, inner sep=1.5pt] {};
\node (f2_2) at (4,3.25) [circle, draw, fill=white, inner sep=1.5pt] {};

\draw[f_arrow] (d2_1) to (c2_1);
\draw[f_arrow] (e2_1) to (d2_1);
\draw[f_arrow] (e2_2) to (d2_1);

\draw[f_arrow] (f2_1) to (e2_1);
\draw[f_arrow] (f2_2) to (e2_2);

\node at (2, \ylabel-0.4) {$\tau$};

\draw[dotted] (-0.15,0.85) -- (-0.15,2.15);
\draw[dotted] (-0.15,2.15) -- (2.15,2.15);
\draw[dotted] (2.15,2.15) -- (2.15,0.85);
\draw[dotted] (-0.15,0.85) -- (2.15,0.85);

\draw[dotted] 
    (2,2.65) -- (0.9,2.65) -- (0.9,2.85) -- (2,2.85) -- (3, 3.35) -- (4.1, 3.35) -- (4.1, 2.15) -- (3, 2.15) -- cycle;

\node at (-0.65, 1.5) {$\mathcal G(\Sb_1)$};

\node at (4.6, 2.75) {$\mathcal G(\Sb_2)$};

\end{tikzpicture}
\vspace{0.1cm}

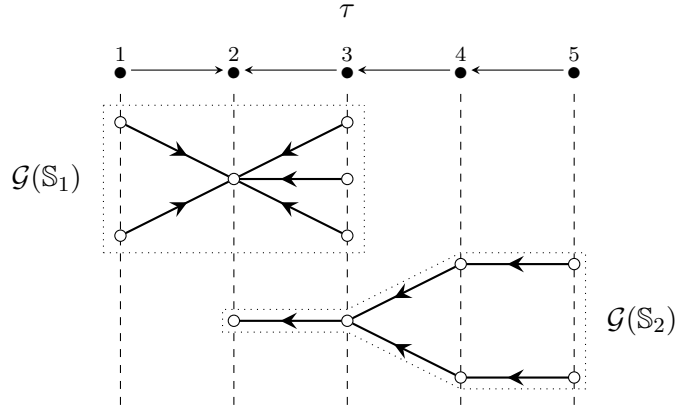
\captionof{figure}{\textbf{Disjoint union of two formigrams $\mathbb{S}_1$ and $\mathbb{S}_2$:}  $\mathcal G(\Sb_1 \sqcup \Sb_2)=\mathcal G(\Sb_1) \sqcup \mathcal G(\Sb_2)$. Moreover, they form the two connected components of this graph.} 
\end{center}

\begin{definition}
Let $\mathcal M$ be a merge graph. A \emph{connected component} of $\mathcal M$ is a connected component of its underlying undirected graph, obtained by forgetting the directions of the edges.
\end{definition}

\begin{remark}
Each connected component $\mathcal C$ of a merge graph $\mathcal M$ determines a subfunctor of $\mathcal F(\mathcal M)$ by restricting to the vertices in $\mathcal C$. Indeed, the partial maps preserve connected components, so the structure maps restrict. In this way, $\mathcal M$ decomposes as the disjoint union of its connected components $\mathcal M = \bigsqcup_i \mathcal C_i$.
\end{remark}

\begin{remark}
The set of connected components of a merge graph $\mathcal M$ can be identified with the colimit $\varinjlim \mathcal F(\mathcal M)$, see Remark~\ref{rmk:cc-as-colim}.
\end{remark}

\subsection{Barcode Graphs}

In this section, we show that a certain type of merge graphs provides a direct way to access interval decompositions. Given a topological diagram $\mathbb X$, we aim to understand the interval decomposition of the associated $H_0$-zigzag module
\[
H_0(\mathbb X;k) \cong \bigoplus_{j=1}^N \mathbb I_{\tau} \langle b_j , d_j \rangle.
\]
We will see that this decomposition can be seen as arising from a special type of merge graphs. In particular, the decomposition is induced by a notion of linearisation of merge graphs. We start by extending the free functor to $\cat{set_{par}}$. We reuse the notation $k[-]$ uniformly for free $k$-linearisations, and ensure that the underlying category will always be clear from context.
\begin{definition}
The \emph{linearisation} assigns to each set $X$ the free vector space
\[
k[X] \coloneq \bigoplus_{x\in X} k\cdot e_x,
\]
and to each partial function $f : X \rightharpoonup Y$ the linear map
$k[f] : k[X] \to k[Y]$ given by
\[
k[f](e_x) =
\begin{cases}
e_{f(x)}, & x \in \mathrm{Dom}(f),\\
0, & x \notin \mathrm{Dom}(f).
\end{cases}
\]
\end{definition}

\begin{lemma}\label{lem:left-adjoint}
The above assignment defines a functor
\[
k[-] : \cat{set_{par}} \to \cat{vect_k},
\]
called the \emph{linearisation functor}. It is 
to the forgetful functor 
\(
U : \cat{vect_k} \to \cat{set_{par}}.
\)
\end{lemma}
\begin{proof}
See Appendix~\ref{app:proof-adjoint}.
\end{proof}

This definition enables us to associate a zigzag module to merge graphs over $\tau$.

\begin{definition}
Let $\mathbb{S}:\mathcal P_\tau\to\cat{set_{par}}$ be a partial formigram. Its \emph{linearisation} is the zigzag module over $\tau$
\[
k[\mathbb S]\coloneq k[-]\circ \mathbb S:\mathcal P_\tau\to\cat{vect_k}.
\]
If $\mathcal M$ is a merge graph over $\tau$, we define its \emph{linearisation} to be
\[
k[\mathcal M]\coloneq k[\mathcal F(\mathcal M)]\in\cat{Zig_\tau},
\]
where $\mathcal F(\mathcal M)$ denotes the associated partial formigram.
\end{definition}

\begin{remark}
Choosing an ordering on the fibres $M_i$ fixes an ordered basis for $k[M_i]$, and we obtain a coordinate isomorphism $k[M_i]\overset{\sim}{\to}k^{|M_i|}$, allowing us to represent the morphisms by matrices.
\end{remark}

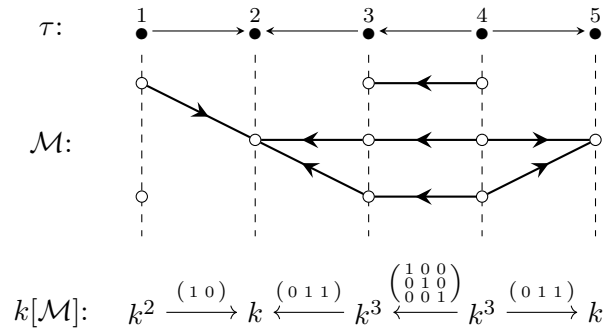
\begin{figure}[htbp]
\begin{center}
\vspace{0.1cm}
\begin{tikzpicture}[yscale=-1, scale=1.5, swatch/.style={draw=none, fill=#1, rectangle, minimum size=2.5mm, inner sep=0pt},
f_arrow/.style={
    base_style,
    postaction={
      decorate,
      decoration={
        markings,
        mark=at position 0.6 with {\arrow[scale=1.5]{stealth}}
      }
    }
  },]
\def\xscale{1pt}
\def\yscale{1pt}
\def\nodesize{1.5pt}
\def\specialsize{2pt}

\def\ytop{1}

\def\yarrow{0.54}
\def\ylinebot{0.7}
\def\ylinetop{2.3}
\def\ylabel{0.42}
\def\upshift{0.05}

\foreach \x in {0,...,4} {
    \draw[dashed] (\x,\ylinebot+\upshift) -- (\x,\ylinetop+\upshift);
    \node at (\x,\ylabel+\upshift) { $\overset{\fpeval{1+\x}}{\bullet}$};
}

\draw[->, >=stealth, shorten >=4pt, shorten <=4pt] (0,\yarrow) -- (1,\yarrow);
\draw[<-, >=stealth, shorten >=4pt, shorten <=4pt] (1,\yarrow) -- (2,\yarrow);
\draw[<-, >=stealth, shorten >=4pt, shorten <=4pt] (2,\yarrow) -- (3,\yarrow);
\draw[->, >=stealth, shorten >=4pt, shorten <=4pt] (3,\yarrow) -- (4,\yarrow);

\node (a) at (0,1) [circle, draw, fill=white, inner sep=1.5pt] {};
\node (a_2) at (0,2) [circle, draw, fill=white, inner sep=1.5pt] {};

\node (b) at (1,1.5) [circle, draw, fill=white, inner sep=1.5pt] {};

\node (d) at (2,1) [circle, draw, fill=white, inner sep=1.5pt] {};
\node (d_2) at (2,1.5) [circle, draw, fill=white, inner sep=1.5pt] {};
\node (d_3) at (2,2) [circle, draw, fill=white, inner sep=1.5pt] {};

\node (e) at (3,1) [circle, draw, fill=white, inner sep=1.5pt] {};
\node (e_2) at (3,1.5) [circle, draw, fill=white, inner sep=1.5pt] {};
\node (e_3) at (3,2) [circle, draw, fill=white, inner sep=1.5pt] {};

\node (f) at (4,1.5) [circle, draw, fill=white, inner sep=1.5pt] {};



\draw[f_arrow] (a) to (b);

\draw[f_arrow] (d_2) to (b);
\draw[f_arrow] (d_3) to (b);

\draw[f_arrow] (e) to (d);
\draw[f_arrow] (e_2) to (d_2);
\draw[f_arrow] (e_3) to (d_3);

\draw[f_arrow] (e_2) to (f);
\draw[f_arrow] (e_3) to (f);


\def\freemod{\ylabel+2.6}

\node (F0) at (0,\freemod) {$\field^2$};

\node (F1) at (1,\freemod) {$\field$};

\draw[->] (F0) -- (F1) node[midway, above] {\scriptsize $\left(\begin{smallmatrix}1&0\end{smallmatrix}\right)$};

\node (F2) at (2,\freemod) {$\field^3$};

\draw[<-] (F1) -- (F2) node[midway, above] {\scriptsize $\left(\begin{smallmatrix}0&1&1\end{smallmatrix}\right)$};

\node (F3) at (3,\freemod) {$\field^3$};

\draw[<-] (F2) -- (F3) node[midway, above] {\scriptsize $\left(\begin{smallmatrix}1&0&0\\0&1&0\\0&0&1\end{smallmatrix}\right)$};

\node (F4) at (4,\freemod) {$\field$};

\draw[->] (F3) -- (F4) node[midway, above] {\scriptsize $\left(\begin{smallmatrix}0&1&1\end{smallmatrix}\right)$};

\node[align=right] at (-0.8, \ylabel+0.1) {$\tau$:};
\node[align=right] at (-0.8, \ylabel+1.1) {$\mathcal{M}$:};
\node[align=right] at (-0.8, \freemod) {$k[\mathcal{M}]$:};

\end{tikzpicture}
\vspace{0.1cm}
\captionof{figure}{\textbf{Linearisation $k[\mathcal M]$ of a Merge graph $\mathcal M$}. Note that an ordering on the fibres (top to bottom) was chosen to get the matrix representations.} 
\end{center}
\end{figure}

The linearisation behaves nicely with respect to disjoint unions.

\begin{proposition}\label{prop:com-direct-sum}
Let $\mathcal M_1$ and $\mathcal M_2$ be merge graphs over $\tau$. Then 
$$k[\mathcal M_1\sqcup\mathcal M_2] \cong k[\mathcal M_1] \oplus k[\mathcal M_2]. $$
\end{proposition}
\begin{proof}
By the definition of the disjoint union of merge graphs,
\(
\mathcal M_1\sqcup \mathcal M_2
=
\mathcal G\bigl(\mathcal F(\mathcal M_1)\sqcup \mathcal F(\mathcal M_2)\bigr),
\)
and hence
\(
k[\mathcal M_1\sqcup \mathcal M_2]
=
k\bigl[\mathcal F(\mathcal M_1)\sqcup \mathcal F(\mathcal M_2)\bigr].
\)
Since the linearisation functor
\(
k[-]:\cat{set_{par}}\to\cat{vect_k}
\)
is left adjoint to the forgetful functor, it preserves coproducts. Therefore
\[
k\bigl[\mathcal F(\mathcal M_1)\sqcup \mathcal F(\mathcal M_2)\bigr]
\cong
k[\mathcal F(\mathcal M_1)] \oplus k[\mathcal F(\mathcal M_2)],
\]
where $\oplus$ is the coproduct in $\cat{vect_k}$. By definition of the linearisation of a merge graph, this is
\[
k[\mathcal M_1\sqcup \mathcal M_2]
\cong
k[\mathcal M_1]\oplus k[\mathcal M_2],
\]
as claimed.
\end{proof}

We now introduce the special type of merge graphs that induce an interval decomposition. We start by identifying distinguished connected components.

\begin{definition}
A connected component $\mathcal B$ of a merge graph $\mathcal M$ is called a \emph{bar component} if it has no merging or splitting vertices.
\end{definition}

\begin{remark}
Let $\mathcal B$ be a bar component of $\mathcal M$. Since $\mathcal B$ is connected and contains no merging or splitting vertices, there exist unique vertices with minimal level $b$ and maximal level $d$. We will therefore write $\mathcal B = \mathcal B\langle b,d\rangle$. 
\end{remark}

\begin{definition}
A merge graph $\mathcal M$ over $\tau$ is a \emph{barcode graph} if it is the disjoint union of bar components
$$\mathcal M = \bigsqcup_i \mathcal B\langle b_i, d_i \rangle.$$
\end{definition}

\begin{remark}
Equivalently, barcode graphs are merge graphs without splitting or merging vertices.

\begin{center}
\begin{tikzpicture}[yscale=-1, scale=1.5, swatch/.style={draw=none, fill=#1, rectangle, minimum size=2.5mm, inner sep=0pt},f_arrow/.style={
    base_style,
    postaction={
      decorate,
      decoration={
        markings,
        mark=at position 0.55 with {\arrow[scale=1.5]{stealth}}
      }
    }
  },]
\def\xscale{1pt}
\def\yscale{1pt}
\def\nodesize{1.5pt}
\def\specialsize{2pt}

\def\ytop{1}

\def\yarrow{0.54}
\def\ylinebot{0.7}
\def\ylinetop{2.5}
\def\ylabel{0.42}

\def\xsep{1.5}

\foreach \x in {0,...,3} {
    \draw[dashed] ({\xsep*\x},\ylinebot) -- ({\xsep*\x},\ylinetop);
    \node at ({\xsep*\x},\ylabel) { $\overset{\fpeval{1+\x}}{\bullet}$};
}

\draw[<-, >=stealth, shorten >=4pt, shorten <=4pt] (0,\yarrow) -- ({\xsep*1},\yarrow);
\draw[->, >=stealth, shorten >=4pt, shorten <=4pt] ({\xsep*1},\yarrow) -- ({\xsep*2},\yarrow);
\draw[->, >=stealth, shorten >=4pt, shorten <=4pt] ({\xsep*2},\yarrow) -- ({\xsep*3},\yarrow);

\pgfkeys{/mynode,
    xs=\xscale, 
    ys=\yscale, 
    isep=2*\nodesize,
    fill=white
  }

\mynode{name=a_1_1, fill=white, x=0, y=1, xs=\xscale, ys=\yscale, isep=\nodesize}

\mynode{name=b_1_1, fill=white, x={\xsep*1}, y=1, xs=\xscale, ys=\yscale, isep=\nodesize}
\mynode{name=b_1_2, fill=white, x={\xsep*1}, y=1.5, xs=\xscale, ys=\yscale, isep=\nodesize}
\mynode{name=b_1_3, fill=white, x={\xsep*1}, y=2, xs=\xscale, ys=\yscale, isep=\nodesize}

\mynode{name=c_1_1, fill=white, x={\xsep*2}, y=1, xs=\xscale, ys=\yscale, isep=\nodesize}
\mynode{name=c_1_2, fill=white, x={\xsep*2}, y=1.5, xs=\xscale, ys=\yscale, isep=\nodesize}
\mynode{name=c_1_3, fill=white, x={\xsep*2}, y=2, xs=\xscale, ys=\yscale, isep=\nodesize}

\mynode{name=d_1_1, fill=white, x={\xsep*3}, y=1, xs=\xscale, ys=\yscale, isep=\nodesize}
\mynode{name=d_1_3, fill=white, x={\xsep*3}, y=2, xs=\xscale, ys=\yscale, isep=\nodesize}

\draw[f_arrow] (b_1_1) to (a_1_1);
\draw[f_arrow] (b_1_1) to (c_1_1);
\draw[f_arrow] (c_1_1) to (d_1_1);

\draw[f_arrow] (b_1_2) to (c_1_2);

\draw[f_arrow] (b_1_3) to (c_1_3);
\draw[f_arrow] (c_1_3) to (d_1_3);









\end{tikzpicture}
\end{center}
\end{remark}

\vspace{0.5cm}

\begin{proposition}\label{prop:lin-barcode}
Let $\mathcal M=\bigsqcup_i \mathcal B\langle b_i, d_i \rangle$ be a barcode graph over $\tau$. Then there exists an isomorphism
$$k[\mathcal M] = k\left[\bigsqcup_i \mathcal B\langle b_i, d_i \rangle\right]\cong \bigoplus_i \mathbb I_\tau \langle b_i, d_i \rangle. $$
\end{proposition}
\begin{proof}
By Proposition~\ref{prop:com-direct-sum}, we have that $$ k\left[\bigsqcup_i \mathcal B\langle b_i, d_i \rangle\right] \cong \bigoplus_i k[\mathcal B\langle b_i, d_i \rangle]. $$
It therefore suffices to show that for each bar component $\mathcal B\langle b_i, d_i \rangle$, there is an isomorphism
\[
k[\mathcal B\langle b_i, d_i \rangle] \cong \mathbb I_\tau \langle b_i, d_i \rangle.
\]
But $\mathcal B\langle b_i, d_i \rangle$ consists of a single vertex at each level $\ell \in \{b_i,\dots,d_i\}$, and no vertices otherwise. Hence, its linearisation assigns the one-dimensional vector space $k$ at each level $\ell \in \{b_i,\dots,d_i\}$ and $0$ elsewhere, with identity maps along the interval and zero maps outside. This is precisely the interval module $\mathbb I_\tau \langle b_i, d_i \rangle$. Combining these isomorphisms yields the claim.
\end{proof}

This proposition shows that barcode graphs provide a combinatorial model for interval-decomposed zigzag modules: the interval decomposition can be read off directly from the graph. We are missing one last ingredient to link back to zigzag modules.

Let $\M$ be a merge graph and $\V$ be a zigzag module over $\tau$. Suppose there exists an isomorphism $\psi : k[\mathcal M] \xrightarrow{\sim} \mathbb V$. Then we have that the set
\[
\{ \psi(e_v) \mid v \in M_i \}
\]
forms a basis of $\mathbb V(i)$ for each level $i\in\tau$, where $(e_v)_{v\in M_i}$ denotes the canonical basis of $k[M_i]$. This motivates the following definition.

\begin{definition}\label{def:realising-basis}
Let $\mathbb V$ be a zigzag module over $\tau$. A merge graph $\mathcal M$ over $\tau$ is said to \emph{realise a basis} of $\mathbb V$ if there is an isomorphism
\[
k[\mathcal M] \cong \mathbb V.
\]
\end{definition}

\begin{example}
Given a topological diagram $\mathbb X$, the merge graph encoding the evolution of its connected components $\mathcal G(\pi_0(\mathbb X))$ realises a basis of $H_0(\mathbb{X})$. This follows directly from the fact that $H_0 = k[-]\circ \pi_0$.
\end{example}

\begin{proposition}\label{prop:barcode-from-graph}
If a barcode graph $\mathcal M = \bigsqcup_i \mathcal B\langle b_i, d_i \rangle$  over $\tau$ realises a basis of a zigzag module $\V$, we have that
$$\mathrm{Bar}(\V) = \{\langle  b_i , d_i\rangle \}_i. $$ 
\end{proposition}
\begin{proof}
This is an immediate consequence of Definition~\ref{def:realising-basis} together with Proposition~\ref{prop:lin-barcode}.
\end{proof}

\subsection{From Frames to Formigrams}\label{sec:frames-to-formi}

Before we move on to show how merge graphs aid the interval decomposition of $H_0$-zigzag modules in Section~\ref{sec:int-decomp}, we first present an efficient algorithm for constructing merge graphs in the context of binary videos. Therefore, let $\mathcal{V}=(I_i)_{i=1}^n$ be a binary video of resolution \((w,h)\), with common foreground connectivity \(\kappa\in\{4,8\}\). Recall that Definition~\htest{def:top-constr} associates to \(\mathcal{V}\) the topological diagrams $
\mathbb{F}^\cup_{\mathcal{V}},\
\mathbb{F}^\cap_{\mathcal{V}},\
\mathbb{B}^\cup_{\mathcal{V}},\
$
and 
$
\mathbb{B}^\cap_{\mathcal{V}}
$
via insertion of interpolation images \(I_i\vee I_{i+1}\) and \(I_i\wedge I_{i+1}\).


We describe an algorithm that constructs the merge graph $\mathcal{G}(\pi_0(\mathbb{F}^\cup_{\mathcal{V}}))$ of the union construction, but it can be easily adapted to the other three topological diagrams. The runtime is linear $\mathcal{O}(whn)$ in the total number of pixels and \textit{embarrassingly parallel} across the frames of the video, meaning that each image can be processed independently, allowing for straightforward parallelisation. The algorithm proceeds in three steps.

\noindent\textbf{Step 1 (Construction of Interpolation Frames).} In a first step, we build the interpolation frames $I_i \vee I_{i+1}$ via a pointwise maximum. Since this is a pixelwise operation, all such images can be computed in time $\mathcal{O}(nwh)$. This leads to the topological diagram $\mathbb{F}^\cup_{\mathcal{V}}$. For $k\in \{1,\dots, 2n-1 \}$, we have
\[
\mathbb{F}^\cup_{\mathcal{V}}(k) = \begin{cases}
    F_{I_i}, & k=2i-1,\\
    F_{I_i\vee I_{i+1}}, & k=2i.
\end{cases}
\]

\noindent\textbf{Step 2 (Vertex Set).} The second step builds the vertex sets of $\mathcal{G}(\pi_0(\mathbb{F}^\cup_{\mathcal{V}}))$. For $k\in \{1,\dots, 2n-1 \}$, the vertex set at $k$ is in bijection to $\pi_0(\mathbb{F}^\cup_{\mathcal{V}}(k))$. By reducing to local adjacency on the pixel grid $\Omega_{w,h}=\{1,\dots,w\} \times \{1,\dots, h \}$, we can compute the connected components $\pi_0(\mathbb{F}^\cup_{\mathcal{V}}(k))$ without constructing the cubical complexes.

\begin{definition}
Two pixels $p=(i,j)$ and $p'=(i',j')$ in $\Omega_{w,h}$ are \emph{4-adjacent} if $|i-i'|+|j-j'|=1$, and \emph{8-adjacent} if $\max\{|i-i'|, |j-j'|\}=1$. For $\kappa\in \{4,8 \}$, we denote by $N_\kappa(p)$ the set of all pixels that are $\kappa$-adjacent to $p$.
\end{definition}

\begin{definition}
Let $I:\Omega_{w,h}\to\{0,1\}$ be a binary image with foreground connectivity $\kappa\in\{4,8\}$. Two pixels $p,q\in I^{-1}(1)$ are \emph{$\kappa$-connected} if they can be joined by a chain of $\kappa$-adjacent pixels in $I^{-1}(1)$. The equivalence classes are called the \emph{$\kappa$-connected components} of $I^{-1}(1)$ and denoted by $\mathcal{C}_\kappa(I)$.
\end{definition}

\begin{lemma}
There is a bijection
\[
\mathcal{C}_\kappa(I) \cong \pi_0(F_I).
\]
\end{lemma}
\begin{proof}
This follows directly from the definitions of the $T$- and $V$-constructions, since intersections of cells correspond precisely to $\kappa$-adjacency.
\end{proof}

In practice, the set $\mathcal{C}_\kappa(I)$ is computed via a standard connected component labelling (CCL) algorithm~\cite{rosenfeldSequentialOperationsDigital1966,heConnectedcomponentLabelingProblem2017}. Given a binary image $I$ with foreground connectivity $\kappa_I$, the output of CCL is two-fold. Let $m_I=|\mathcal{C}_{\kappa_I}(I)|$. The first output is a labelling map
\[
L_{I}:\Omega_{w,h}\to\mathbb{N}_0,
\]
which assigns to each foreground pixel $p$ an index $L_{I}(p)\in\{ 1,\dots, m_I\}$ of its $\kappa_I$-connected component, and $0$ to background pixels. The collection of sets $L_{I}^{-1}(j)$ with $j\in\{1,\dots, m_I \}$ recovers the equivalence classes $\mathcal{C}_{\kappa_I}(I)$. The second output is a collection of representative pixels 
\[
P_{I}=\{p_1,\dots,p_{m_I}\}
\]
with $p_j\in\Omega_{w,h}$ such that $L_{I}(p_j)=j$ for each $j\in\{1,\dots,m_I\}$. This records a single pixel per connected component.

We describe a straightforward implementation based on depth-first search in Algorithm~\htest{algo:CCL} and provide an Example in Figure~\htest{fig:CCL}.

\begin{figure}[h]
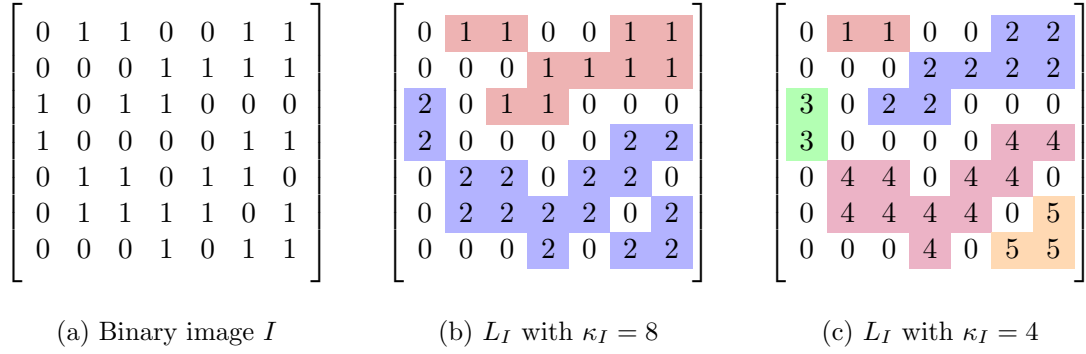
 
\centering
\captionsetup[subfigure]{justification=centering}
\begin{subfigure}{0.32\textwidth}
\centering
\[
\left[
\vphantom{\scalebox{1.085}{$\begin{array}{c}0\\0\\0\\0\\0\\0\\0\end{array}$}}
\begin{array}{ccccccc}
0 & 1 & 1 & 0 & 0 & 1 & 1 \\
0 & 0 & 0 & 1 & 1 & 1 & 1 \\
1 & 0 & 1 & 1 & 0 & 0 & 0 \\
1 & 0 & 0 & 0 & 0 & 1 & 1 \\
0 & 1 & 1 & 0 & 1 & 1 & 0 \\
0 & 1 & 1 & 1 & 1 & 0 & 1 \\
0 & 0 & 0 & 1 & 0 & 1 & 1
\end{array}
\right]
\]
\caption{Binary image $I$}
\end{subfigure}
\hfill
\begin{subfigure}{0.32\textwidth}
\centering
\[
\left[
\vphantom{\scalebox{1.085}{$\begin{array}{c}0\\0\\0\\0\\0\\0\\0\end{array}$}}
\begin{array}{ccccccc}
0 & \cellcolor{red!30}1 & \cellcolor{red!30}1 & 0 & 0 & \cellcolor{red!30}1 & \cellcolor{red!30}1 \\
0 & 0 & 0 & \cellcolor{red!30}1 & \cellcolor{red!30}1 & \cellcolor{red!30}1 & \cellcolor{red!30}1 \\
\cellcolor{blue!30}2 & 0 & \cellcolor{red!30}1 & \cellcolor{red!30}1 & 0 & 0 & 0 \\
\cellcolor{blue!30}2 & 0 & 0 & 0 & 0 & \cellcolor{blue!30}2 & \cellcolor{blue!30}2 \\
0 & \cellcolor{blue!30}2 & \cellcolor{blue!30}2 & 0 & \cellcolor{blue!30}2 & \cellcolor{blue!30}2 & 0 \\
0 & \cellcolor{blue!30}2 & \cellcolor{blue!30}2 & \cellcolor{blue!30}2 & \cellcolor{blue!30}2 & 0 & \cellcolor{blue!30}2 \\
0 & 0 & 0 & \cellcolor{blue!30}2 & 0 & \cellcolor{blue!30}2 & \cellcolor{blue!30}2
\end{array}
\right]
\]
\caption{$L_{I}$ with $\kappa_I=8$}
\end{subfigure}
\hfill
\begin{subfigure}{0.32\textwidth}
\centering
\[
\left[
\vphantom{\scalebox{1.085}{$\begin{array}{c}0\\0\\0\\0\\0\\0\\0\end{array}$}}
\begin{array}{ccccccc}
0 & \cellcolor{red!30}1 & \cellcolor{red!30}1 & 0 & 0 & \cellcolor{blue!30}2 & \cellcolor{blue!30}2 \\
0 & 0 & 0 & \cellcolor{blue!30}2 & \cellcolor{blue!30}2 & \cellcolor{blue!30}2 & \cellcolor{blue!30}2 \\
\cellcolor{green!30}3 & 0 & \cellcolor{blue!30}2 & \cellcolor{blue!30}2 & 0 & 0 & 0 \\
\cellcolor{green!30}3 & 0 & 0 & 0 & 0 & \cellcolor{purple!30}4 & \cellcolor{purple!30}4 \\
0 & \cellcolor{purple!30}4 & \cellcolor{purple!30}4 & 0 & \cellcolor{purple!30}4 & \cellcolor{purple!30}4 & 0 \\
0 & \cellcolor{purple!30}4 & \cellcolor{purple!30}4 & \cellcolor{purple!30}4 & \cellcolor{purple!30}4 & 0 & \cellcolor{orange!30}5 \\
0 & 0 & 0 & \cellcolor{purple!30}4 & 0 & \cellcolor{orange!30}5 & \cellcolor{orange!30}5
\end{array}
\right]
\]
\caption{$L_{I}$ with $\kappa_I=4$}
\end{subfigure}
\caption{\textbf{Connected component labeling (CCL)} of a binary image $I$. The two choices $\kappa\in\{4,8\}$ lead to distinct $\kappa$-connected components $\mathcal{C}_\kappa(I)$.}\label{fig:CCL}
\end{figure}

\noindent The vertex set at level $k\in\{1,\dots,2n-1\}$, i.e.\ the fibre over $k$ of $\mathcal{G}(\pi_0(\mathbb{F}^\cup_{\mathcal{V}}))$, is then simply given by the labels
\[
\lvl^{-1}(k) = 
\begin{cases}
\{1,\dots,m_{I_i}\}, & k=2i-1,\\
\{1,\dots,m_{I_i\vee I_{i+1}}\}, & k=2i.
\end{cases}
\]
where again $m_I \coloneq |\mathcal{C}_{\kappa_I}(I)|$.

\noindent\textbf{Step 3 (Edge Set).} It remains to determine the edges via the induced maps
\[
\pi_0(F_{I_i}) \to \pi_0(F_{I_i\vee I_{i+1}}) \leftarrow \pi_0(F_{I_{i+1}}).
\]
By Lemma~\ref{lem:inter-uni}, these maps are induced by inclusions. Hence each connected component of $I_i$ or $I_{i+1}$ is contained in a unique connected component of $I_i \vee I_{i+1}$, so it suffices to track a single representative pixel per component.

Let $P_{I_i}=\{p_1,\dots,p_{m_{I_i}}\}$ and $P_{I_{i+1}}=\{q_1,\dots,q_{m_{I_{i+1}}}\}$ be the representative pixels returned by CCL. The induced maps are computed by evaluating the labelling of the union image. Accordingly, the edge sets between level $2i-1$ and $2i$ are given by
\[
E_{2i-1,\,2i}
=
\{\, (j,\,L_{I_i\vee I_{i+1}}(p_j)) \mid j=1,\dots,m_{I_i} \,\},
\]
and
\[
E_{2i+1,\,2i}
=
\{\, (j,\,L_{I_i\vee I_{i+1}}(q_j)) \mid j=1,\dots,m_{I_{i+1}} \,\}.
\]
These define the directed edges between the vertex sets at levels $2i-1$, $2i$, and $2i+1$. Since each representative pixel lies in the foreground of $I_i\vee I_{i+1}$ due to the aforementioned inclusions, the labels are nonzero, and the edges are well-defined.
Repeating this for each $i\in\{1,\dots,n-1\}$ yields the merge graph. The full algorithm is given in Algorithm~\ref{algo:formi2}.

\noindent\textbf{Complexity.}
For a binary video $\mathcal{V}=(I_i)_{i=1}^n$ of resolution $(w,h)$, Algorithm~\htest{algo:formi2} runs in $\mathcal{O}(nwh)$ time. Indeed, Step~1 computes the interpolation frames in $\mathcal{O}(nwh)$ time. Step~2 applies CCL to $2n-1$ images, each in linear time, yielding $\mathcal{O}(nwh)$. Step~3 processes each connected component once, which is bounded by the number of pixels, and hence also $\mathcal{O}(nwh)$. The overall complexity is therefore linear in the total number of pixels.

 Moreover, the runtime is dominated by the CCL step, which is \emph{embarrassingly parallel} across images, meaning that each image can be processed independently. This enables efficient parallel implementations; we explore the resulting speed-up in Section~\ref{sec:benchmarking}.

\begin{center}
\begin{tikzpicture}
\hspace{-1cm}
\def\dx{2.17}
\foreach \i in {1,...,7} {
\node at ({\dx*(\i-1)},0) {\includegraphics[width=1.7cm]{images/colored_circles/circle_\i.png}};
}

\foreach \i in {1,...,7} {
    \draw[dotted, gray] ({\dx*(\i-1)}, -1.8) -- ({\dx*(\i-1)}, -5.6);
}

\foreach \i/\label in {
    1/$I_1$,
    2/$I_1 \vee I_2$,
    3/$I_2$,
    4/$I_2 \vee I_3$,
    5/$I_3$,
    6/$I_3 \vee I_4$,
    7/$I_4$
} {
    \node at ({\dx*(\i-1)}, 1.2) {\label};
}

\foreach \i/\label in {
    1/$F_{I_1}$,
    2/$F_{I_1 \vee I_2}$,
    3/$F_{I_2}$,
    4/$F_{I_2 \vee I_3}$,
    5/$F_{I_3}$,
    6/$F_{I_3 \vee I_4}$,
    7/$F_{I_4}$
} {
    \node (F\i) at ({\dx*(\i-1)}, -1.3) {\label};
}

\draw[{Hooks[right]}->] (F1) -- (F2);
\draw[{Hooks[left]}->] (F3) -- (F2);

\draw[{Hooks[right]}->] (F3) -- (F4);
\draw[{Hooks[left]}->] (F5) -- (F4);

\draw[{Hooks[right]}->] (F5) -- (F6);
\draw[{Hooks[left]}->] (F7) -- (F6);

\node at (-1.5,-1.3) {$\mathbb{F}^\cup_\mathcal{V}$:};
\node at (-1.5,0) {$\mathcal{V}$:};
\node at (-1.5,-3.5) {$\pi_0(\mathbb{F}^\cup_\mathcal{V})$:};

\def\xscale{1pt}
\def\yscale{1pt}
\def\nodesize{1.5pt}
\def\specialsize{3pt}

\def\ytop{1}

\mynode{name=a_1_1, fill=my_violet, x=0, y=-2.5, xs=\xscale, ys=\yscale, isep=\specialsize}
\mynode{name=a_1_2, fill=my_teal, x=0, y=-5, xs=\xscale, ys=\yscale, isep=\specialsize}

\mynode{name=b_1_1, fill=my_teal, x={1*\dx}, y=-3.75, xs=\xscale, ys=\yscale, isep=\specialsize}

\mynode{name=c_1_1, fill=my_teal, x={2*\dx}, y=-3.75, xs=\xscale, ys=\yscale, isep=\specialsize}

\mynode{name=d_1_1, fill=my_teal, x={3*\dx}, y=-3.75, xs=\xscale, ys=\yscale, isep=\specialsize}

\mynode{name=e_1_1, fill=my_actual_red, x={4*\dx}, y=-2.5, xs=\xscale, ys=\yscale, isep=\specialsize}
\mynode{name=e_1_2, fill=my_teal, x={4*\dx}, y=-3.75, xs=\xscale, ys=\yscale, isep=\specialsize}
\mynode{name=e_1_3, fill=my_mustard, x={4*\dx}, y=-5, xs=\xscale, ys=\yscale, isep=\specialsize}

\mynode{name=f_1_1, fill=my_teal, x={5*\dx}, y=-2.5, xs=\xscale, ys=\yscale, isep=\specialsize}
\mynode{name=f_1_2, fill=my_mustard, x={5*\dx}, y=-5, xs=\xscale, ys=\yscale, isep=\specialsize}

\mynode{name=g_1_1, fill=my_teal, x={6*\dx}, y=-2.5, xs=\xscale, ys=\yscale, isep=\specialsize}
\mynode{name=g_1_2, fill=my_mustard, x={6*\dx}, y=-5, xs=\xscale, ys=\yscale, isep=\specialsize}

\draw[b_arrow] (b_1_1) to (a_1_1);
\draw[b_arrow] (b_1_1) to (a_1_2);

\draw[b_arrow] (b_1_1) to (c_1_1);

\draw[b_arrow] (d_1_1) to (c_1_1);

\draw[b_arrow] (d_1_1) to (e_1_1);
\draw[b_arrow] (d_1_1) to (e_1_2);
\draw[b_arrow] (d_1_1) to (e_1_3);

\draw[b_arrow] (f_1_1) to (e_1_1);
\draw[b_arrow] (f_1_1) to (e_1_2);
\draw[b_arrow] (f_1_2) to (e_1_3);

\draw[b_arrow] (f_1_1) to (g_1_1);
\draw[b_arrow] (f_1_2) to (g_1_2);
  
\end{tikzpicture}

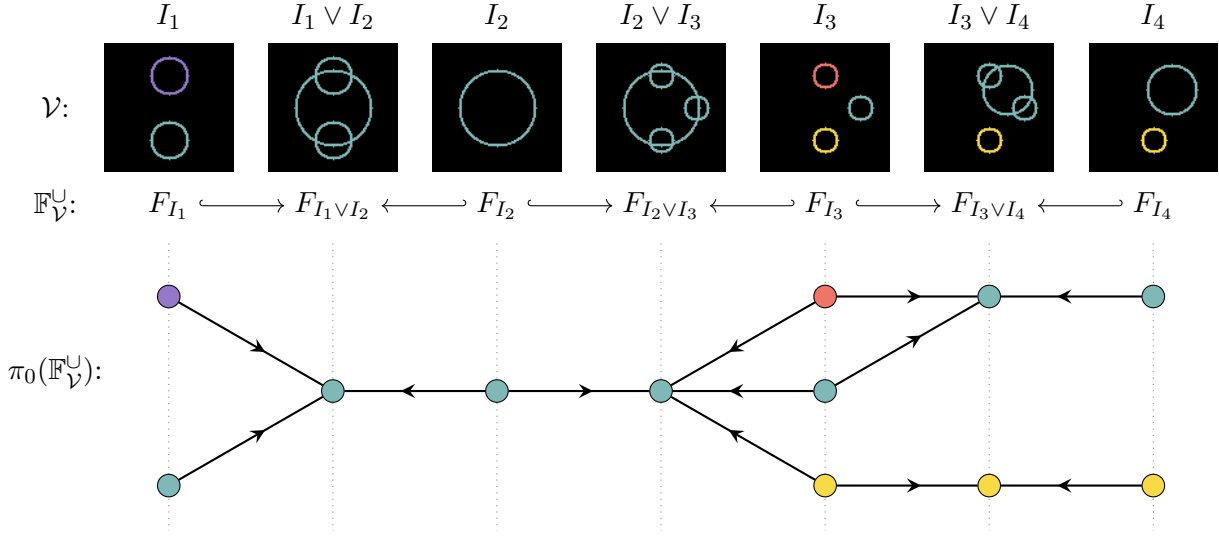
\captionof{figure}{\textbf{Merge graph obtained from a binary video $\mathcal{V}$}: Interpolation frames $I_i\vee I_{i+1} = \max(I_i,I_{i+1})$ are inserted, inducing the topological diagram $\mathbb{F}^\cup_\mathcal{V}$. Vertices correspond to connected components; edges encode their evolution under inclusion.} \label{fig:formi-union}
\end{center}

\section{Interval Decomposition Algorithm} \label{sec:int-decomp}

Given a topological diagram $\mathbb X:\mathcal P_\tau \to \cat{top}$, we turn to the problem of efficiently computing the barcode of its associated $H_0$-zigzag module $$\mathrm{Bar}(H_0(\mathbb X;k))= \{ \langle  b_j, d_j \rangle\}_{j=1}^N.$$ Proposition~\ref{prop:barcode-from-graph} reduces the problem to a combinatorial one: namely, to construct a barcode graph $$\mathcal M^{\mathrm{bar}} = \bigsqcup_i \mathcal B \langle b_i , d_i \rangle $$ over $\tau$ realising a basis of $H_0(\mathbb X)$, in the sense that 
$$k[\mathcal M^{\mathrm{bar}}] \cong H_0(\mathbb X;k).$$
By Proposition~\ref{prop:barcode-from-graph}, the barcode of $H_0(\mathbb X;k)$ is then fully determined by the endpoints of the bar components of $\mathcal M^{\mathrm{bar}}$, and by Theorem~\ref{thm:coeff} independent of the chosen coefficient field $k$. To construct such a graph efficiently, we make use of the state-of-the-art algorithm of Dey and Hou~\cite{deyComputingZigzagPersistence2021}, which applies under a structural assumption on the underlying formigram $\pi_0(\mathbb X)$.

\begin{definition}
A formigram $\Sb : \mathcal P _\tau \to \cat{set}$ is called \emph{elementary} if $S_1=S_n=\emptyset$, and all structure maps between consecutive levels $S[i\to j]:S_i\to S_j$ are of the following forms:
\begin{itemize}
\item injective with $|S_j|= |S_i|+1$; 
\item surjective with $|S_i|= |S_j|+1$; 
\item bijective. 
\end{itemize}
\end{definition}

\begin{remark}
In particular, the merge graph $\mathcal G(\Sb)$ of an elementary formigram $\Sb$ has at most one birth or death event between consecutive levels, and all merging and splitting multiplicities are at most one.
\end{remark}

\begin{definition} \label{def:elem-zigzag-mod}
Let $\mathbb X$ be a topological diagram. We say that $H_0(\mathbb X)$ is an elementary zigzag module if $\pi_0(\mathbb X)$ is an elementary formigram.
\end{definition}


\begin{center}
\begin{tikzpicture}[yscale=-1, scale=1.5, swatch/.style={draw=none, fill=#1, rectangle, minimum size=2.5mm, inner sep=0pt},f_arrow/.style={
    base_style,
    postaction={
      decorate,
      decoration={
        markings,
        mark=at position 0.55 with {\arrow[scale=1.3]{stealth}}
      }
    }
  },]
\def\xscale{1pt}
\def\yscale{1pt}
\def\nodesize{1.5pt}
\def\specialsize{2pt}

\def\ytop{1}

\def\yarrow{0.54}
\def\ylinebot{0.7}
\def\ylinetop{2.75}
\def\ylabel{0.42}

\def\xsep{0.75}

\foreach \x in {0,...,6} {
    \draw[dashed] ({\xsep*\x},\ylinebot) -- ({\xsep*\x},\ylinetop);
    \node at ({\xsep*\x},\ylabel) { $\overset{\fpeval{1+\x}}{\bullet}$};
}

\draw[->, >=stealth, shorten >=4pt, shorten <=4pt] (0,\yarrow) -- ({\xsep*1},\yarrow);
\draw[<-, >=stealth, shorten >=4pt, shorten <=4pt] ({\xsep*1},\yarrow) -- ({\xsep*2},\yarrow);
\draw[<-, >=stealth, shorten >=4pt, shorten <=4pt] ({\xsep*2},\yarrow) -- ({\xsep*3},\yarrow);
\draw[<-, >=stealth, shorten >=4pt, shorten <=4pt] ({\xsep*3},\yarrow) -- ({\xsep*4},\yarrow);
\draw[->, >=stealth, shorten >=4pt, shorten <=4pt] ({\xsep*4},\yarrow) -- ({\xsep*5},\yarrow);
\draw[->, >=stealth, shorten >=4pt, shorten <=4pt] ({\xsep*5},\yarrow) -- ({\xsep*6},\yarrow);

\pgfkeys{/mynode,
    xs=\xscale, 
    ys=\yscale, 
    isep=2*\nodesize,
    fill=white
  }

\begin{scope}
\mynode{name=a_1, fill=white, x=0, y=1.5, xs=\xscale, ys=\yscale, isep=\nodesize}
\mynode{name=a_2, fill=white, x=0, y=2.5, xs=\xscale, ys=\yscale, isep=\nodesize}

\mynode{name=b_1, fill=white, x={\xsep*1}, y=2, xs=\xscale, ys=\yscale, isep=\nodesize}

\mynode{name=c_1, fill=white, x={\xsep*2}, y=2, xs=\xscale, ys=\yscale, isep=\nodesize}

\mynode{name=d_1, fill=white, x={\xsep*3}, y=1.5, xs=\xscale, ys=\yscale, isep=\nodesize}
\mynode{name=d_2, fill=white, x={\xsep*3}, y=2, xs=\xscale, ys=\yscale, isep=\nodesize}
\mynode{name=d_3, fill=white, x={\xsep*3}, y=2.5, xs=\xscale, ys=\yscale, isep=\nodesize}

\mynode{name=e_1, fill=white, x={\xsep*4}, y=1, xs=\xscale, ys=\yscale, isep=\nodesize}
\mynode{name=e_2, fill=white, x={\xsep*4}, y=1.5, xs=\xscale, ys=\yscale, isep=\nodesize}
\mynode{name=e_3, fill=white, x={\xsep*4}, y=2, xs=\xscale, ys=\yscale, isep=\nodesize}
\mynode{name=e_4, fill=white, x={\xsep*4}, y=2.5, xs=\xscale, ys=\yscale, isep=\nodesize}

\mynode{name=f_1, fill=white, x={\xsep*5}, y=1, xs=\xscale, ys=\yscale, isep=\nodesize}
\mynode{name=f_2, fill=white, x={\xsep*5}, y=2, xs=\xscale, ys=\yscale, isep=\nodesize}
\mynode{name=f_3, fill=white, x={\xsep*5}, y=2.5, xs=\xscale, ys=\yscale, isep=\nodesize}

\mynode{name=g_1, fill=white, x={\xsep*6}, y=2, xs=\xscale, ys=\yscale, isep=\nodesize}

\def\highlightsize{4pt}
\draw[red, thick] (b_1) circle (\highlightsize);
\draw[red, thick] (c_1) circle (\highlightsize);
\draw[red, thick] (d_1) circle (\highlightsize);
\draw[red, thick] (g_1) circle (\highlightsize);
\draw[red, thick] (f_2) circle (\highlightsize);

\draw[f_arrow] (a_1) to (b_1);
\draw[f_arrow] (a_2) to (b_1);

\draw[f_arrow] (c_1) to (b_1);

\draw[f_arrow] (d_1) to (c_1);
\draw[f_arrow] (d_2) to (c_1);

\draw[f_arrow] (e_1) to (d_1);
\draw[f_arrow] (e_2) to (d_1);
\draw[f_arrow] (e_3) to (d_2);

\draw[f_arrow] (e_1) to (f_1);
\draw[f_arrow] (e_2) to (f_2);
\draw[f_arrow] (e_3) to (f_2);

\draw[f_arrow] (f_1) to (g_1);
\draw[f_arrow] (f_2) to (g_1);

\draw[f_arrow] (d_3) to (c_1);
\draw[f_arrow] (e_4) to (d_3);
\draw[f_arrow] (e_4) to (f_3);
\draw[f_arrow] (f_3) to (g_1);

\draw[->, thick] (4.9,1.5) -- (5.3,1.5) node[midway, above] {\small algorithm};

\end{scope}

\begin{scope}[xshift=5.75cm]

\draw[->, >=stealth, shorten >=4pt, shorten <=4pt] (0,\yarrow) -- ({\xsep*1},\yarrow);
\draw[<-, >=stealth, shorten >=4pt, shorten <=4pt] ({\xsep*1},\yarrow) -- ({\xsep*2},\yarrow);
\draw[<-, >=stealth, shorten >=4pt, shorten <=4pt] ({\xsep*2},\yarrow) -- ({\xsep*3},\yarrow);
\draw[<-, >=stealth, shorten >=4pt, shorten <=4pt] ({\xsep*3},\yarrow) -- ({\xsep*4},\yarrow);
\draw[->, >=stealth, shorten >=4pt, shorten <=4pt] ({\xsep*4},\yarrow) -- ({\xsep*5},\yarrow);
\draw[->, >=stealth, shorten >=4pt, shorten <=4pt] ({\xsep*5},\yarrow) -- ({\xsep*6},\yarrow);

\foreach \x in {0,...,6} {
    \draw[dashed] ({\xsep*\x},\ylinebot) -- ({\xsep*\x},\ylinetop);
    \node at ({\xsep*\x},\ylabel) { $\overset{\fpeval{1+\x}}{\bullet}$};
}

\pgfkeys{/mynode,
    xs=\xscale, 
    ys=\yscale, 
    isep=2*\nodesize,
    fill=white
  }
  
\mynode{name=a_1, fill=white, x=0, y=1.5, xs=\xscale, ys=\yscale, isep=\nodesize}
\mynode{name=a_2, fill=white, x=0, y=2.5, xs=\xscale, ys=\yscale, isep=\nodesize}

\mynode{name=b_1, fill=white, x={\xsep*1}, y=2, xs=\xscale, ys=\yscale, isep=\nodesize}

\mynode{name=c_1, fill=white, x={\xsep*2}, y=2, xs=\xscale, ys=\yscale, isep=\nodesize}

\mynode{name=d_1, fill=white, x={\xsep*3}, y=1.5, xs=\xscale, ys=\yscale, isep=\nodesize}
\mynode{name=d_2, fill=white, x={\xsep*3}, y=2, xs=\xscale, ys=\yscale, isep=\nodesize}
\mynode{name=d_3, fill=white, x={\xsep*3}, y=2.5, xs=\xscale, ys=\yscale, isep=\nodesize}

\mynode{name=e_1, fill=white, x={\xsep*4}, y=1, xs=\xscale, ys=\yscale, isep=\nodesize}
\mynode{name=e_2, fill=white, x={\xsep*4}, y=1.5, xs=\xscale, ys=\yscale, isep=\nodesize}
\mynode{name=e_3, fill=white, x={\xsep*4}, y=2, xs=\xscale, ys=\yscale, isep=\nodesize}
\mynode{name=e_4, fill=white, x={\xsep*4}, y=2.5, xs=\xscale, ys=\yscale, isep=\nodesize}

\mynode{name=f_1, fill=white, x={\xsep*5}, y=1, xs=\xscale, ys=\yscale, isep=\nodesize}
\mynode{name=f_2, fill=white, x={\xsep*5}, y=2, xs=\xscale, ys=\yscale, isep=\nodesize}
\mynode{name=f_3, fill=white, x={\xsep*5}, y=2.5, xs=\xscale, ys=\yscale, isep=\nodesize}

\mynode{name=g_1, fill=white, x={\xsep*6}, y=2, xs=\xscale, ys=\yscale, isep=\nodesize}

\draw[f_arrow] (a_1) to (b_1);

\draw[f_arrow] (c_1) to (b_1);

\draw[f_arrow] (d_2) to (c_1);

\draw[f_arrow] (e_2) to (d_1);
\draw[f_arrow] (e_3) to (d_2);

\draw[f_arrow] (e_1) to (f_1);
\draw[f_arrow] (e_3) to (f_2);

\draw[f_arrow] (f_2) to (g_1);

\draw[f_arrow] (e_4) to (d_3);
\draw[f_arrow] (e_4) to (f_3);
\end{scope}

\end{tikzpicture}

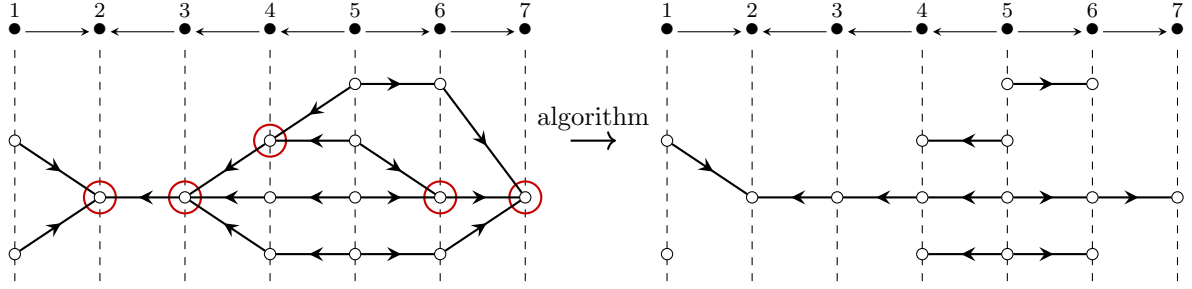
\captionof{figure}{
\textbf{Left:} Merge graph $\mathcal{G}(\pi_0(\mathbb X))$ with merging and splitting vertices highlighted, obstructing a barcode structure. \textbf{Right:} The goal: a barcode basis of the same zigzag module.
}
\end{center}

Given an elementary formigram $\Sb$, the algorithm of Dey and Hou constructs a barcode graph $\mathcal{M}^{\mathrm{bar}}$ from $\mathcal G(\Sb)$ such that 
$$k[\Sb] \cong k[\mathcal{M}^{\mathrm{bar}}]$$
in near-linear time in the number of vertices. It proceeds by iteratively simplifying the merge graph $\mathcal G(\Sb)$. We summarise the algorithm in Section~\ref{sec:the-algo}.

In practice, formigrams constructed from binary videos (Section~\ref{sec:frames-to-formi}) are not elementary, as higher merging and splitting multiplicities may occur. In Section~\ref{sec:elem-refin}, we show that any formigram can be \emph{refined} (cf.\ Definition~\ref{def:elem-ref}) into an elementary one, allowing the above algorithm to be applied. We then show that the barcode of the unperturbed formigram can be recovered, and that this process is independent of the generic perturbation.

\subsection{$i$-Frontier Forests}

A barcode graph is a graph without merging or splitting vertices. The algorithm presented in~\cite{deyComputingZigzagPersistence2021} therefore proceeds iteratively on the level, removing splitting and merging starting at level 1 and ending at level $n$. To formalise this, we introduce an intermediate state of merge graphs, which lies between formigrams and barcode graphs.

\begin{definition}
Let $\mathcal M$ be a merge graph over a zigzag type $\tau$. For $i\in\tau$, we say that $\mathcal M$ is an \emph{$i$-frontier forest} if the following holds on the truncation $\mathcal M[1,i]$:
\begin{itemize}
\item if $v\in\mathcal M[1,i]_d$ is a death vertex, then $d<i$ and $v$ is part of a bar component ending before $i$:
$$v\in \mathcal B\langle j,d\rangle, \quad \text{for $j\leq d.$}$$
\item $\mathcal M[1,i]$ contains no merging vertices. 
\end{itemize}
\end{definition}

\begin{remark}
Unravelling the definition, an $i$-frontier forest is thus a merge graph, such that the truncation $\mathcal{M}[1,i]$ decomposes into bars that end before the level $i$, and \emph{rooted trees}, i.e.\ connected components $\mathcal T$ of $\mathcal M[1,i]$ with a unique minimal level vertex $r$ called the \emph{root}. These graphs correspond to partially decomposed formigrams.
\end{remark}

\begin{center}
\begin{tikzpicture}[yscale=-1, scale=1.5, swatch/.style={draw=none, fill=#1, rectangle, minimum size=2.5mm, inner sep=0pt},f_arrow/.style={
    base_style,
    postaction={
      decorate,
      decoration={
        markings,
        mark=at position 0.55 with {\arrow[scale=1.3]{stealth}}
      }
    }
  },]
\def\xscale{1pt}
\def\yscale{1pt}
\def\nodesize{1.5pt}
\def\specialsize{2pt}

\def\ytop{1}

\def\yarrow{0.9}
\def\ylinebot{1}
\def\ylinetop{4}
\def\ylabel{0.82}

\def\xsep{0.75}

\node at (-0.45,\yarrow) {$\tau:$};

\foreach \x in {0,...,6} {
    \ifnum\x=3\relax
     \draw[thick, dashed, red] ({\xsep*\x},\ylinebot) -- ({\xsep*\x},\ylinetop);
      \node at ({\xsep*\x},\ylabel) { $\overset{\fpeval{1+\x}}{\bullet}$};
    \else
    \draw[dotted] ({\xsep*\x},\ylinebot) -- ({\xsep*\x},\ylinetop);
  \node at ({\xsep*\x},\ylabel) { $\overset{\fpeval{1+\x}}{\bullet}$};
    \fi
  
}

\node[anchor=south, red]
    at ({\xsep*3},\ylinetop+0.4) {frontier};

\draw[->, >=stealth, shorten >=4pt, shorten <=4pt] (0,\yarrow) -- ({\xsep*1},\yarrow);
\draw[<-, >=stealth, shorten >=4pt, shorten <=4pt] ({\xsep*1},\yarrow) -- ({\xsep*2},\yarrow);
\draw[<-, >=stealth, shorten >=4pt, shorten <=4pt] ({\xsep*2},\yarrow) -- ({\xsep*3},\yarrow);
\draw[<-, >=stealth, shorten >=4pt, shorten <=4pt] ({\xsep*3},\yarrow) -- ({\xsep*4},\yarrow);
\draw[->, >=stealth, shorten >=4pt, shorten <=4pt] ({\xsep*4},\yarrow) -- ({\xsep*5},\yarrow);
\draw[<-, >=stealth, shorten >=4pt, shorten <=4pt] ({\xsep*5},\yarrow) -- ({\xsep*6},\yarrow);

\def\frontier{3} 

\draw[decorate, decoration={brace, amplitude=4pt}]
  (0,0.6) -- ({\xsep*\frontier},0.6)
  node[midway, yshift=12pt] {$\mathcal M[1,4]$};

\pgfkeys{/mynode,
    xs=\xscale, 
    ys=\yscale, 
    isep=2*\nodesize,
    fill=white
  }

\begin{scope}
\mynode{name=a_1, fill=white, x=0, y=1.25, xs=\xscale, ys=\yscale, isep=\nodesize}
\mynode{name=a_2, fill=white, x=0, y=2, xs=\xscale, ys=\yscale, isep=\nodesize}

\mynode{name=b_1, fill=white, x={\xsep*1}, y=1.25, xs=\xscale, ys=\yscale, isep=\nodesize}
\mynode{name=b_2, fill=white, x={\xsep*1}, y=2, xs=\xscale, ys=\yscale, isep=\nodesize}
\mynode{name=b_3, fill=white, x={\xsep*1}, y=3, xs=\xscale, ys=\yscale, isep=\nodesize}
\mynode{name=b_4, fill=white, x={\xsep*1}, y=3.5, xs=\xscale, ys=\yscale, isep=\nodesize}

\mynode{name=c_1, fill=white, x={\xsep*2}, y=1.75, xs=\xscale, ys=\yscale, isep=\nodesize}
\mynode{name=c_2, fill=white, x={\xsep*2}, y=2.25, xs=\xscale, ys=\yscale, isep=\nodesize}
\mynode{name=c_3, fill=white, x={\xsep*2}, y=3, xs=\xscale, ys=\yscale, isep=\nodesize}
\mynode{name=c_4, fill=white, x={\xsep*2}, y=3.5, xs=\xscale, ys=\yscale, isep=\nodesize}

\mynode{name=d_1, fill=white, x={\xsep*3}, y=1.5, xs=\xscale, ys=\yscale, isep=\nodesize}
\mynode{name=d_2, fill=white, x={\xsep*3}, y=1.75, xs=\xscale, ys=\yscale, isep=\nodesize}
\mynode{name=d_3, fill=white, x={\xsep*3}, y=2.25, xs=\xscale, ys=\yscale, isep=\nodesize}
\mynode{name=d_4, fill=white, x={\xsep*3}, y=2.5, xs=\xscale, ys=\yscale, isep=\nodesize}
\mynode{name=d_5, fill=white, x={\xsep*3}, y=3.25, xs=\xscale, ys=\yscale, isep=\nodesize}
\mynode{name=d_6, fill=white, x={\xsep*3}, y=3.75, xs=\xscale, ys=\yscale, isep=\nodesize}

\mynode{name=e_1, fill=white, x={\xsep*4}, y=1.75, xs=\xscale, ys=\yscale, isep=\nodesize}
\mynode{name=e_2, fill=white, x={\xsep*4}, y=2.25, xs=\xscale, ys=\yscale, isep=\nodesize}
\mynode{name=e_3, fill=white, x={\xsep*4}, y=3.25, xs=\xscale, ys=\yscale, isep=\nodesize}
\mynode{name=e_4, fill=white, x={\xsep*4}, y=3.75, xs=\xscale, ys=\yscale, isep=\nodesize}

\mynode{name=f_1, fill=white, x={\xsep*5}, y=2, xs=\xscale, ys=\yscale, isep=\nodesize}
\mynode{name=f_2, fill=white, x={\xsep*5}, y=3.5, xs=\xscale, ys=\yscale, isep=\nodesize}

\mynode{name=g_1, fill=white, x={\xsep*6}, y=3.25, xs=\xscale, ys=\yscale, isep=\nodesize}
\mynode{name=g_2, fill=white, x={\xsep*6}, y=3.75, xs=\xscale, ys=\yscale, isep=\nodesize}


\draw[f_arrow] (a_1) to (b_1);
\draw[f_arrow] (a_2) to (b_2);

\draw[f_arrow] (c_1) to (b_2);
\draw[f_arrow] (c_2) to (b_2);
\draw[f_arrow] (c_3) to (b_3);
\draw[f_arrow] (c_4) to (b_4);

\draw[f_arrow] (d_1) to (c_1);
\draw[f_arrow] (d_2) to (c_1);
\draw[f_arrow] (d_3) to (c_2);
\draw[f_arrow] (d_4) to (c_2);
\draw[f_arrow] (d_5) to (c_4);
\draw[f_arrow] (d_6) to (c_4);

\draw[f_arrow] (e_1) to (d_2);
\draw[f_arrow] (e_2) to (d_3);
\draw[f_arrow] (e_3) to (d_5);
\draw[f_arrow] (e_4) to (d_6);

\draw[f_arrow] (e_1) to (f_1);
\draw[f_arrow] (e_2) to (f_1);
\draw[f_arrow] (e_3) to (f_2);
\draw[f_arrow] (e_4) to (f_2);

\draw[f_arrow] (g_1) to (f_2);
\draw[f_arrow] (g_2) to (f_2);










\end{scope}

\end{tikzpicture}

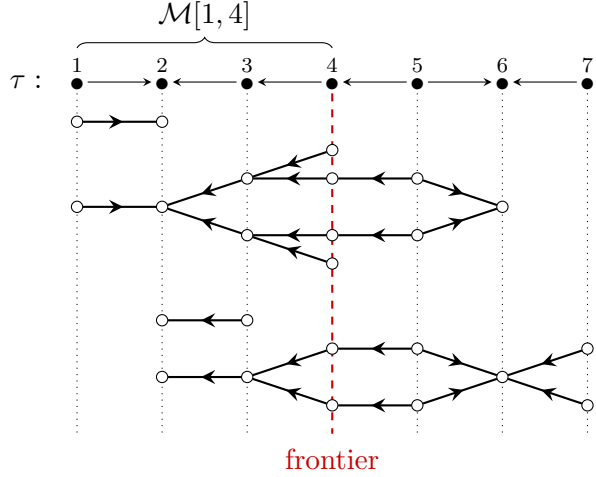
\captionof{figure}{
\textbf{$4$-Frontier Forest.} The truncation $\M[1,4]$ contains no merging vertices and decomposes into bars and rooted trees.}
\end{center}

To better understand the relation between the vertices of merge graphs, we introduce the following definition. From this point onward, we denote the discrete interval $ \{i\in\mathbb{Z}\, | \, a\leq i\leq b\}$ by $\llbracket a , b \rrbracket$.

\begin{definition}
Let $\mathcal{M}$ be a merge graph over $\tau$. Let $1\leq t_{\mathrm{start}} \leq t_{\mathrm{end}} \leq |\tau|$, a \emph{level path} from level $t_{\mathrm{start}}$ to $t_{\mathrm{end}}$ is a map
\[
\gamma : \llbracket t_\mathrm{start}, t_{\mathrm{end}}\rrbracket \to \mathrm{Vertex}(\mathcal{M})
\]
such that $\lvl(\gamma(t))=t$ for all $t\in\llbracket t_\mathrm{start}, t_{\mathrm{end}}\rrbracket$, and $\gamma(t)$ and $\gamma(t+1)$ are connected by an edge of $\mathcal{M}$ for each $i$.
\end{definition}

\begin{remark}
Given a merge graph $\mathcal{M}$ over $\tau$, vertices $v_1\in \mathcal M_{l_1}$ and $v_2\in \mathcal M_{l_2}$ are connected by a level path if there exists a level path $\gamma:\llbracket l_1, l_2 \rrbracket \to \mathrm{Vertex}(\mathcal M)$ with $\gamma(l_j) = v_j$ for $j\in \{1,2\}$. In this case, we say that $v_1$ is an \emph{ancestor} of $v_2$, and that $v_2$ is a \emph{descendant} of $v_1$. We denote the set of ancestors of a vertex $v$ by $\mathrm{Anc}(v)$
\end{remark}

\begin{remark}
If $\mathcal T$ is a rooted tree of an $i$-frontier forest $\mathcal M$ with root $r$, then every vertex $v\in\mathcal T$ with $v\neq r$ has a unique ancestor. The condition on the death vertices of an $i$-frontier forest ensures that every vertex $v\in\mathcal T$ with $\lvl(v)<i$ has at least one descendant (all leaves must be at the frontier $i$).
\end{remark}

\begin{lemma}\label{lem:forests}
Every merge graph is an $1$-frontier forest. If $\mathcal M$ is a merge graph with terminal level $n$ and $\mathcal M_n=\emptyset$, then $\mathcal M$ is a barcode graph.
\end{lemma}

\begin{proof}
If $i=1$, then every vertex of $\mathcal M[1,1]$ can be considered its own rooted tree. If $\mathcal M_n=\emptyset$, then every death vertex of $\mathcal M$, if it exists, must occur at level $<n$, and no rooted trees can exist. It follows that every component must be a bar, so $\mathcal M$ is a barcode graph.
\end{proof}

\subsection{The Algorithm}\label{sec:the-algo}

In this section, we describe the details of the algorithm presented in~\cite{deyComputingZigzagPersistence2021}. Let $k$ be a field.

\noindent{\textbf{Input.}} An elementary formigram $\Sb:\mathcal P_\tau \to \cat{set}$ over a zigzag type $\tau$ of length $|\tau|=n$.

\noindent{\textbf{Output.}} $\mathrm{Bar}(k[\mathbb S])$.

\noindent{\textbf{Description.}} We initialise an empty multiset $\mathcal{B}_{\mathrm{out}}$ and record bars as they are identified during the execution of the algorithm. The algorithm proceeds iteratively over the levels of $\tau$, constructing a sequence
\[
\mathcal M^{(1)}, \mathcal M^{(2)}, \dots, \mathcal M^{(n)},
\]
with $\mathcal M^{(1)} = \mathcal G(\Sb)$, such that for each $i$, the following invariants hold:
\begin{itemize}
    \item $\mathcal M^{(i)}$ is an $i$-frontier forest,
    \item $k[\Sb] \cong k[\M^{(i)}]$,
    \item the graph has not been modified beyond level $i$, i.e.
    \[
    \mathcal M^{(i)}[i,n] = \mathcal G(\Sb)[i,n].
    \]
\end{itemize}

\noindent At step $i$, suppose that $\mathcal M^{(i)}$ is the $i$-frontier forest satisfying the above invariants, constructed by the previous steps of the algorithm. If the structure map between $i$ and $i+1$ is a bijection, we set $M^{(i+1)}\coloneq M^{(i)}$, and continue. To extend the frontier to level $i+1$ for the other cases, we consider the orientation of the edge between $i$ and $i+1$.

\paragraph{\textbf{Case} $(i\to i+1)$.}
Since $\mathcal M^{(i)}$ coincides with $\mathcal G(\Sb)$ on levels $\geq i$, and $\Sb$ is elementary, the only possible obstruction to $\M^{(i)}$ being an $(i+1)$-frontier forest is the presence of a single merging vertex $m \in \mathcal M^{(i)}_{i+1}$. In this case, there exist distinct vertices $v_1,v_2 \in \mathcal M^{(i)}_i$ such that
\[
\mathbb S[i\to i+1](v_1) = \mathbb S[i\to i+1](v_2) = m.
\]
By the definition of an $i$-frontier forest, $v_1$ and $v_2$ belong to rooted trees $\mathcal T_1$ and $\mathcal T_2$.
\begin{enumerate}[label=\alph*)]
\item \textbf{If} $\T_1=\T_2$: let $a\in \M^{(i)}_{j-1}$ denote the highest level common ancestor of $v_1$ and $v_2$. Since $\T_1$ is a rooted tree, there exist unique level paths 
\[
\gamma_1, \gamma_2 : \llbracket j,i \rrbracket \to \mathrm{Vertex}(\mathcal M^{(i)})
\]
connecting $v_1$ and $v_2$ to their respective level $j$ ancestors (starting just after $a$).

\vspace{0.5cm}
\begin{center}
\begin{tikzpicture}[yscale=-0.75, scale=1.5, swatch/.style={draw=none, fill=#1, rectangle, minimum size=2.5mm, inner sep=0pt},f_arrow/.style={
    base_style,
    postaction={
      decorate,
      decoration={
        markings,
        mark=at position 0.55 with {\arrow[scale=1.3]{stealth}}
      }
    }
  },]
\def\xscale{1pt}
\def\yscale{1pt}
\def\nodesize{1.5pt}
\def\specialsize{2pt}

\def\ytop{1}

\def\yarrow{0.42}
\def\ylinebot{0.6}
\def\ylinetop{4.5}
\def\ylabel{0.42}

\def\xsep{0.75}


\pgfkeys{/mynode,
    xs=\xscale, 
    ys=\yscale, 
    isep=2*\nodesize,
    fill=white
  }

\begin{scope}
\node at (-0.45,\yarrow) {$\tau:$};
\foreach \x in {0,...,4} {

  \ifnum\x=1
    \node at ({\xsep*\x},0.05) {$j$};
  \fi

  \ifnum\x=3
    \node at ({\xsep*\x},0.05) {$i$};
  \fi

  \ifnum\x=4
    \node at ({\xsep*\x},0.05) {$i+1$};
  \fi

  \node at ({\xsep*\x},\ylabel) {$\bullet$};
  \draw[dotted] ({\xsep*\x},\ylinebot) -- ({\xsep*\x},\ylinetop);
}

\draw[<-, >=stealth, shorten >=4pt, shorten <=4pt] (0,\yarrow) -- ({\xsep*1},\yarrow);
\draw[<-, >=stealth, shorten >=4pt, shorten <=4pt] ({\xsep*1},\yarrow) -- ({\xsep*2},\yarrow);
\draw[<-, >=stealth, shorten >=4pt, shorten <=4pt] ({\xsep*2},\yarrow) -- ({\xsep*3},\yarrow);
\draw[->, >=stealth, shorten >=4pt, shorten <=4pt] ({\xsep*3},\yarrow) -- ({\xsep*4},\yarrow);

\mynode{name=a_1, ltxt=$a$, lpos=above left, fill=white, x=0, y=2.75, xs=\xscale, ys=\yscale, isep=\nodesize}

\mynode{name=b_1, ltxt=$ $, lpos=above left, fill=white, x={\xsep*1}, y=2, xs=\xscale, ys=\yscale, isep=\nodesize}
\mynode{name=b_2, fill=white, x={\xsep*1}, y=3.5, xs=\xscale, ys=\yscale, isep=\nodesize}

\mynode{name=c_1, ltxt=$b$, fill=white, x={\xsep*2}, y=1.25, xs=\xscale, ys=\yscale, isep=\nodesize}
\mynode{name=c_2, ltxt=$ $, fill=white, x={\xsep*2}, y=2, xs=\xscale, ys=\yscale, isep=\nodesize}
\mynode{name=c_3, fill=white, x={\xsep*2}, y=3.5, xs=\xscale, ys=\yscale, isep=\nodesize}

\mynode{name=d_1, fill=white, x={\xsep*3}, y=1.25, xs=\xscale, ys=\yscale, isep=\nodesize}
\mynode{name=d_2, ltxt=$v_1$, lpos=above right, fill=white, x={\xsep*3}, y=2., xs=\xscale, ys=\yscale, isep=\nodesize}
\mynode{name=d_3, ltxt=$v_2$, lpos=below right, fill=white, x={\xsep*3}, y=3.5, xs=\xscale, ys=\yscale, isep=\nodesize}
\mynode{name=d_4, ltxt=$ $, fill=white, x={\xsep*3}, y=4.25, xs=\xscale, ys=\yscale, isep=\nodesize}

\mynode{name=e_1, fill=white, x={\xsep*4}, y=1.25, xs=\xscale, ys=\yscale, isep=\nodesize}
\mynode{name=e_2, ltxt=$m$, lpos=below right,, fill=white, x={\xsep*4}, y=2.75, xs=\xscale, ys=\yscale, isep=\nodesize}
\mynode{name=e_3, ltxt=$ $, fill=white, x={\xsep*4}, y=4.25, xs=\xscale, ys=\yscale, isep=\nodesize}

\draw[b_arrow] (a_1) to (b_1);
\draw[b_arrow] (a_1) to (b_2);

\draw[f_arrow] (c_1) to (b_1);
\draw[f_arrow] (c_2) to (b_1);
\draw[f_arrow] (c_3) to (b_2);

\draw[f_arrow] (d_1) to (c_1);
\draw[f_arrow] (d_2) to (c_2);
\draw[f_arrow] (d_3) to (c_3);
\draw[f_arrow] (d_4) to (c_3);

\draw[f_arrow] (d_1) to (e_1);
\draw[f_arrow] (d_2) to (e_2);
\draw[f_arrow] (d_3) to (e_2);
\draw[f_arrow] (d_4) to (e_3);

\draw[red, dashed]
  ($(b_1)+(0,1.2mm)$) -- ($(d_2)+(0,1.2mm)$)
  node[midway, below] {$\gamma_1$};

\draw[blue, dashed]
($(b_2)-(0,1.2mm)$) -- ($(d_3)-(0,1.2mm)$)
node[midway, above] {$\gamma_2$};

\node at ({2*\xsep},\ylinetop+0.5) {$\M^{(i)}$};

\end{scope}

\coordinate (L) at ({4*\xsep},{(\ylinebot+\ylinetop)/2});
\coordinate (R) at ({4*\xsep+4.75cm},{(\ylinebot+\ylinetop)/2});

\draw[->] ($(L)!0.25!(R)$) -- ($(L)!0.75!(R)$);

\begin{scope}[xshift=4.75cm]

\foreach \x in {0,...,4} {

  \ifnum\x=1
    \node at ({\xsep*\x},0.05) {$j$};
  \fi

  \ifnum\x=3
    \node at ({\xsep*\x},0.05) {$i$};
  \fi

  \ifnum\x=4
    \node at ({\xsep*\x},0.05) {$i+1$};
  \fi

  \node at ({\xsep*\x},\ylabel) {$\bullet$};
  \draw[dotted] ({\xsep*\x},\ylinebot) -- ({\xsep*\x},\ylinetop);
}

\draw[<-, >=stealth, shorten >=4pt, shorten <=4pt] (0,\yarrow) -- ({\xsep*1},\yarrow);
\draw[<-, >=stealth, shorten >=4pt, shorten <=4pt] ({\xsep*1},\yarrow) -- ({\xsep*2},\yarrow);
\draw[<-, >=stealth, shorten >=4pt, shorten <=4pt] ({\xsep*2},\yarrow) -- ({\xsep*3},\yarrow);
\draw[->, >=stealth, shorten >=4pt, shorten <=4pt] ({\xsep*3},\yarrow) -- ({\xsep*4},\yarrow);

\mynode{name=a_1, ltxt=$a$, lpos=below left, fill=white, x=0, y=2.5, xs=\xscale, ys=\yscale, isep=\nodesize}

\mynode{name=b_1, ltxt=$\gamma_1(j)$, lpos=above left, fill=white, x={\xsep*1}, y=1.5, xs=\xscale, ys=\yscale, isep=\nodesize}
\mynode{name=b_2, ltxt=$ $, fill=white, x={\xsep*1}, y=2.5, xs=\xscale, ys=\yscale, isep=\nodesize}

\mynode{name=c_1, ltxt=$b$, lpos=below left, fill=white, x={\xsep*2}, y=3.9, xs=\xscale, ys=\yscale, isep=\nodesize}
\mynode{name=c_2, ltxt=$ $, fill=white, x={\xsep*2}, y=1.5, xs=\xscale, ys=\yscale, isep=\nodesize}
\mynode{name=c_3,  fill=white, x={\xsep*2}, y=2.5, xs=\xscale, ys=\yscale, isep=\nodesize}

\mynode{name=d_1, fill=white, x={\xsep*3}, y=3.9, xs=\xscale, ys=\yscale, isep=\nodesize}
\mynode{name=d_2, ltxt=$v_1$, lpos=above right,  fill=white, x={\xsep*3}, y=1.5, xs=\xscale, ys=\yscale, isep=\nodesize}
\mynode{name=d_3, ltxt=$v_2$, lpos=below right, fill=white, x={\xsep*3}, y=2.5, xs=\xscale, ys=\yscale, isep=\nodesize}
\mynode{name=d_4, ltxt=$ $ ,fill=white, x={\xsep*3}, y=3.25, xs=\xscale, ys=\yscale, isep=\nodesize}

\mynode{name=e_1, fill=white, x={\xsep*4}, y=3.9, xs=\xscale, ys=\yscale, isep=\nodesize}
\mynode{name=e_2, ltxt=$m$, lpos=above right,  fill=white, x={\xsep*4}, y=2.5, xs=\xscale, ys=\yscale, isep=\nodesize}
\mynode{name=e_3, ltxt=$ $ ,fill=white, x={\xsep*4}, y=3.25, xs=\xscale, ys=\yscale, isep=\nodesize}

\draw[b_arrow] (a_1) to (b_2);

\draw[f_arrow] (c_1) to (b_2);
\draw[f_arrow] (c_2) to (b_1);
\draw[f_arrow] (c_3) to (b_2);

\draw[f_arrow] (d_1) to (c_1);
\draw[f_arrow] (d_2) to (c_2);
\draw[f_arrow] (d_3) to (c_3);
\draw[f_arrow] (d_4) to (c_3);

\draw[f_arrow] (d_1) to (e_1);
\draw[f_arrow] (d_3) to (e_2);
\draw[f_arrow] (d_4) to (e_3);

\node at ({2*\xsep},\ylinetop+0.5) {$\M^{(i+1)}$};
\end{scope}

\end{tikzpicture}
\end{center}

\vspace{0.5cm}

\item \textbf{If} $\T_1\neq\T_2$: Let $r_1$ and $r_2$ denote the roots of $\T_1$ and $\T_2$, and assume without loss of generality that $j \coloneq \lvl(r_1) \geq \lvl(r_2)$. Then there exist unique level paths
\[
\gamma_1, \gamma_2 : \llbracket j,i \rrbracket \to \mathrm{Vertex}(\mathcal M^{(i)})
\]
where $\gamma_1$ connects $r_1$ to $v_1$, and $\gamma_2$ connects the level-$j$ ancestor of $v_2$ to $v_2$

\vspace{1cm}
\begin{center}
\begin{tikzpicture}[yscale=-0.75, scale=1.5, swatch/.style={draw=none, fill=#1, rectangle, minimum size=2.5mm, inner sep=0pt},f_arrow/.style={
    base_style,
    postaction={
      decorate,
      decoration={
        markings,
        mark=at position 0.55 with {\arrow[scale=1.3]{stealth}}
      }
    }
  },]
\def\xscale{1pt}
\def\yscale{1pt}
\def\nodesize{1.5pt}
\def\specialsize{2pt}

\def\ytop{1}

\def\yarrow{0.42}
\def\ylinebot{0.6}
\def\ylinetop{4.5}
\def\ylabel{0.42}

\def\xsep{0.75}


\pgfkeys{/mynode,
    xs=\xscale, 
    ys=\yscale, 
    isep=2*\nodesize,
    fill=white
  }

\node at (-0.45,\yarrow) {$\tau:$};

\begin{scope}
\foreach \x in {0,...,4} {

  \ifnum\x=1
    \node at ({\xsep*\x},0.05) {$j$};
  \fi

  \ifnum\x=3
    \node at ({\xsep*\x},0.05) {$i$};
  \fi

  \ifnum\x=4
    \node at ({\xsep*\x},0.05) {$i+1$};
  \fi

  \node at ({\xsep*\x},\ylabel) {$\bullet$};
  \draw[dotted] ({\xsep*\x},\ylinebot) -- ({\xsep*\x},\ylinetop);
}

\draw[->, >=stealth, shorten >=4pt, shorten <=4pt] (0,\yarrow) -- ({\xsep*1},\yarrow);
\draw[<-, >=stealth, shorten >=4pt, shorten <=4pt] ({\xsep*1},\yarrow) -- ({\xsep*2},\yarrow);
\draw[<-, >=stealth, shorten >=4pt, shorten <=4pt] ({\xsep*2},\yarrow) -- ({\xsep*3},\yarrow);
\draw[->, >=stealth, shorten >=4pt, shorten <=4pt] ({\xsep*3},\yarrow) -- ({\xsep*4},\yarrow);

\mynode{name=a_1, ltxt=$r_2$, lpos=above left, fill=white, x=0, y=3.25, xs=\xscale, ys=\yscale, isep=\nodesize}

\mynode{name=b_1, ltxt=$r_1$, lpos=above left, fill=white, x={\xsep*1}, y=2, xs=\xscale, ys=\yscale, isep=\nodesize}
\mynode{name=b_2, ltxt=$ $, fill=white, x={\xsep*1}, y=3.25, xs=\xscale, ys=\yscale, isep=\nodesize}

\mynode{name=c_1, ltxt=$b$, fill=white, x={\xsep*2}, y=1.25, xs=\xscale, ys=\yscale, isep=\nodesize}
\mynode{name=c_2, ltxt=$ $, fill=white, x={\xsep*2}, y=2, xs=\xscale, ys=\yscale, isep=\nodesize}
\mynode{name=c_3, fill=white, x={\xsep*2}, y=3.25, xs=\xscale, ys=\yscale, isep=\nodesize}

\mynode{name=d_1, fill=white, x={\xsep*3}, y=1.25, xs=\xscale, ys=\yscale, isep=\nodesize}
\mynode{name=d_2, ltxt=$v_1$, lpos=above right, fill=white, x={\xsep*3}, y=2., xs=\xscale, ys=\yscale, isep=\nodesize}
\mynode{name=d_3, ltxt=$v_2$, lpos=below right, fill=white, x={\xsep*3}, y=3.25, xs=\xscale, ys=\yscale, isep=\nodesize}

\mynode{name=e_1, ltxt=$ $, fill=white, x={\xsep*4}, y=1.25, xs=\xscale, ys=\yscale, isep=\nodesize}
\mynode{name=e_2, ltxt=$m$, lpos=below right,, fill=white, x={\xsep*4}, y=2.75, xs=\xscale, ys=\yscale, isep=\nodesize}

\draw[f_arrow] (a_1) to (b_2);

\draw[f_arrow] (c_1) to (b_1);
\draw[f_arrow] (c_2) to (b_1);
\draw[f_arrow] (c_3) to (b_2);

\draw[f_arrow] (d_1) to (c_1);
\draw[f_arrow] (d_2) to (c_2);
\draw[f_arrow] (d_3) to (c_3);

\draw[f_arrow] (d_1) to (e_1);
\draw[f_arrow] (d_2) to (e_2);
\draw[f_arrow] (d_3) to (e_2);

\draw[red, dashed]
  ($(b_1)+(0,1.2mm)$) -- ($(d_2)+(0,1.2mm)$)
  node[midway, below] {$\gamma_1$};

\draw[blue, dashed]
($(b_2)-(0,1.2mm)$) -- ($(d_3)-(0,1.2mm)$)
node[midway, above] {$\gamma_2$};

\node at ({2*\xsep},\ylinetop+0.5) {$\M^{(i)}$};

\end{scope}

\coordinate (L) at ({4*\xsep},{(\ylinebot+\ylinetop)/2});
\coordinate (R) at ({4*\xsep+4.75cm},{(\ylinebot+\ylinetop)/2});
\draw[->] ($(L)!0.25!(R)$) -- ($(L)!0.75!(R)$);

\begin{scope}[xshift=4.75cm]

\foreach \x in {0,...,4} {

  \ifnum\x=1
    \node at ({\xsep*\x},0.05) {$j$};
  \fi

  \ifnum\x=3
    \node at ({\xsep*\x},0.05) {$i$};
  \fi

  \ifnum\x=4
    \node at ({\xsep*\x},0.05) {$i+1$};
  \fi

  \node at ({\xsep*\x},\ylabel) {$\bullet$};
  \draw[dotted] ({\xsep*\x},\ylinebot) -- ({\xsep*\x},\ylinetop);
}

\draw[->, >=stealth, shorten >=4pt, shorten <=4pt] (0,\yarrow) -- ({\xsep*1},\yarrow);
\draw[<-, >=stealth, shorten >=4pt, shorten <=4pt] ({\xsep*1},\yarrow) -- ({\xsep*2},\yarrow);
\draw[<-, >=stealth, shorten >=4pt, shorten <=4pt] ({\xsep*2},\yarrow) -- ({\xsep*3},\yarrow);
\draw[->, >=stealth, shorten >=4pt, shorten <=4pt] ({\xsep*3},\yarrow) -- ({\xsep*4},\yarrow);

\mynode{name=a_1, ltxt=$r_2$, lpos=below left, fill=white, x=0, y=3.25, xs=\xscale, ys=\yscale, isep=\nodesize}

\mynode{name=b_1, ltxt=$r_1$, lpos=above left, fill=white, x={\xsep*1}, y=2, xs=\xscale, ys=\yscale, isep=\nodesize}
\mynode{name=b_2, ltxt=$ $, fill=white, x={\xsep*1}, y=3.25, xs=\xscale, ys=\yscale, isep=\nodesize}

\mynode{name=c_1, ltxt=$b$, lpos=below left, fill=white, x={\xsep*2}, y=4, xs=\xscale, ys=\yscale, isep=\nodesize}
\mynode{name=c_2, ltxt=$ $, fill=white, x={\xsep*2}, y=2, xs=\xscale, ys=\yscale, isep=\nodesize}
\mynode{name=c_3, fill=white, x={\xsep*2}, y=3.25, xs=\xscale, ys=\yscale, isep=\nodesize}

\mynode{name=d_1, fill=white, x={\xsep*3}, y=4, xs=\xscale, ys=\yscale, isep=\nodesize}
\mynode{name=d_2, ltxt=$v_1$, lpos=above right, fill=white, x={\xsep*3}, y=2., xs=\xscale, ys=\yscale, isep=\nodesize}
\mynode{name=d_3, ltxt=$v_2$, lpos=below right, fill=white, x={\xsep*3}, y=3.25, xs=\xscale, ys=\yscale, isep=\nodesize}

\mynode{name=e_1, ltxt=$ $, fill=white, x={\xsep*4}, y=4, xs=\xscale, ys=\yscale, isep=\nodesize}
\mynode{name=e_2, ltxt=$m$, lpos=right,, fill=white, x={\xsep*4}, y=2.75, xs=\xscale, ys=\yscale, isep=\nodesize}

\draw[f_arrow] (a_1) to (b_2);

\draw[f_arrow] (c_1) to (b_2);
\draw[f_arrow] (c_2) to (b_1);
\draw[f_arrow] (c_3) to (b_2);

\draw[f_arrow] (d_1) to (c_1);
\draw[f_arrow] (d_2) to (c_2);
\draw[f_arrow] (d_3) to (c_3);

\draw[f_arrow] (d_1) to (e_1);
\draw[f_arrow] (d_3) to (e_2);

\node at ({2*\xsep},\ylinetop+0.5) {$\M^{(i+1)}$};

\end{scope}

\end{tikzpicture}
\end{center}

\vspace{0.5cm}

\end{enumerate}

In both cases, we thus obtain level paths $\gamma_1, \gamma_2 : \llbracket j,i \rrbracket \to \mathrm{Vertex}(\mathcal M^{(i)})$, and we proceed as follows. We construct $\mathcal M^{(i+1)}$ from $\mathcal M^{(i)}$ by keeping the vertex set fixed and redirecting all off-path edges incident to $\gamma_1$ onto $\gamma_2$, levelwise. Concretely, for each $t\in \llbracket j, i-1 \rrbracket$, we perform the following replacement: if $(u,\gamma_1(t))$ is an edge with $u\neq \gamma_1(t\pm1)$, replace it by $(u,\gamma_2(t))$. Finally, at the endpoints, any outgoing edges from $\gamma_1(i)$ and $\gamma_1(j)$ are removed (if they exist). 

In both cases, the path $\gamma_1$ becomes isolated and hence determines a bar component in $\M^{(i+1)}$ with endpoints $\langle j,i\rangle$. We record a corresponding bar in $\mathcal B_{\mathrm{out}}$. The type of bar is determined by the direction of the maps at the endpoints: in case a), the map preceding level $j$ is backward, yielding a bar $(j,i)$, whereas in case b), it is forward, yielding a bar $[j,i)$.

\paragraph{\textbf{Case} $(i+1\to i)$.} Since $\Sb$ is elementary, the only obstruction to $\mathcal M^{(i)}$ being an $(i+1)$-frontier forest is the presence of a death vertex $v$ at level $i$. If $v$ is part of a bar component $\mathcal B\langle j,i\rangle$, then $\M^{(i)}$ is already an $(i+1)$-forest and we set $\M^{(i+1)}\coloneq\M^{(i)}$. Moreover, we record a bar $[j,i]$ in $\mathcal B_{\mathrm{out}}$.

In case $v$ does not belong to a bar component, let $s\in  \M^{(i)}$ be the highest level ancestor of $v$ that is also a splitting vertex and set $j \coloneq \lvl(s)+1$. We obtain $\mathcal M^{(i+1)}$ from $\mathcal M^{(i)}$ by keeping the vertex set fixed and deleting the edge between $s$ and the level-$j$ ancestor of $v$. In this case, the path from $v$ to its level-$j$ ancestor becomes a bar component in $\mathcal M^{(i+1)}$. The existence of the splitting vertex $s$ forces the arrow direction $j\to j-1$ in $\tau$. Hence, we record a bar $(j,i]$ in $\mathcal B_{\mathrm{out}}$.

\vspace{0.5cm}

\begin{center}
\begin{tikzpicture}[yscale=-0.75, scale=1.5, swatch/.style={draw=none, fill=#1, rectangle, minimum size=2.5mm, inner sep=0pt},f_arrow/.style={
    base_style,
    postaction={
      decorate,
      decoration={
        markings,
        mark=at position 0.55 with {\arrow[scale=1.3]{stealth}}
      }
    }
  },]
\def\xscale{1pt}
\def\yscale{1pt}
\def\nodesize{1.5pt}
\def\specialsize{2pt}

\def\ytop{1}

\def\yarrow{0.42}
\def\ylinebot{0.6}
\def\ylinetop{3.75}
\def\ylabel{0.42}

\def\xsep{0.75}


\pgfkeys{/mynode,
    xs=\xscale, 
    ys=\yscale, 
    isep=2*\nodesize,
    fill=white
  }

\begin{scope}

\node at (-0.45,\yarrow) {$\tau:$};

\foreach \x in {0,...,4} {

  \ifnum\x=1
    \node at ({\xsep*\x},0.05) {$j$};
  \fi

  \ifnum\x=3
    \node at ({\xsep*\x},0.05) {$i$};
  \fi

  \ifnum\x=4
    \node at ({\xsep*\x},0.05) {$i+1$};
  \fi

  \node at ({\xsep*\x},\ylabel) {$\bullet$};
  \draw[dotted] ({\xsep*\x},\ylinebot) -- ({\xsep*\x},\ylinetop);
}

\draw[->, >=stealth, shorten >=4pt, shorten <=4pt] (0,\yarrow) -- ({\xsep*1},\yarrow);
\draw[<-, >=stealth, shorten >=4pt, shorten <=4pt] ({\xsep*1},\yarrow) -- ({\xsep*2},\yarrow);
\draw[<-, >=stealth, shorten >=4pt, shorten <=4pt] ({\xsep*2},\yarrow) -- ({\xsep*3},\yarrow);
\draw[<-, >=stealth, shorten >=4pt, shorten <=4pt] ({\xsep*3},\yarrow) -- ({\xsep*4},\yarrow);

\mynode{name=a_1, ltxt=$ $, lpos=above left, fill=white, x=0, y=2.75, xs=\xscale, ys=\yscale, isep=\nodesize}

\mynode{name=b_1, ltxt=$a$, lpos=below left, fill=white, x={\xsep*1}, y=2.75, xs=\xscale, ys=\yscale, isep=\nodesize}

\mynode{name=c_1, ltxt=$ $, fill=white, x={\xsep*2}, y=2.75, xs=\xscale, ys=\yscale, isep=\nodesize}
\mynode{name=c_2, ltxt=$ $, fill=white, x={\xsep*2}, y=1.5, xs=\xscale, ys=\yscale, isep=\nodesize}

\mynode{name=d_1, ltxt=$v_2$, lpos=below right, fill=white, x={\xsep*3}, y=2.75, xs=\xscale, ys=\yscale, isep=\nodesize}
\mynode{name=d_2, ltxt=$v_1$, lpos=above right, fill=white, x={\xsep*3}, y=1.5, xs=\xscale, ys=\yscale, isep=\nodesize}

\mynode{name=e_1, ltxt=$ $, lpos=right,, fill=white, x={\xsep*4}, y=2.75, xs=\xscale, ys=\yscale, isep=\nodesize}

\draw[f_arrow] (a_1) to (b_1);

\draw[f_arrow] (c_1) to (b_1);
\draw[f_arrow] (c_2) to (b_1);

\draw[f_arrow] (d_1) to (c_1);
\draw[f_arrow] (d_2) to (c_2);

\draw[f_arrow] (e_1) to (d_1);

\draw[red, dashed]
  ($(c_2)+(0,1.2mm)$) -- ($(d_2)+(0,1.2mm)$)
  node[midway, below] {$\gamma_1$};

\draw[blue, dashed]
($(c_1)-(0,1.2mm)$) -- ($(d_1)-(0,1.2mm)$)
node[midway, above] {$\gamma_2$};

\node at ({2*\xsep},\ylinetop+0.5) {$\M^{(i)}$};

\end{scope}

\coordinate (L) at ({4*\xsep},{(\ylinebot+\ylinetop)/2});
\coordinate (R) at ({4*\xsep+4.75cm},{(\ylinebot+\ylinetop)/2});

\draw[->] ($(L)!0.25!(R)$) -- ($(L)!0.75!(R)$);

\begin{scope}[xshift=4.75cm]

\foreach \x in {0,...,4} {

  \ifnum\x=1
    \node at ({\xsep*\x},0.05) {$j$};
  \fi

  \ifnum\x=3
    \node at ({\xsep*\x},0.05) {$i$};
  \fi

  \ifnum\x=4
    \node at ({\xsep*\x},0.05) {$i+1$};
  \fi

  \node at ({\xsep*\x},\ylabel) {$\bullet$};
  \draw[dotted] ({\xsep*\x},\ylinebot) -- ({\xsep*\x},\ylinetop);
}

\draw[->, >=stealth, shorten >=4pt, shorten <=4pt] (0,\yarrow) -- ({\xsep*1},\yarrow);
\draw[<-, >=stealth, shorten >=4pt, shorten <=4pt] ({\xsep*1},\yarrow) -- ({\xsep*2},\yarrow);
\draw[<-, >=stealth, shorten >=4pt, shorten <=4pt] ({\xsep*2},\yarrow) -- ({\xsep*3},\yarrow);
\draw[<-, >=stealth, shorten >=4pt, shorten <=4pt] ({\xsep*3},\yarrow) -- ({\xsep*4},\yarrow);

\mynode{name=a_1, ltxt=$ $, lpos=above left, fill=white, x=0, y=2.75, xs=\xscale, ys=\yscale, isep=\nodesize}

\mynode{name=b_1, ltxt=$a$, lpos=below left, fill=white, x={\xsep*1}, y=2.75, xs=\xscale, ys=\yscale, isep=\nodesize}

\mynode{name=c_1, ltxt=$ $, fill=white, x={\xsep*2}, y=2.75, xs=\xscale, ys=\yscale, isep=\nodesize}
\mynode{name=c_2, ltxt=$ $, fill=white, x={\xsep*2}, y=1.5, xs=\xscale, ys=\yscale, isep=\nodesize}

\mynode{name=d_1, ltxt=$v_2$, lpos=below right, fill=white, x={\xsep*3}, y=2.75, xs=\xscale, ys=\yscale, isep=\nodesize}
\mynode{name=d_2, ltxt=$v_1$, lpos=above right, fill=white, x={\xsep*3}, y=1.5, xs=\xscale, ys=\yscale, isep=\nodesize}

\mynode{name=e_1, ltxt=$ $, lpos=right,, fill=white, x={\xsep*4}, y=2.75, xs=\xscale, ys=\yscale, isep=\nodesize}

\draw[f_arrow] (a_1) to (b_1);

\draw[f_arrow] (c_1) to (b_1);

\draw[f_arrow] (d_1) to (c_1);
\draw[f_arrow] (d_2) to (c_2);

\draw[f_arrow] (e_1) to (d_1);

\node at ({2*\xsep},\ylinetop+0.5) {$\M^{(i+1)}$};

\end{scope}

\end{tikzpicture}
\end{center}

\vspace{0.5cm}

\paragraph{Termination and output.}
After processing the final level $n$, the merge graph $\mathcal M^{(n)}$ is an $n$-frontier forest. As $\Sb$ is elementary by assumption, Lemma~\ref{lem:forests} implies that $\mathcal M^{(n)}$ is a barcode graph. By the invariant kept up during the algorithm, $k[\Sb] \cong k[\M^{(n)}]$ and the endpoints of the bar components of $\M^{(n)}$ determine the barcode of $k[\Sb]$ by Proposition~\ref{prop:barcode-from-graph}. Therefore the multiset $\mathcal B_{\mathrm{out}}$ contains exactly the intervals  corresponding to the decomposition of $k[\Sb]$, and we return 
\[
\mathcal{B}_{\mathrm{out}} = \mathrm{Bar}(k[\mathbb S]).
\]

\paragraph{Correctness.} The correctness of the algorithm was established in~\cite[Sec.~3]{deyComputingZigzagPersistence2021} for coefficients in $\mathbb{Z}/2\mathbb{Z}$. By Lemma~\ref{lem:field-indep-linearisations}, the barcode of $k[\Sb]$ is independent of the choice of field $k$. It follows that the algorithm correctly computes $\mathrm{Bar}(k[\Sb])$ for an arbitrary field $k$.

\paragraph{Complexity.} Given an elementary formigram $\Ss:\mathcal{P}_\tau \to \cat{set}$ over zigzag type $\tau$ of length $|\tau|=n$, the algorithm can be implemented using the \textit{mergeable trees} data structure of~\cite{georgiadisDataStructuresMergeable2011} with a complexity of $\mathcal{O}(n\log(n))$~\cite{deyComputingZigzagPersistence2021}\footnote{The complexity given in~\cite{deyComputingZigzagPersistence2021} also includes a term accounting for the fact that the formigram is build from a graph $G$. In our setting, we assume that the formigram is provided as an input, c.f. Section~\ref{sec:pipeline} for the complexity of the full pipeline.}.

\subsection{Elementary Refinements}\label{sec:elem-refin}

The setting of binary videos now presents us with a practical problem: the formigrams constructed in Section~\ref{sec:frames-to-formi} are in general not elementary, as merging and splitting vertices may occur with multiplicity greater than one.

The goal of this section is twofold. First, we show that any formigram $\mathbb S$ admits a (generally non-unique) refinement to an elementary formigram $\widetilde{\mathbb S}$ obtained by inserting intermediate spaces. Second, we prove that the barcode of $k[\mathbb S]$ can be recovered from that of $k[\widetilde{\mathbb S}]$, independently of the chosen refinement.

At the end of the section, we will see that this allows us to apply the decomposition algorithm of~\cite{deyComputingZigzagPersistence2021} for general formigrams, such as those arising from binary videos.

\begin{definition}
Let $\tau$ and $\tilde\tau$ be zigzag types with vertex sets $\{1,\dots,n\}$ and $\{1,\dots,m\}$, respectively. A formigram $\widetilde{\mathbb S}: \mathcal P_{\tilde\tau} \to \cat{set}$ is a \emph{refinement} of a formigram $\mathbb S: \mathcal P_\tau \to \cat{set}$ if there exist indices
\[
1 = p_1 < p_2 < \cdots < p_n = m
\]
such that:
\begin{itemize}
    \item $\widetilde{S}_{p_i} = S_i$ for all $i=1,\dots,n$;
     \item For each morphism $e:i\to j$ in $\mathcal{P}_\tau$, there exists a morphism $\tilde e:p_i\to p_j$ in $\mathcal{P}_{\tilde\tau}$ such that:
    $$\Sb[e] = \widetilde\Sb[\tilde e]:S_i\to S_j. $$
\end{itemize}
\end{definition}

\begin{definition}\label{def:elem-ref}
A refinement $\widetilde{\mathbb S}$ of a formigram $\mathbb S$ is called \emph{elementary} if $\widetilde{\mathbb S}$ is elementary.
\end{definition}

\begin{lemma}\label{lem:elem-ref-always-exists}
Every formigram $\mathbb S : \mathcal P_\tau \to \cat{set}$ admits an elementary refinement.
\end{lemma}

\begin{proof}
We show that each structure map of $\mathbb S$ can be decomposed into elementary maps, and then insert the corresponding intermediate levels to obtain a refinement.

Let $f:A \to B\in\Hom_{\cat{set}} (A,B)$. We can always factor $f$ as
\[
A \overset{\pi_f}{\twoheadrightarrow}  \im(f) \overset{\iota_f}{\hookrightarrow} B,
\]
where $\pi_f$ is the canonical surjection and $\iota_f$ is the inclusion. Next, choose an ordering $\{b_1,\dots,b_{|B|}\}$ of $B$ such that $b_i\in\im (f)$ for $i\leq |\im (f)|$. To decompose $\pi_f$, we choose an ordering of each fibre of $f$
\[
f^{-1}(b_i) = \{a_{i,1}, \dots, a_{i,k_i}\}, \quad \text{for $b_i\in\im(f)$}.
\]
Proceeding in order of the fibres, for each $1 \leq i \leq |\im(f)|$ we identify the elements of $f^{-1}(b_i)$ one pair at a time:
\[
a_{i,1} \sim a_{i,2}, \quad (a_{i,1}\sim a_{i,2}) \sim a_{i,3}, \quad \dots
\]
This yields a sequence of maps
\[
A = \widetilde A_0 \twoheadrightarrow \widetilde A_1 \twoheadrightarrow \cdots \twoheadrightarrow \widetilde A_r = \im(f),
\]
where each map identifies exactly two elements in a common fibre, and is the identity elsewhere. By construction, we have that $\widetilde A_i \twoheadrightarrow \widetilde A_{i+1}$ is surjective, and that $|\widetilde A_{i+1}| = |\widetilde A_i|-1$. To decompose $\iota_f$, write $B \setminus \im(f) = \{b_{|\im(f)|+1},\dots,b_{|B|}\}.$ We then construct a sequence
\[
\im(f) = \widetilde B_0 \hookrightarrow \widetilde B_1 \hookrightarrow \cdots \hookrightarrow \widetilde B_s = B,
\]
where $\widetilde B_{\ell+1} = \widetilde B_\ell \cup \{b_{|\im(f)|+\ell+1}\}$. Each inclusion is elementary. Thus $f$ factors as a composition of elementary maps.

Applying this construction to each structure map of $\mathbb S$ and inserting the corresponding intermediate levels yields a refinement $\widetilde{\mathbb S}$ of $\mathbb S$ whose structure maps are all elementary. Hence $\widetilde{\mathbb S}$ is an elementary refinement.

\end{proof}

The explicit construction of an elementary refinement in the proof of Lemma~\ref{lem:elem-ref-always-exists} depended on a choice of ordering on the sets $S_i$, which is in general non-canonical. Nevertheless, as we now show, the barcode of a formigram can be recovered from any elementary refinement. To this end, we introduce a projection of interval endpoints from the refined index set to the original one.

\begin{definition}
Let $\widetilde{\mathbb S} : \mathcal P_{\tilde\tau} \to \cat{set}$ be a refinement of 
$\mathbb S : \mathcal P_\tau \to \cat{set}$, with indices
\(
1 = p_1 < p_2 < \cdots < p_n = m.
\)
Define maps $\lceil \cdot \rceil, \lfloor \cdot \rfloor : \tilde\tau \to \tau$ by
\[
\lceil t \rceil := \min\{i \in \tau \mid t \le p_i\}, 
\qquad
\lfloor t \rfloor := \max\{i \in \tau \mid p_i \le t\}.
\]
The \emph{projection of an interval} $\langle a,b\rangle$ in $\tilde\tau$ is defined as
\[
\langle a,b\rangle \mapsto \langle \lceil a \rceil, \lfloor b \rfloor \rangle,
\]
whenever $\lceil a \rceil \le \lfloor b \rfloor$, and is discarded otherwise.
\end{definition}

The following lemma is a translation of~\cite[Lem. 3.3]{oudotZigzagZoologyRips2013} to our setting.

\begin{lemma}\label{lem:barcode-refinement-projection}
Let $\widetilde{\mathbb S} : \mathcal P_{\tilde\tau} \to \cat{set}$ be a refinement of 
$\mathbb S : \mathcal P_\tau \to \cat{set}$. Then $\mathrm{Bar}(k[\mathbb S])$ is obtained from $\mathrm{Bar}(k[\widetilde{\mathbb S}])$ by projecting intervals in the sense of the preceding definition.
\end{lemma}

Given a fixed formigram, the next results give a bound on the size of elementary refinements.

\begin{definition} Let \(\mathbb S:\mathcal P_\tau\to\cat{Set}\) be a formigram and let \(M=G(\mathbb S)\) be its associated merge graph. For a vertex \(v\in M\), define \[ b(v):= \begin{cases} 1,& \deg_l(v)=0,\\ 0,& \text{otherwise,} \end{cases} \qquad d(v):= \begin{cases} 1,& \deg_r(v)=0,\\ 0,& \text{otherwise,} \end{cases} \] and recall the merging and splitting multiplicities $\merge(v)=\deg^l(v)-1$ and $\mathrm{split} (v) = \deg^r(v)-1$. The \emph{total event multiplicity} of \(\mathbb S\) is \[ E(\mathbb S) := \sum_{v\in M} \bigl(b(v)+d(v)+m(v)+s(v)\bigr). \] Thus births and deaths contribute one, while a \(k\)-fold merge or split contributes \(k-1\). \end{definition}

\begin{proposition}\label{prop:elementary-refinement-bound}
Let \(\mathbb S:\mathcal P_\tau\to\cat{Set}\) be a formigram. Then \(\mathbb S\) admits an elementary refinement \[ \widetilde{\mathbb S}:\mathcal P_{\widetilde\tau}\to\cat{Set} \] such that \[ |\widetilde\tau| \leq |\tau|+E(\mathbb S). \] \end{proposition} \begin{proof} Let \(M=G(\mathbb S)\). Each birth or death vertex of \(M\) can be resolved by inserting a single elementary birth or death. Similarly, if \(v\) is a \(k\)-fold merging vertex, then \(v\) can be resolved into \(k-1\) pairwise merges by inserting intermediate levels. The same argument applies to \(k\)-fold splitting vertices, which can be resolved into \(k-1\) pairwise splits. Performing this resolution for every non-elementary event of \(M\) produces a refinement \(\widetilde{\mathbb S}\) whose structure maps are elementary. Since each elementary event contributes at most one additional level, and since the original levels of \(\tau\) are retained, the refined zigzag type satisfies \[ |\widetilde\tau| \leq |\tau|+E(\mathbb S). \] \end{proof}

\paragraph{Algorithmic consequences.}
Proposition~\ref{prop:elementary-refinement-bound} provides a conceptual
reduction of the interval decomposition problem for arbitrary formigrams to the
elementary case: one may first construct an elementary refinement
\(
\widetilde{\mathbb S}
\)
and then apply the decomposition algorithm to
\(
k[\widetilde{\mathbb S}]
\).
By Lemma~\ref{lem:barcode-refinement-projection}, the barcode of
\(
k[\mathbb S]
\)
is recovered by projecting the barcode of
\(
k[\widetilde{\mathbb S}]
\)
back to the original index set. Explicitly materialising the refinement is, however, unnecessary in practice.
Instead, one may resolve non-elementary events on-the-fly. For example, a
\(k\)-fold merge may be replaced by an arbitrary sequence of \(k-1\) pairwise
merges, and similarly for splits. Any such choice determines an elementary
refinement and therefore yields the same projected barcode.

The complexity of the decomposition stage is therefore governed by the total
event multiplicity \(E(\mathbb S)\). Since
\(
|\widetilde\tau|
=
O\!\left(|\tau|+E(\mathbb S)\right),
\)
the interval decomposition can be computed in
\[
\mathcal{O}\!\left(
(|\tau|+E(\mathbb S))
\log(|\tau|+E(\mathbb S))
\right)
\]
time. Moreover, bijective structure maps do not affect the interval decomposition and
may be contracted beforehand in linear time. After this contraction, every remaining
transition contains at least one elementary event, so that the effective size
of the refined formigram is \(O(E(\mathbb S))\). Consequently, the
decomposition stage runs in
\begin{equation}
\mathcal{O}\!\left(E(\mathbb S)\log (E(\mathbb S))\right)    
\end{equation}
time.



\section{Dualities}\label{sec:dualities}

Zigzag persistence exhibits several dualities that can be exploited computationally. 
The first addresses whether the union and intersection constructions encode different topological information.

\subsection{Correspondence between Intersection and Union Construction}

Given a binary video $\mathcal V =(I_i)^n_{i=1}$, a natural question to ask is whether the barcodes of the union and intersection constructions\footnote{c.f. Definition~\ref{def:top-constr}} $\mathbb F^\cup_{\mathcal V}$ and $\mathbb F^\cap_{\mathcal V}$ (respectively $\mathbb B^\cup_{\mathcal V}$ and $\mathbb B^\cap_{\mathcal V}$) are related.

The \textit{Strong Diamond Principle} of~\cite{carlssonZigzagPersistence2010} answers this question in the strongest possible sense by showing that the union and intersection constructions contain equivalent information. We verify that the required hypotheses hold in our setting. Consider the following diagram of topological spaces:

\begin{center}
\begin{tikzcd}[row sep=1em, column sep=1.2em]
& & && A\cup B  && \\
X_1 \arrow[r,leftrightarrow] & \dots \arrow[r,leftrightarrow] & X_{k-2}\arrow[r,leftrightarrow] & A \arrow[ru] & & B \arrow[lu, swap] \arrow[r,leftrightarrow] & X_{k+2} \arrow[r,leftrightarrow] & \dots \arrow[r,leftrightarrow,] &   X_n\\
&& & &  \arrow[lu] A\cap B \arrow[ru, swap]  &&
\end{tikzcd}
\end{center}
where each space is the geometric realisation of a subcomplex of a common cubical complex $K$ (e.g. $F_I$ and $B_I$). Let $\mathbb X^+$ and $\mathbb X^+$  denote the following topological diagrams that differ at a single entry
\begin{align*}
&\mathbb X ^+ :   X_1 \leftrightarrow \dots  \leftrightarrow X_{k-2} \leftrightarrow  A \to A\cup B \leftarrow B\leftrightarrow X_{k+2} \leftrightarrow  \dots\leftrightarrow X_n  \\
&\mathbb X ^- :   X_1 \leftrightarrow \dots  \leftrightarrow X_{k-2} \leftrightarrow  A \leftarrow A\cap B \to B\leftrightarrow X_{k+2} \leftrightarrow  \dots\leftrightarrow X_n. 
\end{align*}
Since both $A$ and $B$ are geometric realisations of subcomplexes of $K$, it follows that
$$ A\cup B \quad \text{ and } \quad  A\cap B $$ are again geometric realisations of subcomplexes of $K$, and that unions and intersections are compatible with the cubical structure. In particular, at the level of cubical chain complexes, we obtain 
$$ C_\bullet (A\cup B) = C_\bullet(A)+ C_\bullet (B), \qquad C_\bullet(A\cap B) = C_\bullet (A)\cap C_\bullet(B), $$
and hence a short exact sequence 
$$0\to C_\bullet (A\cap B) \to C_\bullet (A) \oplus C_\bullet (B) \to   C_\bullet (A\cup B)\to 0.$$
This short exact sequence induces the Mayer--Vietoris long exact sequence in homology. 
$$ \dots \to H_l(A\cup B) \to  H_l(A)\oplus H_l(B) \to H_l(A\cup B) \to \dots $$
We are thus firmly placed in a setting where the strong diamond principle of~\cite{carlssonZigzagPersistence2010} holds.

\begin{theorem}[{The Strong Diamond Principle~\cite[Thm. 5.9]{carlssonZigzagPersistence2010}}]
 Given $\mathbb X^+$ and $\mathbb X^-$ as above, there is a (complete) bijection between the multisets $\mathrm{Bar}(H_\bullet (\mathbb X^+))$ and $\mathrm{Bar}(H_\bullet (\mathbb X^-))$. Intervals are matched as follows:
 \begin{itemize}
 \item $[k,k]\in\mathrm{Bar}(H_{l+1}(\mathbb X^+))$ is matched with $(k,k)\in\mathrm{Bar}(H_l(\mathbb X^-))$.
 \end{itemize}
 In the remaining cases, the matchings preserve the homological degree:

\begin{center}
\renewcommand{\arraystretch}{1.3}
\begin{tabular}{|c c c|}
\hline
$\mathrm{Bar}(H_\ell(\mathbb X^+))$ 
& 
& 
$\mathrm{Bar}(H_\ell(\mathbb X^-))$
\\
\hline
$\langle b,k-1), \; \text{for } b \leq k-1$
& $\leftrightarrow$
& $\langle b,k)$
\\
$\langle b,k], \; \text{for } b \leq k-1$
& $\leftrightarrow$
& $\langle b,k-1]$
\\
$[k,d\rangle , \; \text{for } d \geq k+1$
& $\leftrightarrow$
& $[k+1,d\rangle$
\\
$(k+1,d\rangle , \; \text{for } d \geq k+1$
& $\leftrightarrow$
& $(k,d\rangle $
\\
$\langle b,d\rangle \; \text{in all other cases}$
& $\leftrightarrow$
& $\langle b,d\rangle$
\\
\hline
\end{tabular}
\end{center}
where the endpoint decorations of the bars were defined in Section~\ref{sec:dec-barcode}.
 
\end{theorem}

\begin{corollary}\label{cor:matchings}
Let $\mathcal V = (I_i)_{i=1}^n$ be a binary video. There exist canonical complete matchings
\[
\mathrm{Bar}(H_*(\mathbb F^\cup_{\mathcal V}))
\leftrightarrow
\mathrm{Bar}(H_*(\mathbb F^\cap_{\mathcal V})),
\qquad
\mathrm{Bar}(H_*(\mathbb B^\cup_{\mathcal V}))
\leftrightarrow
\mathrm{Bar}(H_*(\mathbb B^\cap_{\mathcal V})).
\]
\end{corollary}
\begin{proof}
 Since all topological spaces involved are of the form $F_I$ or $B_I$ for some binary image $I$, and these are all geometric realisations of some subcomplex of either $K$ or $K^*$ (see Definition~\ref{def:top-real}), we are in the above setting. Applying the strong diamond principle $n-1$ times, once for each interpolation image index, yields the result.
\end{proof}

For an illustration of how the bars are matched through Corollary~\ref{cor:matchings}, we refer to~\cite[Fig. 7]{carlssonZigzagPersistence2010}.

\subsection{Alexander Duality: Accelerating $H_1$-Zigzag Persistence} \label{sec:alex-duality}

So far, this work has focused on the efficient computation of $H_0$-zigzag persistence. One might therefore wonder whether the computation of $H_1$-zigzag persistence for binary videos can also be accelerated. In general, detecting non-contractible cycles is a difficult global problem that cannot be reduced to simple local adjacency relations in the same way as $H_0$-persistence. Making the computation of $H_1$ computationally effective is usually far from straightforward.

\begin{wrapfigure}{r}{0.45\textwidth}
    \vspace{-20pt}
    \centering
    \includegraphics[width=0.9\linewidth]{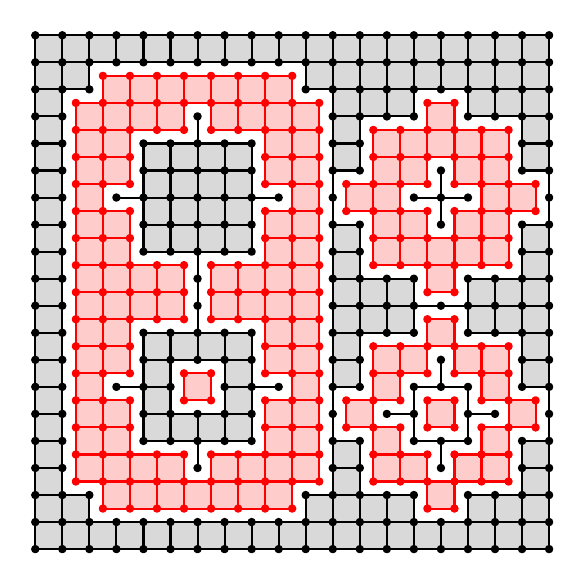}
    \label{fig:nested}
    \vspace{-10pt}
\end{wrapfigure}

One of the first intuitions encountered in algebraic topology is that $H_1$ ``counts holes”. For binary images, this intuition admits a particularly concrete geometric interpretation: a non-contractible loop of one colour encloses a connected component of the opposite colour. This suggests that non-contractible cycles in the foreground should correspond to connected components of the complementary background. In other words, one would like to detect holes simply by studying connected components of the complement, which, as we have seen, can be done in linear time in the number of pixels.

This points towards a relation between $H_0$ and $H_1$ of the two colours. Formulating this correspondence precisely requires some care. Under suitable assumptions on a binary image $I$, which we now make explicit, the authors of~\cite{MR4454782} establish such a relationship in the setting of regular persistence and grey-scale filtrations of binary images.

\begin{definition}[Padded image] \label{def:padded}Let \(I:\Omega_{w,h}\to \{ 0,1\}\) be a binary image. The \emph{padding} of \(I\) is the binary image \( I^\infty: \Omega_{w+2,h+2} \to \{0,1\} \) obtained by adjoining a one-pixel-wide border of foreground pixels around \(I\), and choosing the same foreground connectivity as that of $I$. That is
$$I^\infty(i,j) = \begin{cases}
    1 &\text{if $i\in\{1,w+2 \}$ or $j\in\{1,h+2 \} $}, \\
 I(i-1,j-1) & \text{else.}
\end{cases} $$
\end{definition}

\begin{remark}
For simplicity of exposition, we restrict throughout to padding by foreground pixels. Nevertheless, all subsequent statements admit an analogous formulation for padding with background pixels. Moreover, for ease of notation, we fix a field $k$ and suppress it from our notation of homology and cohomology in this section. All results will be independent of $k$.
\end{remark}

The combinatorial correspondence established in~\cite{MR4454782} admits the following topological formulation. 

\begin{theorem}[Alexander Duality for Binary Images]\label{thm:alexander-duality}
Let $I:\Omega_{w,h}\to \{0,1\}$ be a binary image with foreground connectivity $\kappa\in\{4,8\}$ and background connectivity $\bar\kappa=12-\kappa$. Then
$$\widetilde H_0(F_{I^\infty}) \cong H^1(B_{I^\infty}), \quad \text{and} \quad 
H_1(F_{I^\infty}) \cong H^0(B_{I^\infty}).$$
Moreover, these isomorphisms are natural with respect to inclusions.    
\end{theorem}

\begin{proof}
As the result is essentially contained in~\cite{MR4454782}, we only show how to translate it into our topological setting in Appendix~\ref{app:alex-dual}. The key idea is that the padded image can be compactified to a cubical decomposition of \(S^2\), in which the foreground and background become complementary dual complexes. The claim then follows from standard Alexander duality~\cite[Thm 71.1]{Munkres}.

\end{proof}

\begin{remark}
The asymmetry in the appearance of reduced homology in the statement of Theorem~\ref{thm:alexander-duality} stems from our asymmetric choice of padding with foreground pixels in Definition~\ref{def:padded}. 
\end{remark}

\begin{remark}
Alexander duality has found numerous applications in topological data analysis, where it is frequently exploited to relate topological features of a space to those of its complement. Next to its application to binary images~\cite{MR4454782, garinDualityPersistentHomology2020}, it has been used in~\cite{deyComputingHeightPersistence2018,edelsbrunnerComputationalTopologyIntroduction2010,hartvigsenAllpairsMinCut1994, lenzenPersistentCycleRepresentatives2025, stuckiEfficientBettiMatching2024}. Moreover, it has been implemented in Cubical Ripser~\cite{Kaji2020CubicalRS} to perform $H_2$ computations using $H_0$.
\end{remark}

At the level of zigzag modules arising from binary videos, this has the following deep consequence.

\begin{theorem} \label{thm:dual-barcodes-alexander}
Let $\mathcal V = (I_i)_{i=1}^n$ be a binary video, and assume that each image $I_i$ is padded by a layer of foreground pixels in the sense of Definition~\ref{def:padded}. Then the following barcode correspondences hold:
\[
\begin{array}{l@{\hspace{1cm}} l}
\mathrm{Bar}( H_1(\mathbb B_\mathcal V^\cap))
\xleftrightarrow{\mathrm{rev}}
\mathrm{Bar}( H_0(\mathbb F_\mathcal V^\cup)) \setminus \langle 1,n\rangle
&
\mathrm{Bar}( H_1(\mathbb B_\mathcal V^\cup))
\xleftrightarrow{\mathrm{rev}}
\mathrm{Bar}( H_0(\mathbb F_\mathcal V^\cap)) \setminus \langle 1,n\rangle
\\[1em]
\mathrm{Bar}( H_0(\mathbb B_\mathcal V^\cap))
\xleftrightarrow{\mathrm{rev}}
\mathrm{Bar}( H_1(\mathbb F_\mathcal V^\cup))
&
\mathrm{Bar}( H_0(\mathbb B_\mathcal V^\cup))
\xleftrightarrow{\mathrm{rev}}
\mathrm{Bar}( H_1(\mathbb F_\mathcal V^\cap)).
\end{array}
\]
Here $\xleftrightarrow{\mathrm{rev}}$ denotes a bijective correspondence between intervals obtained by reversing the endpoint decorations while leaving the endpoints unchanged.
\end{theorem}

\begin{proof}
Assuming that the equalities on the left side hold, the equalities on the right follow immediately by the canonical complete matchings between the barcodes of the intersection and union construction of Corollary~\ref{cor:matchings}. 

We first show that $\mathrm{Bar}( H_1(\mathbb B_\mathcal V^ \cap)) = \mathrm{Bar}( H_0(\mathbb F_\mathcal V^ \cup)) \setminus \langle 1,n\rangle$ holds. Recall that the topological diagrams $\mathbb{F}^\cup_\mathcal{V}$ and $\mathbb{B}^\cap_\mathcal{V}$ were defined as follows
\begin{center}
\begin{tikzcd}[row sep=2pt]
\mathbb{F}^\cup_\mathcal{V}: & F_{I_1} \arrow[r, hook] & F_{I_1\vee I_2} & \arrow[l, hook'] F_{I_2} \arrow[r,hook] & \dots \arrow[r,hook]&  F_{I_{n-1}\vee I_n}  &
\arrow[l, hook'] F_{I_n}\\
\mathbb{B}^\cap_\mathcal{V}: & B_{I_1}  & \arrow[l,hook'] B_{I_1\vee I_2} \arrow[r, hook] &  B_{I_2}  & \arrow[l, hook'] \dots & \arrow[l,hook']  B_{I_{n-1}\vee I_n} \arrow[r, hook] &
B_{I_n}
\end{tikzcd}
\end{center}
As the isomorphisms of Theorem~\ref{thm:alexander-duality} are natural with respect to inclusions, they assemble into the following commutative diagram of vector spaces

\vspace{-0.5cm}
\begin{center}
\begin{tikzcd}
\widetilde H_0 (F_{I_1}) \arrow[d]\arrow[r] & \widetilde H_0 (F_{I_1\vee I_2}) \arrow[d] & \arrow[l] \widetilde H_0 (F_{I_2}) \arrow[r] \arrow[d] & \dots \arrow[r]&  \widetilde H_0 (F_{I_{n-1}\vee I_n})  \arrow[d] &
\arrow[l]\widetilde H_0 ( F_{I_n} ) \arrow[d]\\
  H^1(B_{I_1}) \arrow[r] & H^1(B_{I_1\vee I_2}) & \arrow[l]H^1( B_{I_2}) \arrow[r] & \dots \arrow[r]& H^1( B_{I_{n-1}\vee I_n})  &\arrow[l] H^1(B_{I_n})
\end{tikzcd}
\end{center}
in which every vertical map is an Alexander duality isomorphism. Consequently,
\(
\widetilde H_0(\mathbb F_{\mathcal V}^\cup)
\) and \(
H^1(\mathbb B_{\mathcal V}^\cap)
\)
are isomorphic zigzag modules, hence 
$$\mathrm{Bar}(\widetilde H_0(\mathbb F_{\mathcal V}^\cup)) = \mathrm{Bar}(H^1(\mathbb B_{\mathcal V}^\cap)).$$
Since all (co-)homology groups are taken over the field $k$, the cohomological universal coefficient theorem yields natural isomorphisms
\[
H^1(B_I)
\cong
\operatorname{Hom}_k(H_1(B_I),k)
\]
for any binary image $I$. These assemble into a zigzag module isomorphism $$H^1(\mathbb B_{\mathcal V}^\cap) \cong \Hom_k(H_1(\mathbb B_{\mathcal V}^\cap), k).$$

Let $\tau$ denote the zigzag type of $H_1(\mathbb B_{\mathcal V}^\cap)$, and $\tau^*$ the zigzag type of $H^1(\mathbb B_{\mathcal V}^\cap)$. Note that $\tau^*$ is obtained from $\tau$ by reversing the edges. Moreover, let 
$$H_1(\mathbb B_{\mathcal V}^\cap)  \cong \bigoplus_{i} \langle b_i, d_i \rangle $$ be an interval decomposition of $H_1(\mathbb B_{\mathcal V}^\cap)$. Since $\Hom(-,k)$ commutes with finite direct sums in the category of zigzag modules, dualisation preserves the endpoints of the barcodes
\begin{align*}
\Hom_k(H_1(\mathbb B_{\mathcal V}^\cap), k ) &= \Hom\left(\bigoplus_{i} \mathbb I_\tau\langle b_i, d_i \rangle , k\right)\\
&\cong \bigoplus_i \Hom(\mathbb I_\tau\langle b_i, d_i \rangle, k )  \\
&\cong \bigoplus_i \mathbb I_{\tau^*}\langle b_i, d_i \rangle
\end{align*}
Hence, and interval decomposition of $H_1(\mathbb B_{\mathcal V}^\cap)$ induces an interval decomposition of $H^1(\mathbb B_{\mathcal V}^\cap)$. Thus, we have the following equality of undecorated barcodes $$\mathrm{Bar}(H_1(\mathbb B_{\mathcal V}^\cap)) = \mathrm{Bar}(H^1(\mathbb B_{\mathcal V}^\cap)).$$ 
Since $\tau$ and $\tau^*$ are obtained from each other via edge-reversal, the decorations can also be recovered by flipping them, i.e.\ open endpoints become closed, and vice versa.

Hence
\[
\mathrm{Bar}\!\left(\widetilde H_0(\mathbb F_{\mathcal V}^\cup)\right)
=
\mathrm{Bar}\!\left(H_1(\mathbb B_{\mathcal V}^\cap)\right).
\]
Finally, since the foreground padding makes the boundary pixel layer a connected component of $F_{I^\infty}$ that persists throughout the entire diagram, the barcodes of the reduced and unreduced $H_0$-zigzag agree up to a single full interval $\langle 1,n\rangle$, i.e.
\[
\mathrm{Bar}\!\left(\widetilde H_0(\mathbb F_{\mathcal V}^\cup)\right)
=
\mathrm{Bar}\!\left(H_0(\mathbb F_{\mathcal V}^\cup)\right)
\setminus \langle 1,n\rangle.
\]
Gathering the above results, we obtain
\begin{align*}
\mathrm{Bar}\!\left(H_0(\mathbb F_{\mathcal V}^\cup)\right)
\setminus \langle 1,n\rangle &= \mathrm{Bar}(\widetilde H_0(\mathbb F_{\mathcal V}^\cup))\\
&= \mathrm{Bar}(H_1(\mathbb B^\cap_{\mathcal{V}}))\\
&= \mathrm{Bar}(H^1(\mathbb B^\cap_{\mathcal{V}}))
\end{align*}

The equality $\mathrm{Bar}( H_0(\mathbb B_\mathcal V^ \cap)) = \mathrm{Bar}( H_1(\mathbb F_\mathcal V^ \cup))$ follows in the same way, and is even easier as we do not need to consider reduced homology as $H_1(F_{I^\infty}) \cong H^0(B_{I^\infty})$.

\end{proof}

\paragraph{Algorithmic Consequences.} 
Theorem~\ref{thm:dual-barcodes-alexander} shows that, for padded binary videos, the computation of $H_1$-zigzag persistence reduces to the computation of $H_0$-zigzag persistence on the complementary colour. Consequently, all algorithms and complexity bounds developed in the previous sections for connected component tracking immediately extend to the computation of cycle persistence. In particular, the decorated $H_1$-barcode can be recovered in near-linear time.

\section{The Pipeline}\label{sec:pipeline}

We conclude by summarising the complete computational pipeline. Let
\[
\mathcal V=(I_i)_{i=1}^n
\]
be a binary video of resolution $(w,h)$ equipped with a foreground connectivity
$\kappa\in\{4,8\}$.

\begin{enumerate}
\item Construct one of the formigrams
\(
\mathbb S \in
\left\{
\mathbb{F}^\cup_{\mathcal V},
\mathbb{F}^\cap_{\mathcal V},
\mathbb{B}^\cup_{\mathcal V},
\mathbb{B}^\cap_{\mathcal V}
\right\}
\)
with connectivity $\kappa$. This step requires
\(
\mathcal O(nwh)
\)
time.

\item Compute the interval decomposition of the associated
$H_0$-zigzag module
\(
H_0(\mathbb S)
\)
using the decomposition algorithm of Section~\ref{sec:int-decomp}. This step requires
\[
\mathcal O\!\left(E(\mathbb S)\log E(\mathbb S)\right)
\]
time.

\item To compute $H_1$-persistence, apply the same $H_0$-pipeline to the complementary colour with the opposite connectivity and the dual interpolation convention
\[
\cup \longleftrightarrow \cap,
\]
and recover the desired barcode via Theorem~\ref{thm:dual-barcodes-alexander}.
\end{enumerate}

\vspace{0.5cm}
\begin{center}
\begin{tikzpicture}
\node[
    draw,
    rounded corners,
    minimum width=3cm,
    minimum height=1cm,
    align=center
] (A) {Binary video\\ $\mathcal{V} = (I_i)_{i=1}^n$};
\node[
    draw,
    rounded corners,
    minimum width=3cm,
    minimum height=1cm,
    right=of A,
    align=center
] (B) {Formigram\\$\mathbb{F}^\cup_{\mathcal{V}},\
\mathbb{F}^\cap_{\mathcal{V}},\
\mathbb{B}^\cup_{\mathcal{V}},$ or $
\mathbb{B}^\cap_{\mathcal{V}} $};

\node[
    draw,
    rounded corners,
    minimum width=3cm,
    minimum height=1cm,
    right=of B
] (C) {$H_0$-Barcode};

\draw[->] (A) -- (B);
\draw[->] (B) -- (C);
\end{tikzpicture}
\end{center}

\paragraph{Complexity.} Given a formigram $\mathbb{S}$, the total complexity is thus given by 
$$\mathcal{O}(nwh + E(\mathbb S)\log E(\mathbb S) ). $$
In practice, the linear (and embarrassingly parallel) term on the left often dominates, as we show in the next section.

\section{Benchmarking}\label{sec:benchmarking}

We evaluate the computational performance and memory usage of the pipeline described in Section~\ref{sec:pipeline}. The benchmarks are designed to measure:
\begin{itemize}
    \item scaling with image resolution,
    \item the relative cost of connected-component labelling (CCL) and barcode decomposition,
    \item comparison with cubical-complex approaches to zigzag persistence.
\end{itemize}
\noindent All experiments were performed using the Rust implementation of \texttt{ZigVid}~\cite{ZigVidSoftwareRef} on a 2024 MacBook Pro equipped with an Apple M4 processor (10 cores) and 24 GB RAM. Unless otherwise stated, computations were executed on a single thread.

We compare against the implementation named \texttt{ImageZigzag} of~\cite{divasonZigzagPersistenceImage2024}, which constructs the \(V\)-construction for each frame, subdivides the resulting objects into simplicial complexes, and computes zigzag persistence using Dionysus~2~\cite{Dionysus2Ref}.

Moreover, the experiments demonstrate three main computational properties of \texttt{ZigVid}: linear scaling of formigram construction, with negligible cost of barcode decomposition, and near-ideal parallel scaling across frames.

\paragraph{Synthetic Benchmark Data.}
\begin{wrapfigure}{r}{0.4\textwidth}
\vspace{-1em}
\centering
\begin{tikzpicture}
  \node[anchor=south west, inner sep=0] (img) at (0,0)
    {\includegraphics[width=0.4\textwidth]{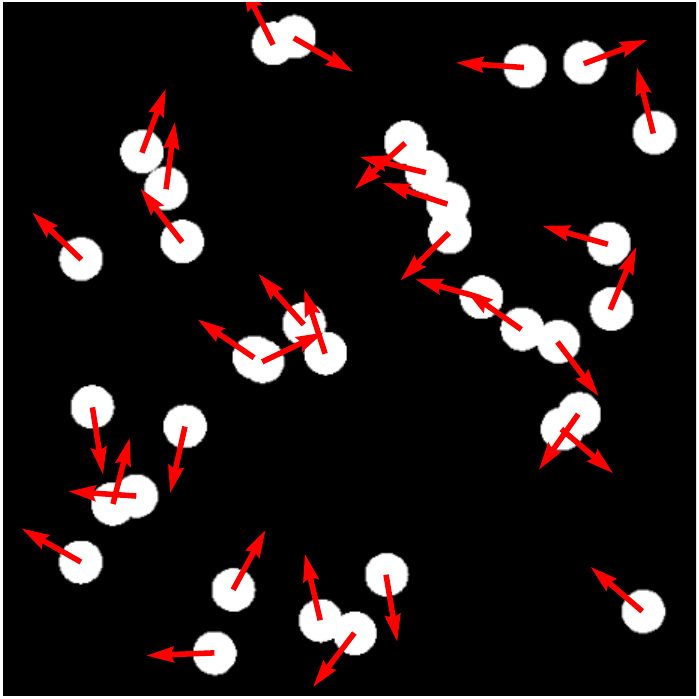}};






\end{tikzpicture}

\label{fig:synthetic-disc-benchmark}
\end{wrapfigure}

To generate videos with controllable topological complexity, we use synthetic binary sequences consisting of \(D\) filled discs of radius \(R\) evolving inside a rectangular domain of size \(w\times h\). A total of \(T\) frames are generated. Each disc moves with constant speed \(s\) and reflects elastically at the image boundary. Intersections between discs produce merging and splitting vertices in the associated formigrams. When \(s=0\), the connected components remain stationary, and no such events occur; increasing \(s\) increases their frequency. A boundary layer of white pixels is additionally included.

\paragraph{Comparison with Cubical Complex Methods.}
Each benchmark video consists of \(T=10\) frames containing \(D=100\) discs. Videos are generated at a base resolution of \(180\times180\) with disc radius \(R=8\) and speed \(s=6\), then downsampled to resolutions \(r\times r\) for
\[
r\in\{30,40,50,\dots,170,180\}.
\]

For each resolution, we perform five runs of both \texttt{ImageZigzag} and \texttt{ZigVid}, computing the \(H_0\)- and \(H_1\)-zigzag barcodes obtained from the union construction. The timing results are shown in Figure~\ref{fig:comparison-cubical}.

\begin{figure}[t]
\centering
\centering
{\large \textbf{Runtime Comparison with Cubical Complexes ($T=10$)}\par}
\vspace{1em}

\begin{minipage}{0.5\textwidth}
\centering
\includegraphics[width=\textwidth]{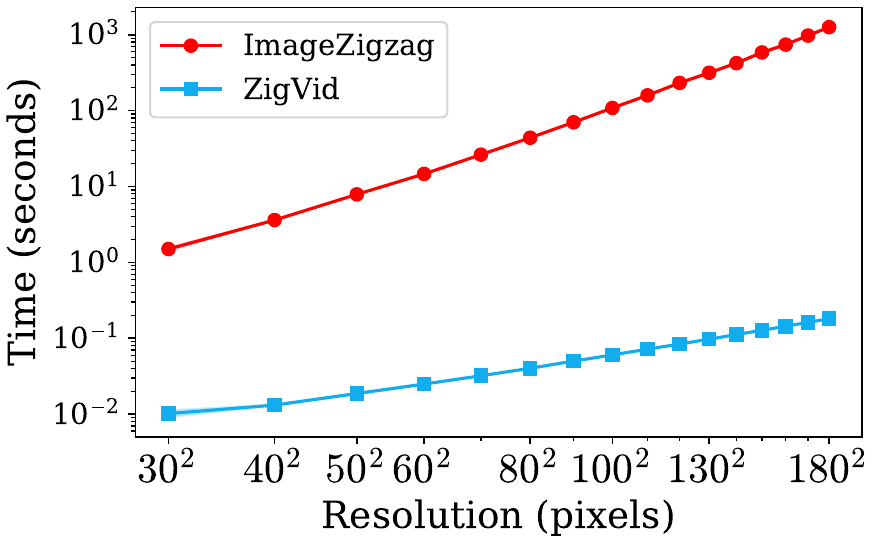}
\end{minipage}
\hfill
\begin{minipage}{0.45\textwidth}
\centering
\small

\vspace{-0.6cm}
\begin{tabular}{rrrr}
\toprule
$r\times r$ & ImageZigzag & ZigVid & Speedup \\
\midrule
$30^2$ & 1.50 s & 0.01 s & 147x \\
$40^2$ & 3.61 s & 0.01 s & 274x \\
$50^2$ & 7.87 s & 0.02 s & 421x \\
 & $\cdots$ & & \\
$150^2$ & 583.58 s & 0.13 s & 4578x \\
$160^2$ & 740.90 s & 0.14 s & 5140x \\
$170^2$ & 975.05 s & 0.16 s & 6015x \\
$180^2$ & 1258.84 s & 0.18 s & 6953x \\
\bottomrule
\end{tabular}

\end{minipage}

\caption{
\textbf{$H_0$- and $H_1$-barcode computation.} Runtime comparison with the cubical-complex pipeline on
synthetic videos consisting of \(T=10\) frames. 5 runs are performed for each method, and the mean is plotted. The shaded area corresponds to the standard error. 
}
\label{fig:comparison-cubical}
\end{figure}

The experiments show markedly different scaling behaviour between the two methods. In the present benchmark regime, cubical-complex methods already become prohibitively expensive at resolutions of \(180\times180\) pixels and \(10\) frames. This substantially limits the range of resolutions for which the cubical-complex pipeline is practically usable.

This is consistent with Section~\ref{sec:curse-of-res}: cubical-complex methods explicitly construct and update complexes whose size grows with image resolution, while \texttt{ZigVid} operates directly on connected-component dynamics encoded by the formigram. We now turn to evaluating the performance in large-scale scenarios enabled by \texttt{ZigVid}, which were not computable with previous methods. We test high-resolution and long-duration videos separately.

\paragraph{Resolution Scaling.}
To benchmark scaling with image resolution, we generate a synthetic video of resolution \(3000\times3000\) consisting of \(T=900\) frames (corresponding to a 30s video at 30 frames per second) and \(D=100\) discs moving with speed \(16\). To compare videos with comparable topological complexity, we downsample this source video to resolutions
\[
r\in\{250,500,750,\dots,2750,3000\}.
\]

For each resolution, we compute the \(H_0\)- and \(H_1\)-zigzag barcodes associated with the union construction. The results are shown in Figure~\ref{fig:res-scaling}. Additionally, the figure separates the runtime contributions of formigram construction and barcode decomposition. The decomposition stage contributes only a negligible fraction of the total runtime and remains essentially constant across resolutions. This is expected, since downsampling primarily affects the geometric realisation of the video, while leaving the underlying connected-component dynamics largely unchanged.

\begin{figure}[t]
\centering
{\large \textbf{Resolution Scaling of Total Barcode Computation ($T=900$)}\par}
\vspace{1em}
\begin{minipage}{0.5\textwidth}
\centering
\includegraphics[width=\textwidth]{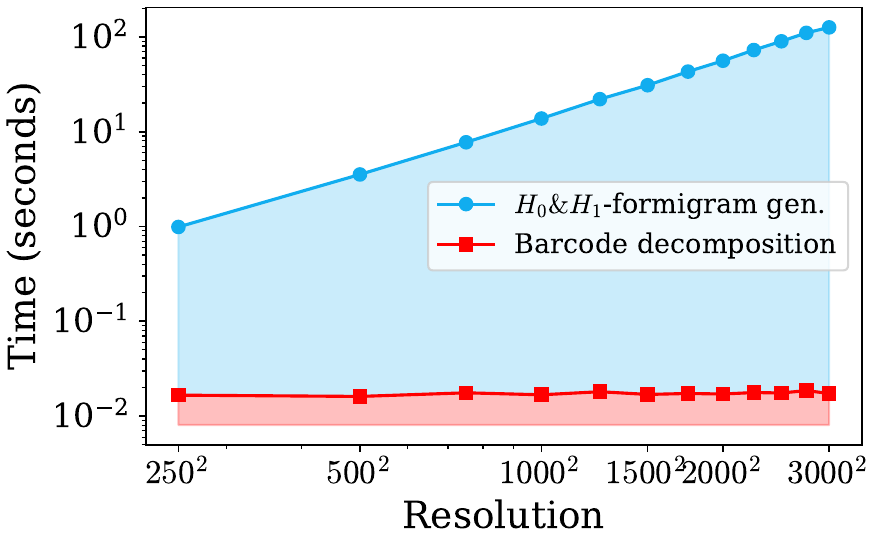}
\end{minipage}
\hfill
\begin{minipage}{0.45\textwidth}
\centering
\small

\vspace{-0.6cm}

\begin{tabular}{rrrr}
\toprule
$r\times r$ & Formigrams & Decomp. & Rel. Cost \\
\midrule
$250^2$ & 0.9717 s & 0.0167 s & 1.71\% \\
$500^2$ & 3.5207 s & 0.0161 s & 0.46\% \\
$750^2$ & 7.7105 s & 0.0176 s & 0.23\% \\
$1000^2$ & 13.7455 s & 0.0168 s & 0.12\% \\
$2000^2$ & 55.9528 s & 0.0171 s & 0.03\% \\
$2500^2$ & 89.7362 s & 0.0176 s & 0.02\% \\
$3000^2$ & 125.8147 s & 0.0173 s & 0.01\% \\

\bottomrule
\end{tabular}

\end{minipage}

\caption{
\textbf{Runtime scaling with image resolution.}
A synthetic video of resolution \(3000\times3000\) is downsampled to varying resolutions, and the \(H_0\)- and \(H_1\)-zigzag barcodes of the resulting 900-frame videos are computed. The plotted values show the mean over 5 runs. The barcode decomposition stage contributes negligibly to the total runtime and exhibits little dependence on image resolution. 
}
\label{fig:res-scaling}
\end{figure}

Consequently, the dominant cost lies in the formigram construction stage, which is \textit{linear} in the number of pixels and \textit{embarrassingly parallel} across frames. This is the key computational feature underlying the scalability of the method and provides the central computational advantage of the pipeline. We now investigate how this can be leveraged through parallelisation.

\paragraph{Real-Time Zigzag Persistence on 4K Video.}
To test the scalability of our method, we generate a \(3840\times2160\) (4K) video consisting of 900 frames, corresponding to 30 seconds at 30 fps, using the same disc parameters as in the resolution scaling benchmark. The results in Figure~\ref{fig:parallel-scaling} demonstrate that real-time barcode computation for high-resolution data is already achievable on modern consumer hardware using 5--6 cores. Moreover, the observed scaling closely follows ideal linear scaling, providing evidence that the pipeline is embarrassingly parallel across frames.

\begin{figure}[t]
\centering
{\large \textbf{Parallel Scaling on 4K Video ($T=900$)}\par}
\vspace{1em}
\begin{minipage}{0.5\textwidth}
\centering
\includegraphics[width=\textwidth]{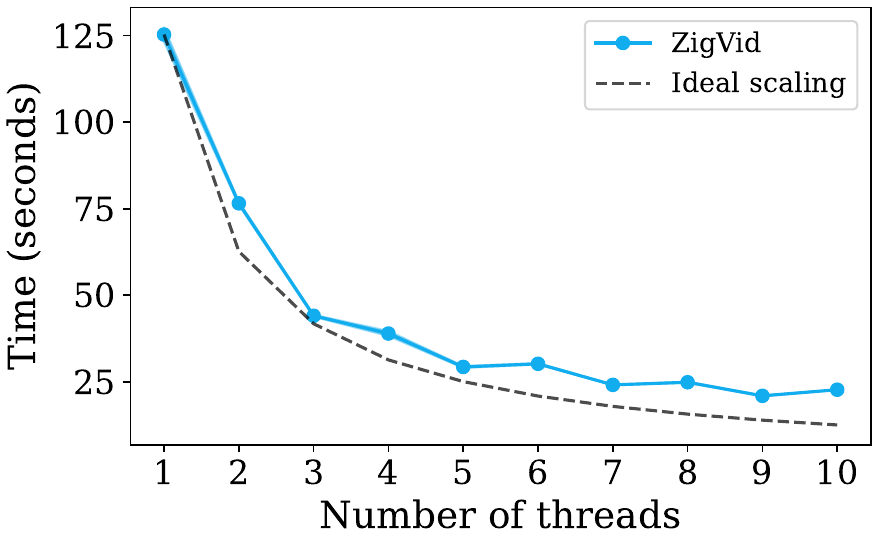}
\end{minipage}
\hfill
\begin{minipage}{0.45\textwidth}
\centering
\small

\vspace{-0.6cm}

\begin{tabular}{cc|cc}
\toprule
Threads & Runtime & Threads & Runtime \\
\midrule
1 & $125.30 s$ & 6 & $30.20 s$ \\
2 & $76.55 s$ & 7 & $24.11 s$ \\
3 & $44.07 s$ & 8 & $24.87 s$ \\
4 & $38.94 s$ & 9 & $20.94 s$ \\
5 & $29.29 s$ & 10 & $22.71 s$ \\

\bottomrule
\end{tabular}

\end{minipage}

\caption{
\textbf{Parallel scaling on 4K video.}
Runtime for computing the \(H_0\)- and \(H_1\)-zigzag barcodes of a \(3840\times2160\), 900-frame video (30 seconds at 30 fps) as a function of thread count. Using 6 threads, the full barcode computation completes in approximately 30 seconds, demonstrating real-time zigzag persistence on 4K video. The dashed curve indicates ideal linear scaling relative to the single-thread runtime.
}
\label{fig:parallel-scaling}
\end{figure}

\section{Conclusion}

We have presented a pipeline for computing the $H_0$- and $H_1$-zigzag persistence of binary videos in near-linear time. The key observation is that the relevant zigzag modules can be encoded by graph structures whose interval decompositions can be computed efficiently. Furthermore, Alexander duality allows the computation of $H_1$-persistence to be reduced entirely to an additional $H_0$ computation on the complementary colour.

In practice, the construction of these graph structures is linear in the number of pixels and embarrassingly parallel, and often constitutes the dominant computational cost. Consequently, the proposed pipeline scales exceptionally well. Our experiments demonstrate that barcode decompositions of high-resolution, high-frame-rate videos---including $4$K data---are feasible even on commodity hardware.

\section*{Acknowledgements}

The author would like to thank Jeffrey Giansiracusa for numerous insightful discussions and for his careful reading of the paper, which greatly improved its exposition. The author also thanks Iolo Jones for many helpful discussions and encouragement.

\newpage

\appendix

\section{Proof of Alexander Duality for Binary Images} \label{app:alex-dual}
\begin{proof}[{Proof of Theorem~\ref{thm:alexander-duality}}]
    We treat the case \(\kappa=8\); the case \(\kappa=4\) is symmetric. Hence, the foreground is modelled by the \(T\)-construction and the background by the \(V\)-construction:
\[
F_{I^\infty}=|T_{I^\infty}(1)|,
\qquad
B_{I^\infty}=|V_{I^\infty}(0)|.
\]

Recall that for the padded image $I^\infty$, the cubical complex $K$ was generated by the 2-cells
\begin{equation*} 
Q_{i,j}\coloneqq\left[i-\frac{1}{2}, i+\frac{1}{2}\right]\times\left[j-\frac{1}{2},j+\frac{1}{2}\right], \quad \text{with $(i,j)\in \Omega_{w+2,h+2}$,}
\end{equation*}
and that the subcomplex $T_{I^\infty}(1)$ of $K$ was generated by cells $Q_{i,j}$ with $I^\infty(i,j)=1$.

Notice that $\partial|K|$ is connected, and that since  $I^\infty:\Omega_{w+2, h+2} \to \{0, 1\}$ contains a boundary layer of foreground pixels, we have that $\partial|K|\subseteq \partial|T_{I^\infty}(1)|$. Moreover, $|B_{I^\infty}|$ is contained in the convex hull of $|F_{I^\infty}|\subset \mathbb R^2$.

One of the key insights of~\cite{MR4454782} is that one can adjoin a 2-cell $\sigma^\infty$ to $K$, by homeomorphically attaching the boundary of $\sigma^\infty$ along $\partial |K|$. The resulting complex $K^\infty\coloneq K \sqcup_{\partial |K|} \sigma^\infty$ then satisfies $|K^\infty|\cong S^2$ .

Since \(\partial |K|\subseteq \partial|T_{I^\infty}(1)|\), the same exterior cell $\sigma^\infty$ may be attached to the foreground complex to form the complex:
\[
T_{I^\infty}(1)\cup_{\partial |K|}\sigma^\infty
\subseteq K^\infty .
\]
Thus \(T_{I^\infty}(1)\cup_{\partial |K|}\sigma^\infty\) is the foreground complex viewed as a subcomplex
of the compactified sphere. As no pixel of $(I^\infty)^{-1}(0)$ lies in the boundary pixel layer, $V_{I^\infty}(0)$ can also be seen as a subcomplex of $K^\infty$ without changes.

In \cite{MR4454782}, the authors then show that \(T_{I^\infty}(1)\cup_{\partial |K|}\sigma^\infty\) and $V_{I^\infty}(0)$ form a dual cubical cell decomposition of $K^\infty$. It follows from the standard complement-dual-cell retraction
(equivalently, Lemma~70.1 of~\cite{Munkres} after barycentric
subdivision) that there exists a deformation retract
\begin{equation}\label{eq:before}
|K^\infty|\setminus |V_{I^\infty}(0)|  
\simeq
 |T_{I^\infty}(1)\cup_{\partial |K|}\sigma^\infty|. 
\end{equation}
For $i\in\{0,1\}$, we have 
\begin{align*}
\widetilde H_i(|T_{I^\infty}(1)\cup_{\partial |K|}\sigma^\infty| ) &\cong \widetilde H_i(|K^\infty|\setminus |V_{I^\infty}(0)| ) \\
&\overset{}{\cong} \widetilde H_i(S^2\setminus |V_{I^\infty}(0)|)  &&\tag*{By Equation (\ref{eq:before}),}\\
&\overset{}{\cong} \widetilde H^{1-i}(|V_{I^\infty}(0)|)&&\tag*{By Alexander duality,~\cite[Thm 71.1]{Munkres}.}
\end{align*}
Hence
$$
\widetilde H_0(|T_{I^\infty}(1)\cup_{\partial |K|}\sigma^\infty|) \cong  H^{1}(|V_{I^\infty}(0)|), \quad \text{and}\quad  H_1(|T_{I^\infty}(1)\cup_{\partial |K|}\sigma^\infty|) \cong  \widetilde H^{0}(V_{I^\infty}(0)).
$$
Undoing the compactification then leads to 
$$\widetilde H_0(F_{I^\infty}) \cong H^1(B_{I^\infty})\ \quad \text{and} \quad H_1(F_{I^\infty}) \cong H^0(B_{I^\infty}),$$
where we pass from reduced to unreduced homology on the right, as undoing the identification creates a non-contractible cycle in the presence of a connected component of the background. The naturality statement is Corollary 72.4 of~\cite{Munkres}.
\end{proof}

\section{Some Category Theory} \label{appendix:some-cat}

In this section, we recall some notions of category theory. 

\begin{definition}
Let $\mathcal C$ be a category, and let $\mathbb D : \mathcal P_\tau \to \mathcal C$ be a $\mathcal C$-valued diagram over a zigzag type $\tau$. A \emph{cone} over $\mathbb D$ consists of:
\begin{itemize}
    \item an object $N\in\mathcal C$, and
    \item a family of morphisms $$\phi_i: N \to \mathbb D(i),\quad \text{for each } i\in \mathcal P_\tau,$$
    such that for every morphism $e:i\to j$ in $\mathcal P_\tau$,
    $$\mathbb D[e] \circ \phi_i = \phi_j. $$
\end{itemize}
\end{definition}

The dual notion is given by the following definition.
\begin{definition}
Let $\mathcal C$ be a category, and let $\mathbb D : \mathcal P_\tau \to \mathcal C$ be a $\mathcal C$-valued diagram over a zigzag type $\tau$. A \emph{cocone} over $\mathbb D$ consists of:
\begin{itemize}
    \item an object $M\in\mathcal C$, and
    \item a family of morphisms $$\psi_i: \mathbb D(i)  \to M,\quad \text{for each } i\in \mathcal P_\tau,$$
    such that for every morphism $e:i\to j$ in $\mathcal P_\tau$,
    $$\mathbb  \psi_i = \psi_j\circ \mathbb D[e]. $$
\end{itemize}
\end{definition}

In some categories, there exist distinguished cones and cocones.

\begin{definition}
Let \( \mathbb D : \mathcal{P}_\tau \to \mathcal{C} \) be a diagram. A \emph{limit} of \( \mathbb D \) is a cone $(L, (\pi_i)_{i\in\mathcal P_\tau})$
for which the following universal property holds: for any cone \( (N, (\phi_i)_{i\in\mathcal P_\tau}) \) there exists a unique morphism \( \eta : N \to L \) such that
\[
\pi_i \circ \eta = \phi_i \quad \text{for all } i \in \mathcal{P}_\tau.
\]
\end{definition}

\begin{definition}
Let \( \mathbb D : \mathcal{P}_\tau \to \mathcal{C} \) be a diagram. A \emph{colimit} of \( \mathbb D \) is a cocone \( (C, (\iota_i)_{i \in \mathcal{P}_\tau}) \)
for which the following universal property holds: for any cocone \( (M, (\psi_i)_{i \in \mathcal{P}_\tau}) \) there exists a unique morphism \( \nu_M : C \to M \) such that
\[
\nu_M \circ \iota_i = \psi_i \quad \text{for all } i \in \mathcal{P}_\tau.
\]
\end{definition}

\begin{remark}\label{rmk:lim-colim}
Limits and colimits do not necessarily exist for a given diagram. However, whenever they do exist, they are unique up to unique isomorphism by the universal property. For a diagram \( \mathbb D : \mathcal{P}_\tau \to \mathcal{C} \), we denote the limit by \( \varprojlim \mathbb D \) and the colimit by \( \varinjlim \mathbb D \).
We will say that \( \mathcal{C} \) admits limits (resp.\ colimits) of shape \( \mathcal{P}_\tau \) if every diagram \( \mathbb D : \mathcal{P}_\tau \to \mathcal{C} \) has a limit (resp.\ colimit).
\end{remark}

\begin{proposition}
Let $\tau$ be a zigzag type and $k$ be a field. Then both $\cat{set}$ and $\cat{vect_k}$ admit limits and colimits of shape $\mathcal P_\tau$:

\begin{enumerate}
\item Let \( \mathbb S : \mathcal{P}_\tau \to \cat{set} \) be a diagram. The limit and colimit exist and are given by
\[
\varprojlim \mathbb S
=
\left\{ (s_i)_{i \in \mathcal{P}_\tau} \in \prod_{i \in \mathcal{P}_\tau} \mathbb S(i)
\;\middle|\;
\mathbb S[e](s_i) = s_j \text{ for all } e:i\to j \right\},
\]
and
\[
\varinjlim \mathbb S
=
\left( \bigsqcup_{i \in \mathcal{P}_\tau} \mathbb S(i) \right)\Big/ \sim,
\]
where \( \sim \) is the equivalence relation generated by
\[
s_i \sim \mathbb S[e](s_i), \quad \text{for all } e:i\to j.
\]

\item Let \( \mathbb V : \mathcal{P}_\tau \to \cat{vect_k} \) be a diagram. The limit and colimit exist and are given by
\[
\varprojlim \mathbb V
=
\left\{ (v_i)_{i \in \mathcal{P}_\tau} \in \prod_{i \in \mathcal{P}_\tau} \mathbb V(i)
\;\middle|\;
\mathbb V[e](v_i) = v_j \text{ for all } e:i\to j \right\},
\]
and
\[
\varinjlim \mathbb V
=
\left( \bigoplus_{i \in \mathcal{P}_\tau} \mathbb V(i) \right)\Big/
\left\langle v_i - \mathbb V[e](v_i) \;\middle|\; e:i\to j \right\rangle.
\]
\end{enumerate}
\end{proposition}



\begin{remark} \label{rmk:cc-as-colim}
Given a formigram $\mathbb{S}:\mathcal P_\tau \to\cat{set}$ over a zigzag type $\tau$ of length $n$, these constructions possess a particularly nice interpretation at the level of their associated merge graph $\mathcal G(\mathbb S)$. Elements of the limit correspond to chains of vertices $(s_1,\dots,s_n)$ such that consecutive vertices are connected by an edge. Note that if a fibre $S_i=\emptyset$, the direct product $\prod_i S_i$ is also empty. The colimit, on the other hand, identifies vertices that are connected through the graph, and can thus be identified with the set of connected components of $\mathcal M$.
\end{remark}

\section{Proof of Proposition~\ref{prop:fields-and-setmaps}} \label{app:proof-of-prop}
Before we start the proof of Proposition~\ref{prop:fields-and-setmaps}, we need some auxiliary results. For a given formigram $\mathbb S:\mathcal P_\tau \to \cat{set}$, and field $k$, let $$(\varprojlim \mathbb S, (\pi_i)_{i\in\tau}), \quad \text{and} \quad (\varinjlim \mathbb S, (\iota_i)_{i\in\tau})$$ denote the limit and colimit of the diagram $\mathbb S$ in $\cat{set}$. Given the linearisation functor $k[-]:\cat{set}\to\cat{vect_k}$, we denote the $\cat{vect_k}$-valued limit and colimit of the diagram $k[\mathbb S]$ by $$(\varprojlim k[\mathbb S], (\Pi_i)_{i\in\tau}), \quad \text{and} \quad (\varinjlim k[\mathbb S], (I_i)_{i\in\tau}).$$
These fit into the following commutative diagrams
\begin{center}
\begin{tikzcd}[column sep=small,]
&& \arrow[lld, "\pi_1",swap]  \varprojlim \mathbb S \arrow[d, "\pi_i",swap]   \arrow[rrd, "\pi_n"]   \\ 
S_1\arrow[drr, "\iota_1", swap] \arrow[r, <->] & \dots \arrow[r, <->] & \arrow[d,"\iota_i",swap]  S_i \arrow[r, <->] & \dots \arrow[r, <->] & S_n \arrow[lld,"\iota_n"]\\
&&      \varinjlim \mathbb S
\end{tikzcd}
\hspace{0.5cm}
\begin{tikzcd}[column sep=small,]
&& \arrow[lld, "\Pi_1",swap]  \varprojlim k[\mathbb S] \arrow[d, "\Pi_i",swap]   \arrow[rrd, "\Pi_n"]   \\ 
k[S_1]\arrow[drr, "I_1", swap] \arrow[r, <->] & \dots \arrow[r, <->] & \arrow[d,"I_i",swap]  k[S_i] \arrow[r, <->] & \dots \arrow[r, <->] & k[S_n] \arrow[lld,"I_n"]\\
&&      \varinjlim k[ \mathbb S]
\end{tikzcd}
\end{center}

\noindent Recall that for every $i\in\tau$, the morphisms $\psi^{(i)}_{k[\mathbb S]}=I_i\circ\Pi_i: \varprojlim k[\mathbb S] \to \mycolim k[\Sb]$ were shown to coincide. Denote this morphism by $\psi_{k[\mathbb S]}$.

\begin{lemma}\label{lem:sections}
Let $\mathbb S :\mathcal P_\tau \to \cat{set}$ be a formigram over a zigzag type $\tau$ of length $n$ and let $k$ be a field. Suppose that the colimit consists of a single element, i.e.\ $\mycolim \Sb = \{ C\}$. Then
$$\varprojlim \mathbb S \neq \emptyset \iff \varprojlim k[\mathbb S] / \ker(\psi_{k[\mathbb S]}) \neq \{ 0 \}.$$ 
\end{lemma}
\begin{proof}
If one of the $S_i=\emptyset$, then $\mylim \Sb = \emptyset$, and $\ker(\psi_{k[\Sb]}) = \mylim k[\Sb]$, thus the statement holds. We can thus assume that all $S_i$ are non-empty from now on. Given $s\in S_i$, we denote the induced basis element of $k[S_i]$ by $e^{(i)}_s$. One direction is easy to prove. If $(s_1,\dots, s_n)\in \mylim \Sb$, then $(e^{(1)}_{s_1},\dots,e^{(n)}_{s_n})$ is a non-trivial element of $\varprojlim k[\mathbb S] / \ker(\psi_{k[\mathbb S]})$.

Conversely, we start with the following observation. Since $k$ is left-adjoint, it commutes with colimits, and hence $k[\mycolim \Sb] \cong \mycolim k[\Sb]$. Therefore, $\mycolim k[\Sb]$ is one-dimensional, and we denote its basis vector by $E_C$.

Furthermore, let $v = (v^{(1)},\dots, v^{(n)} )\in \mylim k[\mathbb S]$, such that $\psi_{k[\Sb]}(v)\neq 0$. Pointwise, we decompose the vectors as $$v^{(i)} = \sum_{s \in S_i} \lambda^{(i)}_{s} e^{(i)}_s.$$
Acting on it with $\psi^{(i)}_{k[\mathbb S]}$, we obtain 
$$\psi^{(i)}_{k[\mathbb S]}(v) = \Big(\sum_{s\in S_i} \lambda^{(i)}_{s}\Big)\cdot E_c.$$
As $\psi^{(i)}_{k[\mathbb S]}$ was independent of $i$, we have 
$$\Big(\sum_{s\in S_i} \lambda^{(i)}_{s}\Big) = \Big(\sum_{s\in S_j} \lambda^{(j)}_{s}\Big), \quad \text{for all $i,j\in \tau$. } $$
We now define $U_1 = \{s\in S_1 \mid \lambda_s^{(i)} \neq 0 \} \subseteq S_1$ and for any subset $U\subseteq S_i$, we define the function $\omega(U)\coloneq \sum_{s\in U} \lambda^{(i)}_{s} \in k$ that projects onto the only coordinate of the colimit. Since $v$ is not in the kernel of $\psi^{(1)}_{k[\Sb]}$, $U_1$ is non empty and $\omega(U_1)\neq 0$. Moreover, for each $i\in\tau$, we define 
$$U_i\coloneq \{ u\in S_i \mid \exists (s_1,\dots,s_{i-1},u) \in \mylim \Sb[1,i], \text{ with $s_1\in U_1$}  \}, $$
i.e.\ the set of elements in $S_i$ directly connected to $U_1$ via a chain of compatible elements. We then have the following claim

\begin{claim}\label{clm:U}
In the above setting, we have
\begin{itemize}
    \item $\omega(U_i) = \omega(U_1),$
    \item $U_i\neq \emptyset$.
\end{itemize}
\end{claim}

\begin{proof}[Proof of Claim \ref{clm:U}]
We prove the statement by induction. If $i=1$, it is true by the definition of $U_1$. For the induction step, assume that the claim holds for $i<n$, and we want to show it for $i+1$. We distinguish between the directions of the edge in the zigzag type $\tau$ between the nodes $i$ and $i+1$:
\begin{itemize}[leftmargin=*]
    \item Case 1 $(i\to i+1)$: set $f_i = \Sb[i\to i+1] : S_i \to S_{i+1}$. Then 
    $$U_{i+1} = f_i(U_i) = \{ u' \in S_{i+1} \mid f_i^{-1}(u')\in U_i\} \neq \emptyset $$ 
    as $f_i$ is a set function, and $S_{i+1}$ is not empty. Moreover,
    $$k[f_i]\left( \sum_{u\in U_i} \lambda^{(i)}_{u} e^{(i)}_u   \right) = \sum_{u\in U_i}\lambda^{(i)}_{u} e^{(i)}_{f_i(u)}  $$
    and hence $\omega(U_{i+1}) = \omega(U_i) = \omega(U_1),$ where the last equality follows from the induction hypothesis.
    \item Case 2 $(i+1\to i)$: set $g_i = \Sb[i+1\to i] : S_{i+1} \to S_{i}$. Then 
    $$U_{i+1} = g_i^{-1}(U_i) = \{ u' \in S_{i+1} \mid g_i(u')\in U_i\}.$$ 
    Since $\omega(U_i)\neq 0$, there must exist a $u^{(i)} \in U_i$ with $\lambda^{(i)}_{u^{(i)}} \neq 0$. If $U_{i+1}$ is empty, there cannot exist a $u'\in S_{i+1}$ such that $g_{i}(u') = u^{(i)}$. This contradicts the fact that  $k[g_i](v^{(i+1)}) = v^{(i)}\neq 0$. Hence $U_{i+1}$ must be non-empty. By the same consistency relation, we also have that $\omega(U_{i+1}) = \omega(U_i)$.
\end{itemize}
\renewcommand\qedsymbol{$\square_\text{ Claim}$}
\qedhere
\end{proof}

\noindent Consequently, $U_n$ is non-empty, showing that $\mylim\Sb\neq \emptyset$.

\end{proof}

It is easy to check that $(k[\varprojlim \mathbb S] , (k[\pi_i])_{i\in\tau})$ forms a cone of the diagram $k[\mathbb S]$, and dually that $(k[\varinjlim \mathbb S] , (k[\iota_i])_{i\in\tau})$ forms a cocone. By the universal properties of limits and colimits, there exist morphisms
$$\eta : k[\varprojlim \mathbb S] \to \varprojlim k[\mathbb S],\quad \text{and} \quad \nu: \varinjlim k[\mathbb S] \to k[\varinjlim \mathbb S].$$
Moreover, since $k[-]$ is left-adjoint to the forgetful functor, it commutes with colimits, making \ $\nu$ is an isomorphism. For a fixed $i\in\tau$, we thus have the following commutative diagram

\begin{center}
\begin{tikzcd}
k[\varprojlim \mathbb S] \arrow[rr,"\eta"] \arrow[dr,"{k[\pi_i]}",swap]  && \varprojlim k [ \mathbb S] \arrow[dl,"\Pi_i" ]\\
& \arrow[dl, "{k[\iota_i]}",swap]k[S_i]\arrow[dr, "I_i"] & \\
k[\varinjlim \mathbb S] && \arrow[ll,"\nu" ,"\sim"'] \varinjlim k[\mathbb S]
\end{tikzcd}
\end{center}

We now have the following claim.
\begin{lemma}\label{clm:ranks}
In the above setting, we have
$$\mathrm{rank}\big( k[\iota_i]\circ k[\pi_i]: k[\varprojlim \mathbb S] \to k[\varinjlim \mathbb S]    \big) 
= \mathrm{rank} \big( I_i\circ\Pi_i: \varprojlim k[\mathbb S] \to \varinjlim k[\mathbb S]    \big) 
$$
\end{lemma}

\begin{proof}
``$\leq$": By the commutativity of the triangle
\begin{center}
\begin{tikzcd}
k[\varprojlim \mathbb S] \arrow[rr,"\eta"] \arrow[dr,"{k[\pi_i]}",swap]  && \varprojlim k [ \mathbb S] \arrow[dl,"\Pi_i" ]\\
& k[S_i]& 
\end{tikzcd}
\end{center}
we immediately have that $\im(k[\pi_i]) \subseteq \im( \Pi_i)$. Since $\nu$ is an isomorphism, we have 
$$\mathrm{rank}( k[\iota_i]\circ k[\pi_i]) \leq \mathrm{rank} ( I_i\circ\Pi_i )  . $$


\noindent ``$\geq$": We first observe that $\mathbb S$ decomposes into the disjoint union of subfunctors
$$\mathbb S = \bigsqcup_{C\in \mycolim S} \mathbb S_C, $$
where for each $s_C^i \in \Sb_C[i] $, we have $\iota_i(s_C^i)=C$. At the level of the corresponding merge graphs, these correspond to the connected components. Since limits and finite disjoint unions commute, we have
\begin{align*}
    \mathrm{rank} \left( \mylim k\left[\Sb \right]         \to  \mycolim k\left[\Sb\right]       \right ) &= \mathrm{rank} \left( \mylim k\left[\bigsqcup_{C\in\mycolim S} \Sb_C \right]         \to  \mycolim k\left[\bigsqcup_{C\in\mycolim S} \Sb_C \right]       \right )  \\
    &= \mathrm{rank} \left( \bigoplus_{C\in\mycolim S} \mylim k\left[ \Sb_C \right]         \to  \bigoplus_{C\in\mycolim S} \mycolim k\left[\Sb \right]       \right )  \\
    &= \sum_{C\in\mycolim S}   \mathrm{rank} \left(  \mylim k[\Sb_C] \to \mycolim k[\Sb_C]     \right)
\end{align*}
Now fix \(C\in\mycolim \mathbb S\). Suppose that the restriction of
\[
\psi_{k[\mathbb S_C]}:\varprojlim k[\mathbb S_C]\to \varinjlim k[\mathbb S_C]
\]
has nonzero image. Since \(\varinjlim k[\mathbb S_C]\) is one-dimensional, this means that there exists
\(v_C\in\varprojlim k[\mathbb S_C]\) such that
\[
\psi_{k[\mathbb S_C]}(v_C)=\lambda_C E_C
\]
for some \(\lambda_C\neq 0\). By Lemma~\ref{lem:sections}, the existence of such a \(v_C\) implies that \(\mathbb S_C\) admits a global section \(s_C\in\varprojlim \mathbb S_C\). Hence \(s_C\) defines a basis vector \(e_{s_C}\in k[\varprojlim \mathbb S]\), and
\[
k[\iota_i]\circ k[\pi_i](\lambda_C e_{s_C})
=
\lambda_C E_C
=
\psi_{k[\mathbb S_C]}(v_C).
\]
Thus the basis vector \(E_C\) lies in the image of \(k[\iota_i]\circ k[\pi_i]\) whenever it lies in the image of \(I_i\circ\Pi_i\). Therefore
\[
\operatorname{im}(I_i\circ\Pi_i)
\subseteq
\operatorname{im}(k[\iota_i]\circ k[\pi_i]),
\]
and hence
\[
\mathrm{rank}(k[\iota_i]\circ k[\pi_i])
\geq
\mathrm{rank}(I_i\circ\Pi_i)
.
\]
    
\end{proof}

We can now tackle the proof.

\begin{proof}[{Proof of Proposition~\ref{prop:fields-and-setmaps}}]
Let $\Sb:\mathcal P _\tau \to \cat{set}$ be a formigram, and let $k$ be a field. As $\psi_{k[\Sb]} = I_i\circ \Pi_i$ and $k[\psi_{\Sb}] = k[\iota_i\circ\pi_i]$, we have by the previous lemma that
$$\mathrm{rank}(\psi_{k[\Sb]}: \mylim k[\Sb] \to \mycolim k[\Sb]) = \mathrm{rank}(k[\psi_{\Sb}]: k[\mylim \Sb] \to k[\mycolim \Sb]).$$
Given $A,B\in\cat{set}$ and $f\in\Hom_{\cat{set}}(A,B)$, it is a general fact that 
$$ \mathrm{rank} ( k[f]: k[A]\to k[B]) = |\im(f)|,$$
from which the lemma follows immediately.
\end{proof}

\section{Proof of Lemma~\ref{lem:left-adjoint}} \label{app:proof-adjoint}
\begin{proof}[Proof of Lemma~\ref{lem:left-adjoint}]
First note that \(k[-]\) is functorial. If \(f:X\rightharpoonup Y\) and
\(g:Y\rightharpoonup Z\) are partial functions, then for each \(x\in X\),
\[
k[g]\circ k[f](e_x)
=
\begin{cases}
e_{g(f(x))}, & x\in\mathrm{Dom}(f)\ \text{and}\ f(x)\in\mathrm{Dom}(g),\\
0, & \text{otherwise},
\end{cases}
\]
which is precisely \(k[g\circ f](e_x)\). Moreover \(k[\mathrm{id}_X]=\mathrm{id}_{k[X]}\).
Hence \(k[-]\) defines a functor \(\cat{set_{par}}\to\cat{vect_k}\).

For the adjunction, let \(V\in\cat{vect_k}\). A linear map
\(\varphi:k[X]\to V\) is uniquely determined by the values
\(\varphi(e_x)\in V\). Equivalently, it determines the partial function
\[
\tilde{\varphi}:X\rightharpoonup U(V),
\qquad
\tilde{\varphi}(x)=\varphi(e_x)
\]
on the subset of \(X\) where \(\varphi(e_x)\neq 0\). Conversely, any partial
function \(h:X\rightharpoonup U(V)\) extends uniquely to a linear map
\(\hat h:k[X]\to V\) by
\[
\hat h(e_x)=
\begin{cases}
h(x), & x\in\mathrm{Dom}(h),\\
0, & x\notin\mathrm{Dom}(h).
\end{cases}
\]
These two constructions are inverse to one another and are natural in \(X\)
and \(V\). Thus \(k[-]\dashv U\).
\end{proof}

\newpage

\section{Algorithms}

\begin{algorithm}[H]\label{algo:CCL}
\caption{Connected Component Labelling (CCL)}
\KwIn{Binary image $I:\Omega_{w,h}\to\{0,1\}$, with foreground connectivity $\kappa_I\in\{4,8\}$}
\KwOut{Labelling $ L_{I}:\Omega_{w,h}\to \mathbb{N}_0$ and representatives $P_{I}\coloneq  (p_j)_{j=1}^{m_I}$}

Initialise $L_I(p)\gets 0$ if $I(p)=0$, and $L_I(p)\gets -1$ (unvisited) if $I(p)=1$\;
$\ell\gets 1$\; 

\For{each pixel $p\in\Omega_{w,h}$ in scan order}{
    \If{$L_I(p)=-1$}{
         set $p_\ell \gets p$\;  
        initialise a stack with $p$\;
        \While{the stack is non-empty}{
            pop $q$ from the stack\;
            \If{$L_I(q)=-1$}{
                $L_I(q)\gets \ell$\;
                push all $r\in N_{\kappa_I}(q)$ with $L_I(r)=-1$ onto the stack\;
            }
        }
        $\ell\gets \ell+1$\;
    }
}
\end{algorithm}

\begin{algorithm}[H]\label{algo:formi2}
\caption{Formigram construction (union case)}
\KwIn{Binary video $\mathcal{V}=(I_i)_{i=1}^n$, connectivity $\kappa\in\{4,8\}$}
\KwOut{Merge graph $\mathcal{G}(\pi_0(\mathbb{F}^\cup_{\mathcal{V}}))$}

\textbf{Step 1: Interpolation frames}\\
\For{$i=1,\dots,n-1$}{
    compute $J_i \gets I_i \vee I_{i+1}$\;
}

\medskip
\textbf{Step 2: Vertices}\\
\For{$i=1,\dots,n$}{
    compute $(L_{I_i},P_{I_i})$ via CCL\;
    set $m_{I_i}\gets |P_{I_i}|$\;
    add vertices $\{1,\dots,m_{I_i}\}$ at level $2i-1$\;
}
\For{$i=1,\dots,n-1$}{
    compute $(L_{J_i},P_{J_i})$ via CCL\;
    set $m_{J_i}\gets |P_{J_i}|$\;
    add vertices $\{1,\dots,m_{J_i}\}$ at level $2i$\;
}

\medskip
\textbf{Step 3: Edges}\\
\For{$i=1,\dots,n-1$}{
    let $P_{I_i}=(p_1,\dots,p_{m_{I_i}})$\;
    let $P_{I_{i+1}}=(q_1,\dots,q_{m_{I_{i+1}}})$\;

    \For{$j=1,\dots,m_{I_i}$}{
        add edge $(j,\,L_{J_i}(p_j))$ from level $2i-1$ to level $2i$\;
    }
    \For{$j=1,\dots,m_{I_{i+1}}$}{
        add edge $(j,\,L_{J_i}(q_j))$ from level $2i+1$ to level $2i$\;
    }
}
return the resulting merge graph\;
\end{algorithm}


\newpage

\medskip
\printbibliography[
heading=bibintoc,
title={References}
]

\end{document}